\documentclass[a4paper,12pt]{article}

\usepackage[T2A]{fontenc}
\usepackage[utf8]{inputenc}
\usepackage[russian]{babel}
\usepackage{amsmath, amssymb, amsthm, eqnarray}
\usepackage{enumerate}
\usepackage{indentfirst}
\usepackage[colorlinks=true,linkcolor=blue,citecolor=blue]{hyperref}
\usepackage{tikz-cd}
\usetikzlibrary{cd}
\usetikzlibrary{matrix, arrows}

\usepackage{tikz}
\usepackage{tikz-3dplot} 
\usetikzlibrary{fadings}
\tikzfading[name=fade out,
            inner color=transparent!0,
            outer color=transparent!100]
\usetikzlibrary{calc,arrows,decorations.pathreplacing,fadings,3d,positioning}

\title{Собственные симметрии трехмерных цепных дробей.
      }
     \date{}

\author{И.\,А.\,Тлюстангелов}

\theoremstyle{definition}
\newtheorem{definition}{Определение}

\newtheorem*{notation*}{Обозначение}
\theoremstyle{remark}

\newtheorem*{remark*}{Замечание}
\theoremstyle{plain}
\newtheorem{theorem}{Теорема}
\newtheorem*{theorem_dir*}{Теорема Дирихле}
\newtheorem{lemma}{Лемма}

\newtheorem{proposition}{Предложение}
\newtheorem{corollary}{Следствие}
\newtheorem*{statement*}{Утверждение}
\newtheorem*{corollary*}{Следствие}

\newtheorem{proof_m*}{Доказательство теоремы 1}
\newtheorem{question}{Вопрос}

\DeclareMathOperator{\conv}{conv}

\DeclareMathOperator{\id}{id}

\renewcommand{\phi}{\varphi}

\renewcommand{\vec}[1]{\mathbf{#1}}

\newcommand{\R}{\mathbb{R}}
\newcommand{\Z}{\mathbb{Z}}
\newcommand{\Q}{\mathbb{Q}}

\newcommand{\cK}{\mathcal{K}}

\newcommand{\gA}{\mathfrak{A}}
\newcommand{\gR}{\mathfrak{R}}
\newcommand{\gQ}{\mathfrak{Q}}

\newcommand{\Sl}{\textup{SL}}
\newcommand{\Gl}{\textup{GL}}

\newcommand{\trace}{\textup{Tr}}
\newcommand{\cf}{\textup{CF}}

\newcommand{\gal}{\textup{Gal}}

\begin{document}

\maketitle

\begin{abstract}
В данной работе доказывается критерий наличия у алгебраической цепной дроби собственной палиндромической симметрии в размерности $4$. В качестве многомерного обобщения цепных дробей рассматриваются полиэдры Клейна. 

\end{abstract}

\section{Введение}\label{intro}

Обыкновенная цепная дробь действительного числа имеет весьма изящную геометрическую интерпретацию, позволяющую перейти от классического случая к многомерному (см. \cite{klein} и, например, \cite{korkina_2dim}, \cite{german_sb}, \cite{german_bordeaux}, \cite{karpenkov_book}). Для описания такого обобщения рассмотрим $l_1,\ldots,l_n$ --- одномерные подпространства пространства $\R^n$, линейная оболочка которых совпадает со всем $\R^n$. Гиперпространства, натянутые на всевозможные $(n-1)$-наборы из этих подпространств, разбивают $\R^n$ на $2^n$ симплициальных конусов. Будем обозначать множество этих конусов через
\[ \mathcal{C}(l_1, \ldots, l_n).\]

Симплициальный конус с вершиной в начале координат $\vec{0}$ будем называть \emph{иррациональным}, если линейная оболочка любой его гиперграни не содержит целых точек, кроме начала координат  $\vec{0}$.

\begin{definition}
  Пусть  $C$ --- иррациональный конус, $C \in \mathcal{C}(l_1, \ldots, l_n)$. Выпуклая оболочка $\cK(C) = \conv(C\cap\Z^{n}\setminus\{\vec{0}\} )$ и его граница  $\partial(\cK(C))$ называются соответственно \emph{полиэдром Клейна} и \emph{парусом Клейна}, соответствующими конусу $C$. Объединение же всех $2^n$ парусов
  \[\cf(l_1, \ldots, l_n) = {\underset{C \, \in \, \mathcal{C}(l_1, \ldots, l_n)}{\bigcup}} \partial(\cK(C))\]
  называется \emph{$(n-1)$-мерной цепной дробью}.
\end{definition}

Особенный интерес представляет так называемый алгебраический случай. Напомним, что оператор из $\Gl_{n}(\Z)$ с вещественными собственными значениями, характеристический многочлен которого неприводим над $\Q$, называется \emph{гиперболическим}. Справедливо следующее утверждение о связи гиперболических операторов с алгебраическими числами в случае произвольного $n$ (подробности см., например, в \cite{german_tlyust_2})

\begin{proposition}\label{prop:more_than_pelle_n_dim}
  Числа $1,\alpha_1,\ldots,\alpha_{n-1}$ образуют базис некоторого вполне вещественного расширения $K$ поля $\Q$ тогда и только тогда, когда вектор $(1,\alpha_1,\ldots,\alpha_{n-1})$ является собственным для некоторого гиперболического оператора $A\in\Sl_n(\Z)$.
  При этом векторы $(1,\sigma_i(\alpha_1),\ldots,\sigma_i(\alpha_{n-1}))$, $i=1,\ldots,n$, где $\sigma_1(=\id),\sigma_2,\ldots,\sigma_n$ --- все вложения $K$ в $\R$, образуют собственный базис оператора $A$.
\end{proposition}

В случае $n=2$ предложение \ref{prop:more_than_pelle_n_dim} позволяет геометрически проинтерпретировать классическую теорему Лагранжа о периодичности обыкновенной цепной дроби. Геометрически теорема Лагранжа означает, что последовательность целочисленных длин и углов паруса одномерной цепной дроби $\cf(l_1, l_2)$ периодична тогда и только тогда, когда направления $l_1$ и $l_2$ являются собственными для некоторого $\Sl_{2}(\Z)$ оператора с различными вещественными собственными значениями (см., например, \cite{german_tlyust}).

\begin{definition}
  Пусть $l_1,\ldots,l_n$ --- собственные подпространства некоторого гиперболического оператора $A\in\Gl_n(\Z)$. Тогда $(n-1)$-мерная геометрическая цепная дробь $\cf(l_1,\ldots,l_n)$ называется \emph{алгебраической}. Мы будем также говорить, что эта дробь \emph{ассоциирована} с оператором $A$ и писать $\cf(A)=\cf(l_1,\ldots,l_n)$. Множество всех $(n-1)$-мерных алгебраических цепных дробей будем обозначать $\gA_{n-1}$.
\end{definition}

Будем называть \emph{группой симметрий} алгебраической цепной дроби $\cf(A)=\cf(l_1,\ldots,l_n)$ множество
\[
  \textup{Sym}_{\Z}\big(\cf(A)\big)=
  \Big\{ G\in\Gl_n(\Z) \ \Big|\ G\big(\cf(A)\big)=\cf(A) \Big\}.
\]
Из соображений непрерывности ясно, что для каждого $G\in\textup{Sym}_{\Z}\big(\cf(A)\big)$ однозначно определена перестановка $\sigma_G$, такая что
\begin{equation} \label{eq:repres}
  G(l_{i})=l_{\sigma_G(i)},\quad i=1,\dots,n.
\end{equation}
И обратно, если для $G\in\Gl_n(\Z)$ существует такая перестановка $\sigma_{G}$, что выполняются соотношения \eqref{eq:repres}, то $G\in\textup{Sym}_{\Z}\big(\cf(A)\big)$.

Благодаря теореме Дирихле об алгебраических единицах (см. \cite{bor_shaf}) существует изоморфная $\Z^{n-1}$ подгруппа группы $\textup{Sym}_{\Z}\big(\cf(A)\big)$ (см., например, \cite{german_tlyust_2}). Относительно действия этой подгруппы на любом из $2^n$ парусов возникает фундаментальная область, которую можно отождествить с $(n-1)$-мерным тором (см. \cite{korkina_2dim}). Для каждого элемента $G$, принадлежащего этой подгруппе, $\sigma_G=\id$. Однако, в $\textup{Sym}_{\Z}\big(\cf(A)\big)$, вообще говоря, могут существовать такие элементы $G$, для которых $\sigma_G\neq\id$.

\begin{definition}
  Оператор $G\in\textup{Sym}_{\Z}\big(\cf(A)\big)$ такой, что $\sigma_G=\id$, будем называть \emph{симметрией Дирихле} дроби $\cf(A)\in\gA_{n-1}$.
\end{definition}

\begin{definition}
  Оператор $G\in\textup{Sym}_{\Z}\big(\cf(A)\big)$, не являющийся симметрией Дирихле, будем называть \emph{палиндромической симметрией} дроби $\cf(A)$. Если множество палиндромических симметрией цепной дроби непусто, то такую цепную дробь будем называть \emph{палиндромичной}.
\end{definition}

\begin{definition}
  Симметрия $G\in\textup{Sym}_{\Z}\big(\cf(A)\big)$ называется \emph{циклической}, если $\sigma_G$ --- циклическая перестановка.
\end{definition}
 
 Очевидно, что все циклические симметрии цепной дроби $\cf(A)$ являются палиндромическими симметриями этой цепной дроби.

\begin{definition}
  Палиндромическая симметрия $G\in\textup{Sym}_{\Z}\big(\cf(A)\big)$ называется \emph{собственной}, если у оператора $G$ существует неподвижная точка на некотором парусе цепной дроби $\cf(A)$. Палиндромическая симметрия $G\in\textup{Sym}_{\Z}\big(\cf(A)\big)$, не являющаяся собственной, называется  \emph{несобственной}.
\end{definition}

Для $n=2$, то есть для одномерных цепных дробей, палиндромичность напрямую связана с симметричностью периодов обыкновенных цепных дробей квадратичных иррациональностей. Критерий симметричности периода цепной дроби квадратичной иррациональности восходит к результатам Галуа \cite{galois}, Лежандра \cite{legendre}, Перрона \cite{perron_book} и Крайтчика \cite{kraitchik}. В работе \cite{german_tlyust} дано геометрическое доказательство этого критерия. При этом приходится рассматривать как собственные, так и несобственные симметрии. Критерий палиндромичности цепной дроби для $n=3$ был получен в работе \cite{german_tlyust_2}. Стоит отметить, что при $n=3$ любая палиндромичная цепная дробь обладает собственной циклической симметрией. Для $n=4$ существует критерий наличия собственной палиндромической симметрии у алгебраической цепной дроби, схема доказательства которого описана в работе \cite{tlyust_proper}. Данная работа посвящена полному доказательству этого критерия. Основной результат работы \cite{tlyust} --- доказательство критерия существования собственной циклической палиндромической симметрии у трехмерной алгебраической цепной дроби. В данной работе с помощью полученных результатов мы по-новому доказываем этот критерий.

\section{Критерии палиндромичности}\label{results}

Здесь и далее обозначение $\vec v_1\sim\vec v_2$ для векторов из $\R^n$ означает существование такого оператора $X\in\Gl_n(\Z)$ и такого ненулевого $\mu\in\R$, что $X\vec v_1=\mu\vec v_2$. Упомянутый критерий палиндромичности для $n=2$ выглядит следующим образом:

\begin{theorem}\label{two_dimension}
 Пусть $\cf(l_1,l_2)\in\gA_1$ и пусть подпространство $l_1$ порождено вектором $(1, \alpha)$. Тогда $\cf(l_1, l_2)$ имеет собственную симметрию (или, что тоже самое, собственную циклическую симметрию) в том и только в том случае, если существует такое алгебраическое число $\omega$ степени $2$ со своим сопряжённым $\omega'$, что выполнено хотя бы одно из следующих условий:

  \textup{(а)} $(1, \alpha) \sim(1,\omega):\hskip 6.5mm \trace(\omega)=\omega + \omega' =0$;

  \textup{(б)} $(1, \alpha) \sim(1,\omega):\hskip 6.5mm \trace(\omega)=\omega + \omega' =1$.
  
\end{theorem}

В размерности $n=3$ критерий палиндромичности имеет следующую формулировку:

\begin{theorem}\label{three_dimension}
 Пусть $\cf(l_1,l_2,l_3)\in\gA_2$ и пусть подпространство $l_1$ порождено вектором $(1, \alpha, \beta)$. Тогда $\cf(l_1, l_2, l_3)$ имеет собственную симметрию (или, что тоже самое, собственную циклическую симметрию) в том и только в том случае, если существует такое алгебраическое число $\omega$ степени $3$ со своими сопряжёнными $\omega'$ и $\omega''$, что выполнено хотя бы одно из следующих условий:

  \textup{(а)} $(1, \alpha, \beta)\sim(1, \omega, \omega'):\hskip 6.5mm \trace(\omega)=\omega + \omega' + \omega'' =0$;

  \textup{(б)} $(1, \alpha, \beta)\sim(1, \omega, \omega'):\hskip 6.5mm \trace(\omega)=\omega + \omega' + \omega'' =1$.

\noindent При выполнении утверждения \textup{(а)} или \textup{(б)} кубическое расширение $\Q(\alpha, \beta)$ будет нормальным.
\end{theorem}

Критерий существования собственной палиндромической симметрии у алгебраической цепной дроби в размерности $n=4$ формулируется следующим образом:

 \begin{theorem}\label{theorem_proper_2_2}
  Пусть $\cf(l_1,l_2,l_3,l_4)\in\gA_3$ и пусть подпространство $l_1$ порождено вектором $(1, \alpha, \beta, \gamma)$. Пусть $K = \Q(\alpha, \beta, \gamma)$ и $\sigma_1(=\id), \sigma_2, \sigma_3, \sigma_4$ --- все вложения поля $K$ в $\R$ (см. предложение \ref{prop:more_than_pelle_n_dim}). Тогда $\cf(l_1, l_2, l_3, l_4)$ имеет собственную палиндромическую симметрию в том и только в том случае, если (с точностью до перестановки индексов) выполняется
\[ \sigma_3(K) = K, \quad \sigma_4(K) = \sigma_2(K), \quad \sigma_{3}^{2} = \sigma_1 = \id, \quad \sigma_4 = \sigma_2 \sigma_3\]
и существуют такие алгебраические числа $\omega$ и $\psi$ степени $4$, принадлежащие полю $K$, что выполнено хотя бы одно из следующих условий:
  
  \textup{(1)} $(1, \alpha, \beta, \gamma)\sim(1, \omega,  \psi, \omega'):\hskip 14.5mm \psi + \psi'= - (\omega + \omega')$;
  
  \textup{(2)} $(1, \alpha, \beta, \gamma)\sim(1, \omega,  \psi, \omega'):\hskip 14.5mm \psi + \psi'= 1 - (\omega + \omega')$;
  
  \textup{(3)} $(1, \alpha, \beta, \gamma)\sim(1, \omega,  \psi, \omega'):\hskip 14.5mm \psi + \psi'= 2 - (\omega + \omega')$;
  
  \textup{(4)} $(1, \alpha, \beta, \gamma)\sim(1, \omega,  \psi, \frac{\omega + \omega'}{2}):\hskip 10.2mm \psi + \psi'= - (\omega + \omega')$;
  
  \textup{(5)} $(1, \alpha, \beta, \gamma)\sim(1, \omega,  \psi, \frac{\omega + \omega'}{2}):\hskip 10.2mm \psi + \psi'= 2 - (\omega + \omega')$;
  
  \textup{(6)} $(1, \alpha, \beta, \gamma)\sim(1, \omega,  \psi,  \frac{\omega + \omega' + 1}{2}):\hskip 6.5mm \psi + \psi'= - (\omega + \omega')$;
  
  \textup{(7)} $(1, \alpha, \beta, \gamma)\sim(1, \omega,  \psi, \frac{\omega + \omega' + 1}{2}):\hskip 6.5mm \psi + \psi'= 2 - (\omega + \omega')$;
  
  \textup{(8)} $(1, \alpha, \beta, \gamma)\sim(1, \omega,  \psi, \frac{\omega + \omega'}{2}):\hskip 10.2mm \psi + \psi'= 1 - \frac{\omega + \omega'}{2}$;
  
  \textup{(9)} $(1, \alpha, \beta, \gamma)\sim(1, \omega,  \psi, \frac{\omega + \omega'}{2}):\hskip 10.2mm \psi + \psi'= 2 - \frac{\omega + \omega'}{2}$;
  
  \textup{(10)} $(1, \alpha, \beta, \gamma)\sim(1, \omega,  \psi, \frac{\omega' - \omega}{4}):\hskip 8.2mm \psi + \psi'= 2 - \frac{\omega + \omega'}{2}$,

\noindent где $\omega' = \sigma_3(\omega), \psi' = \sigma_3(\psi)$.
  
\end{theorem}

Наконец, в размерности $n=4$ критерий существования собственной циклической палиндромической симметрии у алгебраической цепной дроби имеет следующий вид:

\begin{theorem}\label{theorem_proper_4}
  Пусть $\cf(l_1,l_2,l_3,l_4)\in\gA_3$ и пусть подпространство $l_1$ порождено вектором $(1, \alpha, \beta, \gamma)$. Пусть $K = \Q(\alpha, \beta, \gamma)$. Тогда $\cf(l_1, l_2, l_3, l_4)$  имеет собственную циклическую симметрию в том и только в том случае, если $K$ --- циклическое расширение Галуа степени $4$, группа Галуа которого порождается вложением $\sigma$, и существует такое алгебраическое число $\omega \in K$ степени $4$, что выполнено хотя бы одно из семи условий $(1)$ - $(7)$ теоремы \ref{theorem_proper_2_2} для $\psi = \sigma_2(\omega)$.

\end{theorem}

Следующее утверждение, доказанное в работе \cite{german_tlyust_2}, показывает, что, в отличии от случаев $n=2$ и $n=3$, не всякая цепная дробь, обладающая собственными симметриями, обладает собственными циклическими симметриями:

\begin{proposition}\label{proper_but_not_cyclic}
Существуют такие вещественные числа $\alpha$, $\beta$, $\gamma$, что подпространство $l_1$ порождено вектором $(1, \alpha, \beta, \gamma)$, вполне вещественное расширение $K=\Q(\alpha,\beta, \gamma)$ поля $\Q$ не является нормальным и $\cf(l_1,l_2,l_3,l_4)\in\gA_{3}$ --- цепная дробь, обладающая собственными симметриями, но не обладающая собственными циклическими симметриями.
\end{proposition}

Любопытно, что теорема \ref{theorem_proper_4} не следует непосредственно из теоремы \ref{theorem_proper_2_2}. В связи с этим возникает естественный вопрос.

\begin{question}
 Верно ли, что существует такая палиндромичная алгебраическая цепная дробь $\cf(l_1,l_2,l_3,l_4)$, для которой поле $K$ является циклическим расширением Галуа и у которой не существует собственных циклических палиндромических симметрий?
\end{question}
 
Оставшаяся часть статьи имеет следующую структуру: в параграфе \ref{permutation_properity_n_4} мы анализируем то, как у собственных симметрий трехмерных цепных дробей устроены собственные подпространства и перестановки из соотношения \eqref{eq:repres}; в параграфе \ref{geom_proper_n_4} мы изучаем геометрию трехмерных цепных дробей, обладающих собственными симметриями, в том числе циклическими; в параграфе \ref{matrix_and_algebraicity} мы устанавливаем связь между определенными классами цепных дробей и матрицами их собственных симметрий; наконец, параграф \ref{proof_theorem_proper_2_2} посвящен доказательству теорем \ref{theorem_proper_2_2} и \ref{theorem_proper_4}.

\section{Собственные симметрии и собственные подпространства}\label{permutation_properity_n_4}
 Если задана дробь $\cf(l_1, \ldots, l_n)=\cf(A)\in\gA_{n-1}$, будем считать, что подпространство $l_1$ порождается вектором  $\vec l_1=(1,\alpha_1, \dots, \alpha_{n-1})$ (данное допущение корректно в силу предложения \ref{prop:more_than_pelle_n_dim}). Тогда из предложения \ref{prop:more_than_pelle_n_dim} следует, что числа $1,\alpha_1, \dots, \alpha_{n-1}$ образуют базис поля $K=\Q(\alpha_1, \dots, \alpha_{n-1})$ над $\Q$ и каждое $l_i$ порождается вектором $\vec l_i=(1,\sigma_i(\alpha_1), \dots, \sigma_i(\alpha_{n-1}))$, где $\sigma_1(=\id),\sigma_2, \dots, \sigma_n$ --- все вложения $K$ в $\R$. Заметим, что верна следующая 
\begin{lemma}\label{rational_eigen}
  Пусть $G\in\textup{Sym}_{\Z}\big(\cf(A)\big)$ и $\cf(A)=\cf(l_1, \ldots, l_n)$. Пусть $G \neq \pm I_{n}$ и $G(\vec{l}_1) = \lambda \vec{l}_1$. Тогда  $\lambda \notin \Q$.
  \end{lemma}
\begin{proof}
Предположим, что $\lambda \in \Q$. Поскольку $G \in \Gl_{n}(\Z)$ и $G \neq \pm I_{n}$, то $\textup{rank} (G - \lambda I_{n}) > 0$. Так как $(G - \lambda I_{n})(\vec{l}_1) = \vec{0}$, то какие-то числа из набора $1,\alpha_1,\ldots,\alpha_{n-1}$ выражаются через оставшиеся числа этого набора в виде некоторой линейной комбинации с коэффициентами из $\Q$. В силу предложения \ref{prop:more_than_pelle_n_dim} получаем противоречие.
\end{proof}

Отныне будем считать, что $n=4$, то есть будем рассматривать трехмерные цепные дроби. Напомним также, что для каждого $G\in\textup{Sym}_{\Z}\big(\cf(A)\big)$ соотношением \eqref{eq:repres} определена перестановка $\sigma_G$. 

\begin{lemma}\label{ord_3}
  Пусть $G$ --- палиндромическая симметрия $\cf(l_1,l_2,l_3, l_4)\in\gA_3$, ассоциированной с (гиперболическим) оператором $A$.
  Тогда существует такая нумерация подпространств $l_1, l_2, l_3, l_4$, что $\sigma_{G} = (1, 2)(3, 4)$ или $\sigma_{G} = (1, 2, 3, 4)$.
\end{lemma}

\begin{proof}
  Случай $\sigma_{G} = \textup{id}$ невозможен в силу того, что оператор $G$ не является симметрией Дирихле $\cf(A)$.

 Предположим, существует такая нумерация подпространств $l_1, l_2, l_3, l_4$, для которой $\sigma_{G} = (1)(2, 3, 4)$. Таким образом, существуют такие вещественные числа $\mu_{1}$, $\mu_{2}$, $\mu_{3}$, $\mu_{4}$, что матрица оператора $G$ в базисе $\vec{l}_{1}, \vec{l}_{2}, \vec{l}_{3}, \vec{l}_{4}$ имеет вид
 \[
   \begin{pmatrix}
     \mu_{1} & 0 & 0 & 0\\
     0 & 0 & 0 & \mu_{2}\\
      0 & \mu_{3} & 0 & 0\\
       0 & 0 &  \mu_{4} & 0
   \end{pmatrix}.
 \]
  Тогда характеристический многочлен оператора $G$ имеет вид
  \[\chi_{G}(x) = (x - \mu_{1})(x^3 - \mu_{2}\mu_{3}\mu_{4}) = x^4 - \mu_{1} x^3 - \mu_{2}\mu_{3}\mu_{4} x \pm 1 \in \Z[x].\]
  Следовательно, $\mu_{1}$ --- целое число, и при этом $\mu_{1}$ --- корень уравнения $\chi_{G}(x) = 0$, то есть $\mu_{1} = \pm 1$. Стало быть, $l_1$ --- собственное подпространство оператора $G$, соответствующее собственному значению $\mu_{1} = \pm 1$. То есть $l_1$ рационально, что противоречит гиперболичности оператора $A$.
  
 Теперь предположим, существует такая нумерация подпространств $l_1, l_2, l_3, l_4$, что $\sigma_{G} = (1) (2) (3, 4)$. Таким образом, существуют такие вещественные числа $\mu_{1}$, $\mu_{2}$, $\mu_{3}$, $\mu_{4}$, что матрица оператора $G$ в базисе $\vec{l}_{1}, \vec{l}_{2}, \vec{l}_{3}, \vec{l}_{4}$ имеет вид
 \[
   \begin{pmatrix}
     \mu_{1} & 0 & 0 & 0\\
     0 & \mu_{2} & 0 & 0\\
      0 & 0 & 0 & \mu_{3}\\
       0 & 0 &  \mu_{4} & 0
   \end{pmatrix}.
 \]
 Тогда характеристический многочлен оператора $G$, коэффициенты которого целочисленны, имеет вид
  \[(x - \mu_{1})(x - \mu_{2})(x^2 - \mu_{3}\mu_{4}) =\]
  \[ = x^4 - (\mu_{1} + \mu_{2})x^3 + (\mu_{1}\mu_{2} - \mu_{3}\mu_{4})x^2  + (\mu_{1} + \mu_{2}) \mu_{3}\mu_{4}x - \mu_{1}\mu_{2}\mu_{3}\mu_{4}.\]
 Так как $\mu_{1} + \mu_{2} \in \Z$, то $\mu_{3}\mu_{4} \in \Q$. Тогда существуют такие взаимно-простые целые числа $p \ge 1$ и $q \ge 1$, что $|\mu_{3}\mu_{4}| = \frac{p}{q}$, $|\mu_{1}\mu_{2}| = \frac{q}{p}$, а значит,
 \[ |\mu_{1}\mu_{2} - \mu_{3}\mu_{4}| = \frac{\pm p^2 \pm q^2}{pq}.\]
 
 Итак, $p^2$ делится на $q$ и $q^2$ делится на $p$, то есть $p=q=1$ и $\mu_{3}\mu_{4} = \pm 1$. Таким образом, матрица оператора $G^2$ в базисе $\vec{l}_{1}, \vec{l}_{2}, \vec{l}_{3}, \vec{l}_{4}$ имеет вид
 \[
   \begin{pmatrix}
     \mu_{1}^2 & 0 & 0 & 0\\
     0 & \mu_{2}^2 & 0 & 0\\
      0 & 0 & \pm 1 & 0\\
       0 & 0 &  0 & \pm 1
   \end{pmatrix}.
 \]
Из леммы \ref{rational_eigen} следует, что $G^2 = I_{4}$, то есть $\mu_{1} = \pm 1$. Вновь применяя лемму \ref{rational_eigen}, получаем, что $G = \pm I_{4}$, чего не может быть. 
 
Таким образом, существует такая нумерация подпространств $l_1, l_2, l_3, l_4$, что $\sigma_{G} = (1, 2)(3, 4)$ или $\sigma_{G} = (1, 2, 3, 4)$.
  \end{proof}
 
 \begin{corollary}\label{all_to_2_2}
Пусть $G$ --- палиндромическия симметрия $\cf(l_1,l_2,l_3, l_4)\in\gA_3$. Пусть $G' = G^{2}$, если $\textup{ord}(\sigma_{G}) = 4$, и $G' = G$, если $\textup{ord}(\sigma_{G}) = 2$. Тогда $G'$ --- палиндромическия симметрия $\cf(l_1,l_2,l_3, l_4)$ и $\textup{ord}(\sigma_{G'}) = 2$.
\end{corollary}

Пусть $G$ --- палиндромическия симметрия $\cf(l_1,l_2,l_3, l_4)\in\gA_3$.  Изменив при необходимости нумерацию подпространств $l_1, l_2, l_3, l_4$, в силу леммы \ref{ord_3} можно рассмотреть такие вещественные числа $\mu_{1}$, $\mu_{2}$, $\mu_{3}$, $\mu_{4}$, что матрица оператора $G$ в базисе $\vec{l}_{1}, \vec{l}_{2}, \vec{l}_{3}, \vec{l}_{4}$ имеет вид
  \begin{equation}\label{matrix_ord_4}
   \begin{pmatrix}
     0 & 0 & 0 &  \mu_{1}\\
     \mu_{2} & 0 & 0 & 0\\
      0 & \mu_{3} & 0 & 0\\
       0 & 0 &  \mu_{4} & 0
   \end{pmatrix}
 \end{equation}
или вид 
 \begin{equation}\label{matrix_ord_2_0}
   \begin{pmatrix}
     0 & 0 & \mu_{1} &  0\\
     0 & 0 & 0 &  \mu_{2}\\
     \mu_{3} & 0 & 0 & 0\\
     0 & \mu_{4} &  0 & 0
   \end{pmatrix}.
 \end{equation}

Пусть $G$ --- палиндромическая симметрия $\cf(l_1,l_2,l_3, l_4)\in\gA_3$ и матрица оператора $G$ в базисе $\vec{l}_{1}, \vec{l}_{2}, \vec{l}_{3}, \vec{l}_{4}$ имеет вид \eqref{matrix_ord_4}. В работе \cite{tlyust} доказывается следующая

\begin{lemma}\label{prod_lemm_cyclic}
 Пусть $G$ --- палиндромическая симметрия $\cf(l_1,l_2,l_3, l_4)\in\gA_3$ и матрица оператора $G$ в базисе $\vec{l}_{1}, \vec{l}_{2}, \vec{l}_{3}, \vec{l}_{4}$ имеет вид  \eqref{matrix_ord_4}. Тогда $G$ является собственной симметрией дроби $\cf(l_1,l_2, l_3, l_4)$ в том и только том случае, если  $\mu_{1}\mu_{2}\mu_{3}\mu_{4} = 1$.
 \end{lemma}
\begin{proof}
Пусть $G$ является собственной симметрией цепной дроби $\cf(l_1,l_2, l_3, l_4)$. Тогда существуют такие числа $\varepsilon_{1}$, $\varepsilon_{2}$, $\varepsilon_{3}$, $\varepsilon_{4}$ из множества $\{-1, 1\}$, что
\[G(\varepsilon_{1}\vec{l}_1, \varepsilon_{2}\vec{l}_2, \varepsilon_{3}\vec{l}_3, \varepsilon_{4}\vec{l}_4) =  \big(\mu_2\varepsilon_{1}\vec l_2,\mu_3\varepsilon_{2}\vec l_3,\mu_4\varepsilon_{3}\vec l_4,\mu_1\varepsilon_{4}\vec l_1\big),\]
и выполняются неравенства
\[\mu_1\frac{\varepsilon_{4}}{\varepsilon_{1}} > 0, \, \, \mu_2\frac{\varepsilon_{1}}{\varepsilon_{2}} > 0, \, \, \mu_3\frac{\varepsilon_{2}}{\varepsilon_{3}} > 0, \, \, \mu_4\frac{\varepsilon_{3}}{\varepsilon_{4}} > 0.\]
Стало быть, $\mu_{1}\mu_{2}\mu_{3}\mu_{4} > 0$, а значит, $\mu_{1}\mu_{2}\mu_{3}\mu_{4}  = 1$.

Если $\mu_{1}\mu_{2}\ldots\mu_{n} = 1$, то оператор $G$ имеет собственное направление, которое соответствует собственному значению $1$ и лежит внутри некоторого конуса $C \in \mathcal{C}(l_1, \ldots, l_n)$. 
\end{proof}

Для палиндромических симметрий вида \eqref{matrix_ord_2_0} справедливо аналогичное утверждение:

\begin{lemma}\label{prod_lemm}
  Пусть $G$ --- палиндромическая симметрия $\cf(l_1,l_2,l_3, l_4)\in\gA_3$ и матрица оператора $G$ в базисе $\vec{l}_{1}, \vec{l}_{2}, \vec{l}_{3}, \vec{l}_{4}$ имеет вид \eqref{matrix_ord_2_0}. Тогда $G$ является собственной симметрией дроби $\cf(l_1,l_2, l_3, l_4)$ в том и только том случае, если $\mu_{1}\mu_{3} = \mu_{2}\mu_{4} = 1$.
\end{lemma}
  
\begin{proof}
Пусть $G$ является собственной симметрией цепной дроби $\cf(l_1,l_2, l_3, l_4)$. Тогда существуют такие числа $\varepsilon_{1}$, $\varepsilon_{2}$, $\varepsilon_{3}$, $\varepsilon_{4}$ из множества $\{-1, 1\}$, что
\[G(\varepsilon_{1}\vec{l}_1, \varepsilon_{2}\vec{l}_2, \varepsilon_{3}\vec{l}_3, \varepsilon_{4}\vec{l}_4) =  \big(\mu_3\varepsilon_{1}\vec l_3,\mu_4\varepsilon_{2}\vec l_4,\mu_1\varepsilon_{3}\vec l_1,\mu_2\varepsilon_{4}\vec l_2\big),\]
и выполняются неравенства
\[\mu_1\frac{\varepsilon_{3}}{\varepsilon_{1}} > 0, \, \, \mu_2\frac{\varepsilon_{4}}{\varepsilon_{2}} > 0, \, \, \mu_3\frac{\varepsilon_{1}}{\varepsilon_{3}} > 0, \, \, \mu_4\frac{\varepsilon_{2}}{\varepsilon_{4}} > 0.\]
Стало быть, $\mu_{1}\mu_{3} > 0$ и $\mu_{2}\mu_{4} > 0$. Так как у оператора $G$ существует неподвижная точка на некотором парусе, то у оператора $G$ существует одномерное собственное подпространство, соответствующее собственному значению $1$. Теперь, поскольку характеристический многочлен оператора $G$ имеет вид $(x^2-\mu_{1}\mu_{3})(x^2-\mu_{2}\mu_{4})$, то $\mu_{1}\mu_{3} = 1$ или $\mu_{2}\mu_{4} = 1$. Тогда $\mu_{1}\mu_{3} = \mu_{2}\mu_{4} = 1$.

Если $\mu_{1}\mu_{3} = \mu_{2}\mu_{4} = 1$, то, опять же, характеристический многочлен оператора $G$ имеет вид $x^{4} - 2x^{2} + 1$. Стало быть, у оператора $G$ существует целочисленный собственный вектор, соответствующий собственному значению $1$. Этот вектор лежит внутри некоторого конуса $C \in \mathcal{C}(l_1, l_2, l_3, l_4)$, поскольку цепная дробь $\cf(l_1,l_2,l_3, l_4)$ является алгебраической. 
\end{proof} 

 \begin{corollary}\label{property_ord_eq}
Пусть $G$ --- палиндромическия симметрия $\cf(l_1,l_2,l_3, l_4)\in\gA_3$. Тогда $G$ является собственной симметрией в том и только том случае, если $G'$ (см. следствие \ref{all_to_2_2}) является собственной симметрией. 
\end{corollary}

\begin{proof}
Если $G$ является собственной симметрией $\cf(l_1,l_2,l_3, l_4)$, то, очевидно, оператор $G'$ также является собственной симметрией цепной дроби $\cf(l_1,l_2,l_3, l_4)$. 

Обратно, предположим, что $G'$ --- собственная симметрия $\cf(l_1,l_2,l_3, l_4)$. Случай $\textup{ord}(\sigma_{G}) = 2$, то есть $G' = G$, очевиден. Если $\textup{ord}(\sigma_{G}) = 4$, то, изменив при необходимости нумерацию подпространств $l_1, l_2, l_3, l_4$, можно считать, что матрица оператора $G$ в базисе $\vec{l}_{1}, \vec{l}_{2}, \vec{l}_{3}, \vec{l}_{4}$ имеет вид \eqref{matrix_ord_2_0}. Тогда матрица оператора $G' = G^{2}$ в базисе $\vec{l}_{1}, \vec{l}_{2}, \vec{l}_{3}, \vec{l}_{4}$ имеет вид

\[
   \begin{pmatrix}
     0 & 0 & \mu_{1}\mu_{4} &  0\\
     0 & 0 & 0 &  \mu_{2}\mu_{1}\\
     \mu_{3}\mu_{2} & 0 & 0 & 0\\
     0 & \mu_{4}\mu_{3} &  0 & 0
   \end{pmatrix}.
\]
 Стало быть, $\mu_{1}\mu_{2}\mu_{3}\mu_{4} = 1$ в силу леммы \ref{prod_lemm}, а значит, $G$ --- собственная симметрия цепной дроби $\cf(l_1,l_2, l_3, l_4)$ в силу леммы \ref{prod_lemm_cyclic}.
\end{proof}
   
  \begin{lemma}\label{rational_subspace_2_2}
  Пусть $G$ --- собственная симметрия $\cf(l_1,l_2,l_3, l_4)\in\gA_3$ и $\textup{ord}({\sigma_{G}}) = 2$. Тогда существуют такие одномерные рациональные подпространства $l^{1}_{+}$, $l^{2}_{+}$, $l^{1}_{-}$ и $l^{2}_{-}$, что $G l^{1}_{+} = l^{1}_{+}$, $G l^{2}_{+} = l^{2}_{+}$,  $G l^{1}_{-} = l^{1}_{-}$,  $G l^{2}_{-} = l^{2}_{-}$ и $l^{1}_{+} + l^{2}_{+} + l^{1}_{-} + l^{2}_{-}= \R^{4}$. При этом подпространства $l^{1}_{+}$ и $l^{2}_{+}$ соответствуют собственному значению $1$, а подпространства $l^{1}_{-}$ и $l^{2}_{-}$ соответствуют собственному значению $-1$.
  \end{lemma}
 
  \begin{proof}
   Изменив при необходимости нумерацию подпространств $l_1, l_2, l_3, l_4$, в силу леммы \ref{prod_lemm} можно считать, что существуют такие вещественные числа $\mu_{1}$ и $\mu_{2}$, что  матрица оператора $G$ в базисе $\vec{l}_{1}, \vec{l}_{2}, \vec{l}_{3}, \vec{l}_{4}$ имеет вид
\[
   \begin{pmatrix}
     0 & 0 & \mu_{1} &  0\\
     0 & 0 & 0 &  \mu_{2}\\
     \frac{1}{\mu_{1}} & 0 & 0 & 0\\
     0 & \frac{1}{\mu_{2}} &  0 & 0
   \end{pmatrix}.
 \]

Так как $\chi_{G}(x) = (x - 1)^2(x + 1)^2$, то у оператора $G$ есть двумерное инвариантное подпространство $L_{+}$, соответствующее собственному значению $1$ и двумерное инвариантное подпространство $L_{-}$, соответствующее собственному значению $-1$. Покажем рациональность подпространств $L_{+}$ и $L_{-}$, из чего будет следовать утверждение леммы.

Поскольку подпространство $L_{+}$ совпадает с решением системы линейных уравнений
\[(G-I_4)\vec{x}^\top = \vec{0},\]
то фундаментальная система решений данной системы линейных уравнений имеет размерность 2. Рассмотрев в качестве значений свободных переменных наборы $(0, 1)$ и $(1, 0)$, мы определим два линейно-независимых рациональных решения данной системы, из чего следует рациональность  $L_{+}$. Рациональность подпространства $L_{-}$ доказывается аналогичным способом.
\end{proof}

\begin{lemma}\label{rational_subspace_4}
  Пусть $G$ --- собственная циклическая симметрия $\cf(l_1,l_2,l_3, l_4)\in\gA_3$. Тогда собственные значения оператора $G$ равны $1$, $-1$, $i$ и $-i$. Более того, собственные подпространства $l_{+}$ (соответствующее собственному значению $1$), $l_{-}$ (соответствующее собственному значению $-1$) и $L$ (соответствующее собственным значениям $i$ и $-i$) являются рациональными. В частности, подпространство $L$ не содержит собственных для $G$ одномерных подпространств и для любого $\vec v \in L$ верно, что $G^2({\vec v}) = -\vec v$.
\end{lemma}
 
\begin{proof}
   Изменив при необходимости нумерацию подпространств $l_1, l_2, l_3, l_4$, можно считать, что в силу леммы \ref{prod_lemm_cyclic} существуют такие вещественные числа $\mu_{1}$, $\mu_{2}$, $\mu_{3}$, $\mu_{4}$, что $\mu_{1}\mu_{2}\mu_{3}\mu_{4} = 1$ и матрица оператора $G$ в базисе $\vec{l}_{1}, \vec{l}_{2}, \vec{l}_{3}, \vec{l}_{4}$ имеет вид \eqref{matrix_ord_4}. Так как $\chi_{G}(x) = x^4 - \mu_{1}\mu_{2}\mu_{3}\mu_{4}$, то собственные значения оператора $G$ равны $1$, $-1$, $i$ и $-i$, а значит, у $G$ есть ровно два одномерных собственных подпространства и двумерное инвариантное подпространство, которое не содержит собственных для $G$ одномерных подпространств. Обозначим через $l_{+}$ рациональное одномерное собственное подпространство оператора $G$, соответствующее собственному значению $1$, через $l_{-}$ --- рациональное одномерное собственное подпространство оператора $G$, соответствующее собственному значению $-1$, а через $L$ --- двумерное инвариантное подпространство, соответствующее собственным значениям $i$ и $-i$. Покажем рациональность подпространства $L$.

Поскольку $l_{-} + l_{+} + L = \R^{4}$, то для любого вектора $\vec v \in \R^{4}$ существуют такие единственные векторы $\vec{p}(\vec v, l_{-}) \in  l_{-}$, $\vec{p}(\vec v,  l_{+})  \in  l_{+}$ и $\vec{p}(\vec v, L) \in L$, что выполняется равенство
\[\vec v = \vec{p}(\vec v, l_{-}) + \vec{p}(\vec v,  l_{+}) + \vec{p}(\vec v, L).\]
Заметим, что $\vec{p}\big(G^{2}(\vec{v}),  L\big) = \vec{p}(-\vec v, L)$, $\vec{p}\big(G^{2}(\vec{v}), l_{-}\big) = \vec{p}(\vec v, l_{-})$ и $\vec{p}\big(G^{2}(\vec{v}), l_{+}\big) = \vec{p}(\vec v, l_{+})$ для любого вектора $\vec v \in \R^{4}$. Таким образом, для любой точки $\vec{z} \in \Z^{4} \setminus (l_{+} \, + \, l_{-})$ ненулевые целочисленные векторы $\vec{z}-G^{2}(\vec{z})$ и $G(\vec{z})-G^{3}(\vec{z})$ лежат в двумерном подпространстве $L$. Эти два целочисленных вектора неколлинеарны, поскольку $G(\vec{z})-G^{3}(\vec{z}) =  G\big(\vec{z}-G^{2}(\vec{z})\big)$ и подпространство $L$ не содержит собственных для оператора $G$ одномерных подпространств.  Итак, мы показали, что подпространство $L$ рационально. 
\end{proof}

\section{Геометрия собственных симметрий}\label{geom_proper_n_4}
  
 \begin{lemma}\label{main_lem_2_2}
   Пусть $G$ --- собственная симметрия дроби $\cf(l_1,l_2,l_3, l_4)\in\gA_3$. Пусть $F=G'$ (см. следствие \ref{all_to_2_2}) --- собственная симметрия $\cf(l_1,l_2,l_3, l_4)$ (см. следствие \ref{property_ord_eq}). Тогда существуют $\vec{z}_1$, $\vec{z}_2$, $\vec{z}_3$, $\vec{z}_4$ $\in$ $\Z^4$, такие что
\[F(\vec{z}_{1}) = \vec{z}_{3}, \, F(\vec{z}_{2}) = \vec{z}_{4}, \, F(\vec{z}_{3}) = \vec{z}_{1}, \, F(\vec{z}_{4}) = \vec{z}_{2}\] 
и выполняется хотя бы одно из следующих одиннадцати утверждений:

\textup{(1)} векторы $\vec{z}_{1}$, $\vec{z}_{2}$, $\vec{z}_{3}$, $\frac{1}{4}(\vec{z}_{1}+\vec{z}_{2}+\vec{z}_{3}+\vec{z}_{4})$ образуют базис решетки $\Z^4$;

\textup{(2)} векторы $\vec{z}_{1}$, $\vec{z}_{2}$,  $\vec{z}_{3}$, $\vec{z}_{4}$ образуют базис решетки $\Z^4$;

\textup{(3)} векторы $\vec{z}_{1}$, $\frac{1}{2}(\vec{z}_{1} + \vec{z}_{2})$, $\frac{1}{2}(\vec{z}_{1}+\vec{z}_{3})$, $\frac{1}{2}(\vec{z}_{1} + \vec{z}_{4})$ образуют базис решетки $\Z^4$;

\textup{(4)} векторы $\vec{z}_{1}$, $\vec{z}_{2}$, $\frac{1}{2}(\vec{z}_{1}+\vec{z}_{3})$, $\frac{1}{4}(\vec{z}_{1}+\vec{z}_{2}+\vec{z}_{3}+\vec{z}_{4})$ образуют базис решетки $\Z^4$;

\textup{(5)} векторы $\vec{z}_{1}$, $\vec{z}_{2}$, $\frac{1}{2}(\vec{z}_{1}+\vec{z}_{3})$, $\frac{1}{2}(\vec{z}_{2}+\vec{z}_{4})$ образуют базис решетки $\Z^4$;

\textup{(6)} векторы $\vec{z}_{1}$, $\vec{z}_{2}$, $\vec{z}_{3}$, $\frac{1}{2}(\vec{z}_{1} + \vec{z}_{3} + \vec{z}_{4} - \vec{z}_{2})$ образуют базис решетки $\Z^4$;

\textup{(7)} векторы $\vec{z}_{1}$, $\vec{z}_{2}$, $\vec{z}_{3}$, $\frac{1}{2}(\vec{z}_{1}+\vec{z}_{2}) + \frac{1}{4}(\vec{z}_{1}+\vec{z}_{4} - \vec{z}_{3} - \vec{z}_{2})$ образуют базис решетки $\Z^4$;

\textup{(8)} векторы $\vec{z}_{1}$, $\vec{z}_{2}$, $\vec{z}_{3}$, $\frac{1}{2}(\vec{z}_{2} + \vec{z}_{4})$ образуют базис решетки $\Z^4$;

\textup{(9)} векторы $\vec{z}_{1}$, $\vec{z}_{2}$, $\frac{1}{2}(\vec{z}_{1} + \vec{z}_{3})$, $\frac{1}{4}\vec{z}_{1} +  \frac{1}{2}\vec{z}_{2} - \frac{1}{4}\vec{z}_{3} + \frac{1}{2}\vec{z}_{4}$ образуют базис решетки $\Z^4$;

\textup{(10)} векторы $\vec{z}_{1}$, $\vec{z}_{2}$, $\frac{1}{2}(\vec{z}_{1}+\vec{z}_{3})$, $\frac{1}{2}\vec{z}_{1} +  \frac{1}{4}\vec{z}_{2} +  \frac{1}{2}\vec{z}_{4}$ образуют базис решетки $\Z^4$;

\textup{(11)} векторы $\vec{z}_{1}$, $\frac{1}{2}(\vec{z}_{1} + \vec{z}_{2})$, $\vec{z}_{3}$, $\frac{1}{2}(\vec{z}_{1} + \vec{z}_{3} + \vec{z}_{4} - \vec{z}_{2})$ образуют базис решетки $\Z^4$.
\end{lemma}
\begin{proof}
Будем называть плоскость \emph{рациональной}, если множество содержащихся в нем целых точек является (аффинной) решеткой ранга, равного размерности этой плоскости.

Рассмотрим для собственной симметрии $F$ подпространства $l^{1}_{+}$, $l^{2}_{+}$, $l^{1}_{-}$ и $l^{2}_{-}$ из леммы \ref{rational_subspace_2_2} и положим $S = l^{2}_{+} + l^{1}_{-} + l^{2}_{-}$. Обозначим через $S_1$ ближайшую к $S$ рациональную гиперплоскость, параллельную $S$ и не совпадающую с $S$ (любую из двух). Тогда $G(S_1) = S_1$. Также обозначим через $\vec{p}$ точку пересечения гиперплоскости $S_1$ и $l^{1}_{+}$, а через $l$ и $\pi$ прямую и плоскость, проходящие через точку $\vec{p}$ и параллельные $ l^{2}_{+}$ и $L_{-} = l^{1}_{-} + l^{2}_{-}$ соответственно. При этом $F(l) = l$, $F(\pi) = \pi$ и $F(\vec{p}) = \vec{p}$.

 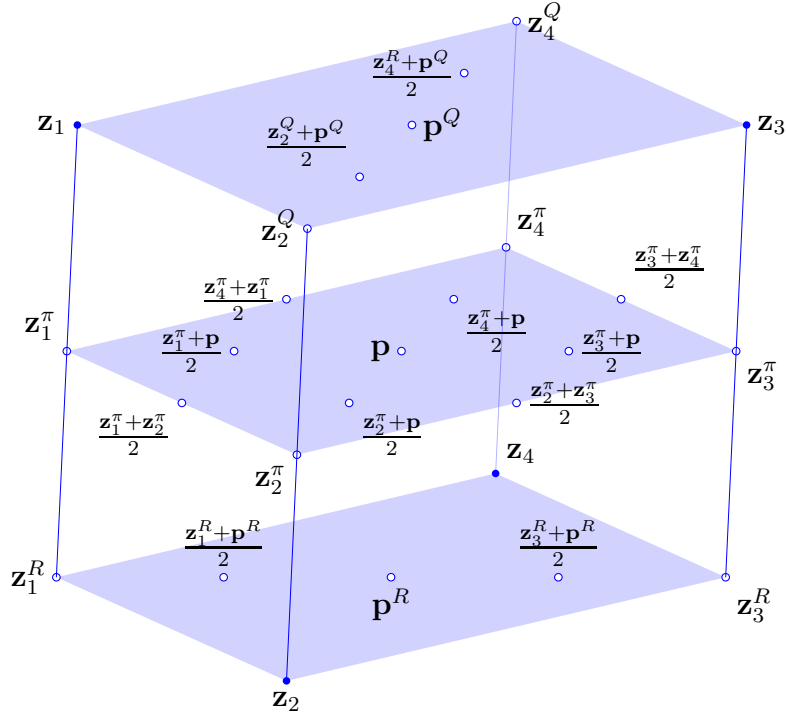
\begin{figure}[h]
  \centering
  \begin{tikzpicture}[x=10mm, y=7mm, z=-5mm, scale=1.1]
    \begin{scope}[rotate around z=0]

    \coordinate (z_1) at (-6,3,0);
    \coordinate (z_3) at (2,3,0);
    \coordinate (z_2) at (-1,-3,5);
    \coordinate (z_4) at (-1,-3,0);
    
    \coordinate (p_pl_1) at ($ (z_1)!1/2!(z_3) $);
    \coordinate (p_mi_1) at ($ (z_2)!1/2!(z_4) $);
    \coordinate (z_1_z_2_half) at ($ (z_1)!1/2!(z_2) $);
    \coordinate (z_2_z_3_half) at ($ (z_2)!1/2!(z_3) $);
    \coordinate (z_3_z_4_half) at ($ (z_3)!1/2!(z_4) $);
    \coordinate (z_4_z_1_half) at ($ (z_4)!1/2!(z_1) $); 
    
    \coordinate (p) at ($ (p_pl_1)!1/2!(p_mi_1) $);
    \coordinate (z_1_pl_1) at ($(p_mi_1) + (z_1) - (p_pl_1) $);
    \coordinate (z_3_pl_1) at ($(p_mi_1) + (z_3) - (p_pl_1) $);
    \coordinate (z_2_mi_1) at ($(p_pl_1) + (z_2) - (p_mi_1) $);
    \coordinate (z_4_mi_1) at ($(p_pl_1) + (z_4) - (p_mi_1) $);
    
    \coordinate (z_1_p) at ($ (z_1)!1/2!(z_1_pl_1) $);
    \coordinate (z_3_p) at ($(z_3)!1/2!(z_3_pl_1)$);
    \coordinate (z_2_p) at ($(z_2)!1/2!(z_2_mi_1)$);
    \coordinate (z_4_p) at ($(z_4)!1/2!(z_4_mi_1)$);
    
    \coordinate (p_1_p) at ($(z_1_pl_1)!1/2!(p_mi_1)$);
    \coordinate (p_3_p) at ($(z_3_pl_1)!1/2!(p_mi_1)$);
    
    \coordinate (p_2_q) at ($(z_2_mi_1)!1/2!(p_pl_1)$);
    \coordinate (p_4_q) at ($(z_4_mi_1)!1/2!(p_pl_1)$);
    
    \coordinate (z_1_p_1_2) at ($(p)!1/2!(z_1_p)$);
    \coordinate (z_3_p_1_2) at ($(p)!1/2!(z_3_p)$);
    \coordinate (z_2_p_1_2) at ($(p)!1/2!(z_2_p)$);
    \coordinate (z_4_p_1_2) at ($(p)!1/2!(z_4_p)$);

     \fill[blue!25,opacity=0.7] (z_1) -- (z_2_mi_1) -- (z_3) -- (z_4_mi_1) -- cycle;
     
      \fill[blue!25,opacity=0.7] (z_1_pl_1) -- (z_2) -- (z_3_pl_1) -- (z_4) -- cycle;
      
     \fill[blue!25,opacity=0.7] (z_1_p) -- (z_2_p) -- (z_3_p) -- (z_4_p) -- cycle;
    
     \node[fill=blue,circle,inner sep=1pt] at (z_1) {};
    \draw (z_1) node[left] {$\vec{z}_{1}$};
    
    \node[fill=blue,circle,inner sep=1pt] at (z_3) {};
    \draw (z_3) node[right] {$\vec{z}_{3}$};
    
    \node[fill=blue,circle,inner sep=1pt] at (z_2) {};
    \draw (z_2) node[below] {$\vec{z}_{2}$};
    
     \node[fill=blue,circle,inner sep=1pt] at (z_4) {};
    \draw (z_4) node[above right] {$\vec{z}_{4}$};

    \node[fill=white,circle,inner sep=1pt,draw=blue] at (z_1_pl_1) {};
    \draw (z_1_pl_1) node[left] {$\vec{z}_{1}^{R}$};
    
    \node[fill=white,circle,inner sep=1pt,draw=blue] at (z_3_pl_1) {};
    \draw (z_3_pl_1) node[below right] {$\vec{z}_{3}^{R}$};
    
    \node[fill=white,circle,inner sep=1pt,draw=blue] at (z_2_mi_1) {};
    \draw (z_2_mi_1) node[left] {$\vec{z}_{2}^{Q}$};
    
    \node[fill=white,circle,inner sep=1pt,draw=blue] at (z_4_mi_1) {};
    \draw (z_4_mi_1) node[right] {$\vec{z}_{4}^{Q}$};
    
     \node[fill=white,circle,inner sep=1pt,draw=blue] at (z_1_p) {};
    \draw (z_1_p) node[above left] {$\vec{z}_{1}^{\pi}$};
    
    \node[fill=white,circle,inner sep=1pt,draw=blue] at (z_3_p) {};
    \draw (z_3_p) node[below right] {$\vec{z}_{3}^{\pi}$};
    
     \node[fill=white,circle,inner sep=1pt,draw=blue] at (z_2_p) {};
    \draw (z_2_p) node[below left] {$\vec{z}_{2}^{\pi}$};
    
     \node[fill=white,circle,inner sep=1pt,draw=blue] at (z_4_p) {};
    \draw (z_4_p) node[above right] {$\vec{z}_{4}^{\pi}$};
    
     \node[fill=white,circle,inner sep=1pt,draw=blue] at (p) {};
    \draw (p) node[left] {$\vec p$};
    
    \node[fill=white,circle,inner sep=1pt,draw=blue] at (p_pl_1) {};
    \draw (p_pl_1) node[right] {$\vec{p}^{Q}$};
    
    \node[fill=white,circle,inner sep=1pt,draw=blue] at (p_mi_1) {};
    \draw (p_mi_1) node[below] {$\vec{p}^{R}$};
    
     \node[fill=white,circle,inner sep=1pt,draw=blue] at (z_1_z_2_half) {};
    \draw (z_1_z_2_half) node[below left] {$\frac{\vec{z}_{1}^{\pi} + \vec{z}_{2}^{\pi}}{2}$};
    
    \node[fill=white,circle,inner sep=1pt,draw=blue] at (z_2_z_3_half) {};
    \draw (z_2_z_3_half) node[right] {$\frac{\vec{z}_{2}^{\pi} + \vec{z}_{3}^{\pi}}{2}$};
    
         \node[fill=white,circle,inner sep=1pt,draw=blue] at (z_3_z_4_half) {};
    \draw (z_3_z_4_half) node[above right] {$\frac{\vec{z}_{3}^{\pi} + \vec{z}_{4}^{\pi}}{2}$};
    
     \node[fill=white,circle,inner sep=1pt,draw=blue] at (z_4_z_1_half) {};
    \draw (z_4_z_1_half) node[left] {$\frac{\vec{z}_{4}^{\pi} + \vec{z}_{1}^{\pi}}{2}$};
    
    \node[fill=white,circle,inner sep=1pt,draw=blue] at (p_1_p) {};
    \draw (p_1_p) node[above] {$\frac{\vec{z}_{1}^{R} + \vec{p}^{R}}{2}$};
    
    \node[fill=white,circle,inner sep=1pt,draw=blue] at (p_3_p) {};
    \draw (p_3_p) node[above] {$\frac{\vec{z}_{3}^{R} + \vec{p}^{R}}{2}$};
    
      \node[fill=white,circle,inner sep=1pt,draw=blue] at (p_2_q) {};
    \draw (p_2_q) node[above left] {$\frac{\vec{z}_{2}^{Q} + \vec{p}^{Q}}{2}$};
    
    \node[fill=white,circle,inner sep=1pt,draw=blue] at (p_4_q) {};
    \draw (p_4_q) node[left] {$\frac{\vec{z}_{4}^{R} + \vec{p}^{Q}}{2}$};
    
     \node[fill=white,circle,inner sep=1pt,draw=blue] at (z_1_p_1_2) {};
    \draw (z_1_p_1_2) node[left] {$\frac{\vec{z}^{\pi}_{1} +  \vec{p}}{2}$};
    
    \node[fill=white,circle,inner sep=1pt,draw=blue] at (z_3_p_1_2) {};
    \draw (z_3_p_1_2) node[right] {$\frac{\vec{z}^{\pi}_{3} +  \vec{p}}{2}$};
    
    \node[fill=white,circle,inner sep=1pt,draw=blue] at (z_2_p_1_2) {};
    \draw (z_2_p_1_2) node[below right] {$\frac{\vec{z}^{\pi}_{2} +  \vec{p}}{2}$};
    
    \node[fill=white,circle,inner sep=1pt,draw=blue] at (z_4_p_1_2) {};
    \draw (z_4_p_1_2) node[below right] {$\frac{\vec{z}^{\pi}_{4} +  \vec{p}}{2}$};
    
   \draw[-] [blue,very thin] (z_1_pl_1) -- (z_1);
   \draw[-] [blue,very thin] (z_3_pl_1) -- (z_3);
   \draw[-] [blue,very thin, opacity=2] (z_2_mi_1) -- (z_2);
   \draw[-] [blue,very thin, opacity=0.3] (z_4_mi_1) -- (z_4);
   
    \end{scope}
  \end{tikzpicture}
  \caption{Возможное расположение точек решетки $\Z^4$ в параллелограммах из построенной тройки $(\Delta_{k}^{\pi}, \Delta_{k}^{Q}, \Delta_{k}^{R})$}
  \label{integer_points}
\end{figure}

Плоскость $\pi$ разделяет гиперплоскость $S_1$ на два множества $S^{+}_1$ и $S^{-}_1$. Пусть $Q$ и $R$ --- рациональные плоскости ближайшие к $\pi$, параллельные $\pi$ и не совпадающие с $\pi$, принадлежащие множествам $S^{+}_1$ и $S^{-}_1$ соответственно. Отметим, что, вообще говоря, расстояния от $\pi$ до $Q$ и от $\pi$ до $R$ не обязательно равны. Положим $\vec{p}^{Q} = Q \cap l$ и $\vec{p}^{R} = R \cap l$. Построим точки $\vec{z}_{1}$, $\vec{z}_{2}$, $\vec{z}_{3}$, $\vec{z}_{4}$ при помощи следующей итерационной процедуры. 

Для начала предположим, $(\vec{v}_{1} , \vec{v}_{2})$ --- такая пара точек решетки $\Z^4$, что $\vec{v}_{1} \in Q, \vec{v}_{2} \in R$, векторы $\vec{v}_{1} - \vec{p}^{Q}$ и $\vec{v}_{2} - \vec{p}^{R}$ неколлинеарны. Тогда можно построить точки 
\begin{equation}\label{points}
\begin{gathered}
\vec{v}_{3} = F(\vec{v}_{1}) \in \Z^4 \cap Q, \quad  \vec{v}_{4} = F(\vec{v}_{2}) \in \Z^4  \cap R,\\
\vec{v}^{R}_{1} = \vec{v}_{1} + (\vec{p}^{R} - \vec{p}^{Q}), \quad \vec{v}^{R}_{3} = \vec{v}_{3} + (\vec{p}^{R} - \vec{p}^{Q}),\\
\vec{v}^{Q}_{2} = \vec{v}_{2} + (\vec{p}^{Q} - \vec{p}^{R}), \quad \vec{v}^{Q}_{4} = \vec{v}_{4} + (\vec{p}^{Q} - \vec{p}^{R}),\\
\vec{v}^{\pi}_{1} = \vec{v}_{1} + (\vec{p} - \vec{p}^{Q}), \quad \vec{v}^{\pi}_{3} = \vec{v}_{3} + (\vec{p} - \vec{p}^{Q}),\\
\vec{v}^{\pi}_{2} = \vec{v}_{2} + (\vec{p} - \vec{p}^{R}), \quad \vec{v}^{\pi}_{4} = \vec{v}_{4} + (\vec{p} - \vec{p}^{R}).
\end{gathered}
\end{equation}

Рассмотренные точки определяют тройку параллелограммов $(\Delta^{\pi}, \Delta^{Q}, \Delta^{R})$, где
\begin{equation}\label{parallelograms}
\begin{gathered}
\Delta^{\pi} = \textup{conv}(\vec{v}^{\pi}_{1}, \vec{v}^{\pi}_{2}, \vec{v}^{\pi}_{3}, \vec{v}^{\pi}_{4}) \subset \pi,\\
\Delta^{Q} = \textup{conv}(\vec{v}_{1}, \vec{v}^{Q}_{2}, \vec{v}_{3}, \vec{v}^{Q}_{4}) \subset Q, \quad \Delta^{R} = \textup{conv}(\vec{v}^{R}_{1}, \vec{v}_{2}, \vec{v}^{R}_{3}, \vec{v}_{4}) \subset R.
\end{gathered}
\end{equation}

Возьмем произвольную целочисленную точку $\vec{v}_{1,1} \in Q \setminus l$. Тогда существует такая целочисленная точка $\vec{v}_{1,2} \in R \setminus l$, что вектор $\vec{v}_{1,1} - \vec{p}^{Q}$ неколлинеарен вектору $\vec{v}_{1,2} - \vec{p}^{R}$.

Теперь предположим, мы построили такую пару точек $(\vec{v}_{j,1}, \vec{v}_{j,2})$ решетки $\Z^4$, что $\vec{v}_{j,1} \in Q,  \vec{v}_{j,2} \in R$, векторы $\vec{v}_{j,1} - \vec{p}^{Q}$ и $ \vec{v}_{j,2} - \vec{p}^{R}$ неколлинеарны. Положив $(\vec{v}_{1} , \vec{v}_{2}) = (\vec{v}_{j,1} , \vec{v}_{j,2})$, определим с помощью \ref{points} точки $\vec{v}_{j,3}=\vec{v}_{3}$, $\vec{v}_{j,4}=\vec{v}_{4}$, $\vec{v}^{R}_{j,1} = \vec{v}^{R}_{1}$, $\vec{v}^{R}_{j,3} = \vec{v}^{R}_{3}$, $\vec{v}^{Q}_{j,2} = \vec{v}^{Q}_{2}$, $\vec{v}^{Q}_{j,4} = \vec{v}^{Q}_{4}$, $\vec{v}^{\pi}_{j,1} = \vec{v}^{\pi}_{1}$, $\vec{v}^{\pi}_{j,2} = \vec{v}^{\pi}_{2}$, $\vec{v}^{\pi}_{j,3} = \vec{v}^{\pi}_{3}$, $\vec{v}^{\pi}_{j,4} = \vec{v}^{\pi}_{4}$. Также, с помощью \ref{parallelograms} определим параллелограммы $\Delta^{\pi}_{j} = \Delta^{\pi}, \Delta^{Q}_{j} = \Delta^{Q}, \Delta^{R}_{j} = \Delta^{R}$. Кроме того, положим $\vec{p}^{R}_{j,1} =  \frac{1}{2}(\vec{v}^{R}_{j,1} + \vec{p}^{R})$, $\vec{p}^{R}_{j,3} =  \frac{1}{2}(\vec{v}^{R}_{j,3} + \vec{p}^{R})$, $\vec{p}^{Q}_{j,2} =  \frac{1}{2}(\vec{v}^{Q}_{j,2} + \vec{p}^{Q})$ и $\vec{p}^{Q}_{j,4} =  \frac{1}{2}(\vec{v}^{R}_{j,4} + \vec{p}^{Q})$.

 Если на плоскостях $Q$ и $R$ существует целая точка, не совпадающая с точками $\vec{p}^{Q}$, $\vec{p}^{R}$, $\vec{p}^{R}_{j,1}$, $\vec{p}^{R}_{j,3}$, $\vec{p}^{Q}_{j,2}$, $\vec{p}^{Q}_{j,4}$ и не с какой из вершин параллелограммов $\Delta^{Q}_{j}$ и $\Delta^{R}_{j}$, при этом лежащая в одном из этих параллелограммов (без ограничения общности, в $\Delta^{Q}_{j}$), то обозначим ее через $\vec{v}$. Иначе будем называть пару точек $(\vec{v}_{j,1}, \vec{v}_{j,2})$ \emph{допустимой парой для оператора $G$ и цепной дроби $\cf(l_1,l_2,l_3, l_4)\in\gA_3$}.

Предположим, что вектор $\vec{v} - \vec{p}^{Q}$ неколлинеарен вектору $\vec{v}_{j,2} - \vec{p}^{R}$. Тогда положим $\vec{v}_{j+1,1} = \vec{v}$, $\vec{v}_{j+1,2} = \vec{v}_{j,2}$. 

Теперь предположим, что вектор $\vec{v} - \vec{p}^{Q}$ коллинеарен вектору $\vec{v}_{j,2} - \vec{p}^{R}$. Заметим, что $\vec{v} - F(\vec{v}) = 2(\vec{v} - \vec{p}^{Q})$, а значит, $|\vec{v} - F(\vec{v})| < |\vec{v}_{j,4} - \vec{v}_{j,2}|$ и векторы $\pm\big(\vec{v} - F(\vec{v})\big)$ не совпадают ни с каким из векторов $\vec{v}_{j,4} - \vec{v}_{j,2}$ и $\vec{p}^{R} - \vec{v}_{j,2}  = \vec{p}^{Q}_{j,4} - \vec{p}^{Q}_{j,2}$. Таким образом, либо точка $\vec{v}_{j,2} + \big(\vec{v} - F(\vec{v}))$, либо точка $\vec{v}_{j,2} - (\vec{v} - F(\vec{v}))$ лежит в параллелограмме $\Delta^{R}_{j}$ и не совпадает с точками $\vec{p}^{R}$, $\vec{p}^{R}_{j,1}$, $\vec{p}^{R}_{j,3}$ и не с какой из вершин параллелограмма $\Delta^{R}_{j}$. Обозначим эту точку через $\vec{v}_{j+1,2}$ и положим $\vec{v}_{j+1,1} = \vec{v}_{j,1}$. Заметим, что вектор $\vec{v}_{j+1,1}  - \vec{p}^{Q}$ неколлинеарен вектору $\vec{v}_{j+1,2} - \vec{p}^{R}$.

Последовательность пар $(\vec{v}_{j,1}, \vec{v}_{j,2})$ конечна, так как, по построению  
\[\Delta^{\pi}_{j} \subset \Delta^{\pi}_{j+1}, \quad \Delta^{Q}_{j} \subset \Delta^{Q}_{j+1}, \quad \Delta^{R}_{j} \subset \Delta^{R}_{j+1}.\]
Пусть $(\vec{z}_{1}, \vec{z}_{2}) = (\vec{v}_{k,1}, \vec{v}_{k,2})$ --- последний элемент такой последовательности, то есть $(\vec{z}_{1}, \vec{z}_{2})$ --- допустимая пара для оператора $G$ и цепной дроби $\cf(l_1,l_2,l_3, l_4)\in\gA_3$. Положив $(\vec{v}_{1} , \vec{v}_{2}) = (\vec{z}_{1} , \vec{z}_{2})$, определим с помощью \ref{points} точки $\vec{z}_{3}=\vec{v}_{3}$, $\vec{z}_{4}=\vec{v}_{4}$, $\vec{z}^{R}_{1} = \vec{v}^{R}_{1}$, $\vec{z}^{R}_{3} = \vec{v}^{R}_{3}$, $\vec{z}^{Q}_{2} = \vec{v}^{Q}_{2}$, $\vec{z}^{Q}_{4} = \vec{v}^{Q}_{4}$, $\vec{z}^{\pi}_{1} = \vec{v}^{\pi}_{1}$, $\vec{z}^{\pi}_{2} = \vec{v}^{\pi}_{2}$, $\vec{z}^{\pi}_{3} = \vec{v}^{\pi}_{3}$, $\vec{z}^{\pi}_{4} = \vec{v}^{\pi}_{4}$. Также, с помощью \ref{parallelograms} определим параллелограммы $\Delta^{\pi}_{k} = \Delta^{\pi}, \Delta^{Q}_{k} = \Delta^{Q}, \Delta^{R}_{k} = \Delta^{R}$. Кроме того, положим $\vec{p}^{R}_{1} =  \frac{1}{2}(\vec{z}^{R}_{1} + \vec{p}^{R})$, $\vec{p}^{R}_{3} =  \frac{1}{2}(\vec{z}^{R}_{3} + \vec{p}^{R})$, $\vec{p}^{Q}_{2} =  \frac{1}{2}(\vec{z}^{Q}_{2} + \vec{p}^{Q})$ и $\vec{p}^{Q}_{4} =  \frac{1}{2}(\vec{z}^{R}_{4} + \vec{p}^{Q})$.

Покажем, что множество $(\Delta^{\pi}_{k} \cup \Delta^{Q}_{k} \cup \Delta^{R}_{k}) \, \cap \, \Z^{4}$ совпадает с одним из множеств (с точностью до перенумерации точек $\vec{z}_1, \vec{z}_2, \vec{z}_3, \vec{z}_4$)

\[\{\vec{z}_{1}, \, \vec{z}^{Q}_{2}, \, \vec{z}_{3}, \, \vec{z}^{Q}_{4}, \, \vec{z}^{R}_{1}, \, \vec{z}_{2}, \, \vec{z}^{R}_{3}, \, \vec{z}_{4}, \, \vec{p}\},\]
\[\{\vec{z}_{1}, \, \vec{z}_{3}, \, \vec{z}_{2}, \, \vec{z}_{4}\},\]
\[\{\vec{z}_{1}, \, \vec{z}_{3}, \, \vec{p}^{Q}, \, \vec{z}_{2},  \, \vec{z}_{4}, \,  \vec{p}^{R}_{1}, \, \vec{p}^{R}_{3}\},\]
\[\{\vec{z}_{1}, \, \vec{z}^{Q}_{2}, \, \vec{z}_{3}, \, \vec{z}^{Q}_{4}, \, \vec{z}^{R}_{1}, \, \vec{z}_{2}, \, \vec{z}^{R}_{3}, \, \vec{z}_{4}, \, \vec{p}^{Q}, \, \vec{p}^{R}, \, \frac{\vec{z}^{\pi}_{1} + \vec{z}^{\pi}_{2}}{2}, \, \frac{\vec{z}^{\pi}_{2} + \vec{z}^{\pi}_{3}}{2}, \, \frac{\vec{z}^{\pi}_{3} + \vec{z}^{\pi}_{4}}{2}, \, \frac{\vec{z}^{\pi}_{4} + \vec{z}^{\pi}_{1}}{2}\},\]
\[\{\vec{z}_{1}, \, \vec{z}^{Q}_{2}, \, \vec{z}_{3}, \, \vec{z}^{Q}_{4}, \, \vec{z}^{R}_{1}, \, \vec{z}_{2}, \, \vec{z}^{R}_{3}, \, \vec{z}_{4}, \, \vec{p}^{Q}, \, \vec{p}^{R}, \, \frac{\vec{z}^{\pi}_{1} +  \vec{p}}{2}, \, \frac{\vec{z}^{\pi}_{3} +  \vec{p}}{2} \},\]
\[\{\vec{z}_{1}, \, \vec{z}^{Q}_{2}, \, \vec{z}_{3}, \, \vec{z}^{Q}_{4}, \, \vec{z}^{R}_{1}, \, \vec{z}_{2}, \, \vec{z}^{R}_{3}, \, \vec{z}_{4}, \, \frac{\vec{z}^{\pi}_{1} + \vec{z}^{\pi}_{2}}{2}, \, \frac{\vec{z}^{\pi}_{3} + \vec{z}^{\pi}_{4}}{2} \},\]
\[\{\vec{z}_{1}, \, \vec{z}^{Q}_{2}, \, \vec{z}_{3}, \, \vec{z}^{Q}_{4}, \, \vec{z}^{R}_{1}, \, \vec{z}_{2}, \, \vec{z}^{R}_{3}, \, \vec{z}_{4}, \, \vec{z}^{\pi}_{1}, \, \vec{z}^{\pi}_{2}, \, \vec{z}^{\pi}_{3}, \, \vec{z}^{\pi}_{4}\},\]
\[\{\vec{z}_{1}, \, \vec{z}_{3}, \, \vec{z}_{2}, \, \vec{z}_{4}, \vec{p}^{R}\},\]
\[\{\vec{z}_{1}, \, \vec{z}^{Q}_{2}, \, \vec{z}_{3}, \, \vec{z}^{Q}_{4}, \, \vec{z}^{R}_{1}, \, \vec{z}_{2}, \, \vec{z}^{R}_{3}, \, \vec{z}_{4}\},\]
\[\{\vec{p}, \, \vec{z}_{1}, \, \vec{z}^{Q}_{2}, \, \vec{z}_{3}, \, \vec{z}^{Q}_{4}, \, \vec{z}^{R}_{1}, \, \vec{z}_{2}, \, \vec{z}^{R}_{3}, \, \vec{z}_{4}, \, \vec{p}^{Q}, \, \vec{p}^{R}, \, \vec{z}^{\pi}_{1}, \, \vec{z}^{\pi}_{2}, \, \vec{z}^{\pi}_{3}, \, \vec{z}^{\pi}_{4} \},\]
\[\{\vec{z}_{1}, \, \vec{z}^{Q}_{2}, \, \vec{z}_{3}, \, \vec{z}^{Q}_{4}, \, \vec{z}^{R}_{1}, \, \vec{z}_{2}, \, \vec{z}^{R}_{3}, \, \vec{z}_{4}, \, \vec{p}^{Q}, \, \vec{p}^{R}\}.\]

Рассмотрим параллелограмм $\Delta^{Q}_{k}$. Поскольку $F(\vec{p}^{Q}_{2}) =  \vec{p}^{Q}_{4}$, $F(\vec{p}^{Q}_{4}) =  \vec{p}^{Q}_{2}$, $F(\vec{z}^{Q}_{2}) =  \vec{z}^{Q}_{4}$ и $F(\vec{z}^{Q}_{4}) =  \vec{z}^{Q}_{2}$, то $\vec{p}^{Q}_{2} \in \Z^{4} \Leftrightarrow \vec{p}^{Q}_{4} \in \Z^{4}$ и $\vec{z}^{Q}_{2} \in \Z^{4} \Leftrightarrow \vec{z}^{Q}_{4} \in \Z^{4}$.  Заметим, что если $\vec{p}^{Q}_{2} \in \Z^{4}$ и $\vec{p}^{Q}_{4} \in \Z^{4}$, то точки $\vec{z}^{Q}_{2}, \vec{z}^{Q}_{4}, \vec{p}^{Q}$ не принадлежат решетке $\Z^{4}$ в силу способа построения параллелограмма $\Delta^{R}_{k}$. Таким образом, множество $\Delta^{Q}_{k} \cap \Z^{4}$ совпадает с одним из множеств 
\[ \{\vec{z}_{1}, \vec{z}_{3}\}, \{\vec{z}_{1}, \vec{z}_{3}, \vec{z}^{Q}_{2}, \vec{z}^{Q}_{4}\}, \{\vec{z}_{1}, \vec{z}_{3}, \vec{p}^{Q}\}, \{\vec{z}_{1}, \vec{z}_{3}, \vec{z}^{Q}_{2}, \vec{z}^{Q}_{4},  \vec{p}^{Q}\}, \{\vec{z}_{1}, \vec{z}_{3}, \vec{p}^{Q}_{2},\vec{p}^{Q}_{4}\}. \]
Аналогично, множество $\Delta^{R}_{k} \cap \Z^{4}$ совпадает с одним из множеств 
\[ \{\vec{z}_{2}, \vec{z}_{4}\}, \{\vec{z}_{2}, \vec{z}_{4}, \vec{z}^{R}_{1}, \vec{z}^{R}_{3}\}, \{\vec{z}_{2}, \vec{z}_{4}, \vec{p}^{R}\}, \{\vec{z}_{2}, \vec{z}_{4}, \vec{z}^{R}_{1}, \vec{z}^{R}_{3},  \vec{p}^{R}\}, \{\vec{z}_{2}, \vec{z}_{4}, \vec{p}^{R}_{1},\vec{p}^{R}_{3}\}. \]
Рассмотрим следующие случаи:

\textbf{А}) $\Delta^{Q}_{k} \cap \Z^{4} = \{\vec{z}_{1}, \vec{z}_{3}\}$. Если $\vec{z}^{R}_{1} \in \Z^{4}$ и $\vec{z}^{R}_{3} \in \Z^{4}$, то $\vec{z}^{Q}_{2} = \vec{z}_{1} + (\vec{z}_{2} - \vec{z}^{R}_{1}) \in \Z^{4}$ и $\vec{z}^{Q}_{4} \in \Z^{4}$, что противоречит рассматриваемому случаю. Если $\vec{p}^{R}_{1} \in \Z^{4}$ и $\vec{p}^{R}_{3} \in \Z^{4}$, то $\vec{p}^{Q} = \vec{z}_{1} + (\vec{p}^{R}_{3} - \vec{p}^{R}_{1})  \in \Z^{4}$, что также противоречит рассматриваемому случаю. Таким образом, существует ровно два подслучая:

\textbf{А.1}) $\Delta^{R}_{k} \cap \Z^{4} = \{\vec{z}_{2}, \vec{z}_{4}\}$. Этот случай, в свою очередь, разбивается на два подсулчая:

\textbf{А.1.а}) Плоскость $\pi$ рациональная. Тогда, плоскости $Q$ и $R$ равноудалены от $\pi$. Покажем, что параллелограмм $\Delta^{\pi}_{k}$ не содержит точек решетки $\Z^{4}$. Пусть это не так. Будем считать, без ограничения общности, что существует точка $\vec{w} \in \Z^{4}$, лежащая в параллелограмме $\textup{conv}(\vec{z}^{\pi}_{1}, \frac{\vec{z}^{\pi}_{1} + \vec{z}^{\pi}_{2}}{2}, \vec{p}, \frac{\vec{z}^{\pi}_{4} + \vec{z}^{\pi}_{1}}{2})$. Если $\vec{w} \notin \{\frac{\vec{z}^{\pi}_{1} + \vec{z}^{\pi}_{2}}{2}, \frac{\vec{z}^{\pi}_{4} + \vec{z}^{\pi}_{1}}{2}\}$, то $\vec{w} + (\vec{w} - \vec{z}_1) \in \Delta^{R}_{k} \cap \Z^{4}$ и $\vec{w} + (\vec{w} - \vec{z}_1) \notin \{\vec{z}_2, \vec{z}_4\}$, чего не может быть. Если $\vec{w} =  \frac{\vec{z}^{\pi}_{1} + \vec{z}^{\pi}_{2}}{2}$, то $\frac{\vec{z}^{\pi}_{3} + \vec{z}^{\pi}_{4}}{2} = F(\frac{\vec{z}^{\pi}_{1} + \vec{z}^{\pi}_{2}}{2}) \in \Z^{4}$, а значит, $\vec{z}^{Q}_{4} = \vec{z}_1 + (\frac{\vec{z}^{\pi}_{3} + \vec{z}^{\pi}_{4}}{2} - \frac{\vec{z}^{\pi}_{1} + \vec{z}^{\pi}_{2}}{2}) \in \Z^{4}$, чего также не может быть. Аналогично показывается, что случай $\vec{w}  = \frac{\vec{z}^{\pi}_{4} + \vec{z}^{\pi}_{1}}{2}$ невозможен, а значит, $\Delta^{\pi}_{k}$ не содержит точек решетки $\Z^{4}$. Поскольку $\textup{conv}\big(\vec{z}_{1}, \vec{z}_{3},  \vec{z}_{3} + (\vec{z}_{4} - \vec{z}_{2})\big)  \cap \Z^{4} = \{\vec{z}_{1}, \vec{z}_{3},  \vec{z}_{3} + (\vec{z}_{4} - \vec{z}_{2})\}$, то в параллелограмме $\textup{conv}\big(\vec{z}^{\pi}_{3}, \vec{z}^{\pi}_{4},  \vec{z}^{\pi}_{1} + (\vec{z}^{\pi}_{1} -  \vec{z}^{\pi}_{4}), \vec{z}^{\pi}_{2} + (\vec{z}^{\pi}_{2} -  \vec{z}^{\pi}_{3})\big)$ должна существовать точка $\vec{y} \in \Z^{4}$, а значит, по доказанному выше, точка $\vec{y}$ лежит в параллелограмме $\textup{conv}\big(\vec{z}^{\pi}_{2}, \vec{z}^{\pi}_{1},  \vec{z}^{\pi}_{1} + (\vec{z}^{\pi}_{1} -  \vec{z}^{\pi}_{4}), \vec{z}^{\pi}_{2} + (\vec{z}^{\pi}_{2} -  \vec{z}^{\pi}_{3})\big)$. Но тогда $\vec{z}_{3} + (\vec{z}_3 - \vec{y}) \in \Z^{4}$ и $\vec{z}_{3} + (\vec{z}_3 - \vec{y}) \in \Delta^{\pi}_{k}$, чего, как мы показали, не может быть.

\textbf{А.1.б}) (будет соответствовать утверждению (2)) Плоскость $\pi$ не является рациональной плоскостью. Тогда 
\[(\Delta^{\pi}_{k} \cup \Delta^{Q}_{k} \cup \Delta^{R}_{k}) \cap \Z^{4} = \{\vec{z}_{1}, \, \vec{z}_{3}, \, \vec{z}_{2}, \, \vec{z}_{4}\},\]
набор векторов $\vec{z}_{1}$, $\vec{z}_{2}$, $\vec{z}_{3}$, $\vec{z}_{4}$ образует базис решетки $\Z^{4}$, а значит, выполняется утверждение (2).

 \begin{figure}[h]
  \centering
  \begin{tikzpicture}[x=10mm, y=7mm, z=-5mm]
    \begin{scope}[rotate around x=0]
  
    \draw[->] [very thin] (-5,0,0) -- (5,0,0);
    \draw[->] [very thin] (0,-3,0) -- (0,7,0);
    \draw[->] [very thin] (0,0,-4) -- (0,0,4);   
    
    \node[fill=black,circle,inner sep=1pt,opacity=0.5] at (0,0,0) {};
    \node[fill=black,circle,inner sep=1pt,opacity=0.5] at (0,0,2) {};
    \node[fill=black,circle,inner sep=1pt,opacity=0.5] at (0,0,-2) {};
    \node[fill=black,circle,inner sep=1pt,opacity=0.5] at (2,0,0) {};
    \node[fill=black,circle,inner sep=1pt,opacity=0.5] at (2,0,2) {};
    \node[fill=black,circle,inner sep=1pt,opacity=0.5] at (2,0,-2) {};
    \node[fill=black,circle,inner sep=1pt,opacity=0.5] at (-2,0,0) {};
    \node[fill=black,circle,inner sep=1pt,opacity=0.5] at (-2,0,2) {};
    \node[fill=black,circle,inner sep=1pt,opacity=0.5] at (-2,0,-2) {};   
    \node[fill=black,circle,inner sep=1pt,opacity=0.5] at (0,5,0) {};
    \node[fill=black,circle,inner sep=1pt,opacity=0.5] at (0,5,2) {};
    \node[fill=black,circle,inner sep=1pt,opacity=0.5] at (0,5,-2) {};
    \node[fill=black,circle,inner sep=1pt,opacity=0.5] at (2,5,0) {};
    \node[fill=black,circle,inner sep=1pt,opacity=0.5] at (2,5,2) {};
    \node[fill=black,circle,inner sep=1pt,opacity=0.5] at (2,5,-2) {};
    \node[fill=black,circle,inner sep=1pt,opacity=0.5] at (-2,5,0) {};
    \node[fill=black,circle,inner sep=1pt,opacity=0.5] at (-2,5,2) {};
    \node[fill=black,circle,inner sep=1pt,opacity=0.5] at (-2,5,-2) {};
    
    \coordinate (z_1) at (0,0,0);
    \coordinate (z_3) at (0,0,2);
    \coordinate (z_2) at (0,5,0);
    \coordinate (z_4) at (2,0,0);
    
    \coordinate (p_pl_1) at ($ (z_1)!1/2!(z_3) $);
    \coordinate (p_mi_1) at ($ (z_2)!1/2!(z_4) $);
    \coordinate (z_1_z_2_half) at ($ (z_1)!1/2!(z_2) $);
    \coordinate (z_2_z_3_half) at ($ (z_2)!1/2!(z_3) $);
    \coordinate (z_3_z_4_half) at ($ (z_3)!1/2!(z_4) $);
    \coordinate (z_4_z_1_half) at ($ (z_4)!1/2!(z_1) $); 
    
    \coordinate (p) at ($ (p_pl_1)!1/2!(p_mi_1) $);
    \coordinate (z_1_pl_1) at ($(p_mi_1) + (z_1) - (p_pl_1) $);
    \coordinate (z_3_pl_1) at ($(p_mi_1) + (z_3) - (p_pl_1) $);
    \coordinate (z_2_mi_1) at ($(p_pl_1) + (z_2) - (p_mi_1) $);
    \coordinate (z_4_mi_1) at ($(p_pl_1) + (z_4) - (p_mi_1) $);
    
    \coordinate (z_1_p) at ($ (z_1)!1/2!(z_1_pl_1) $);
    \coordinate (z_3_p) at ($(z_3)!1/2!(z_3_pl_1)$);
    \coordinate (z_2_p) at ($(z_2)!1/2!(z_2_mi_1)$);
    \coordinate (z_4_p) at ($(z_4)!1/2!(z_4_mi_1)$);

     \fill[blue!05,opacity=0.3] (3,0,3) -- (3,0,-3) -- (-3,0,-3) -- (-3,0,3) -- cycle;
     \fill[blue!05,opacity=0.3] (3,5,3) -- (3,5,-3) -- (-3,5,-3) -- (-3,5,3) -- cycle;

     \fill[blue!25,opacity=0.6] (z_1) -- (z_2_mi_1) -- (z_3) -- (z_4_mi_1) -- cycle;     
     \fill[blue!25,opacity=0.6] (z_1_pl_1) -- (z_2) -- (z_3_pl_1) -- (z_4) -- cycle;
       
    \node[fill=blue,circle,inner sep=1pt] at (z_1) {};
    \draw (z_1) node[below right] {$\vec z_1$};
    
    \node[fill=blue,circle,inner sep=1pt] at (z_3) {};
    \draw (z_3) node[left] {$\vec z_3$};
    
    \node[fill=blue,circle,inner sep=1pt] at (z_2) {};
    \draw (z_2) node[left] {$\vec z_2$};
    
     \node[fill=blue,circle,inner sep=1pt] at (z_4) {};
    \draw (z_4) node[above right] {$\vec z_4$};
  
    \node[fill=white,circle,inner sep=1pt,draw=blue] at (z_1_pl_1) {};
    \draw (z_1_pl_1) node[above right] {$\vec z_1^{R}$};
    
    \node[fill=white,circle,inner sep=1pt,draw=blue] at (z_3_pl_1) {};
    \draw (z_3_pl_1) node[above] {$\vec z_3^{R}$};
    
    \node[fill=white,circle,inner sep=1pt,draw=blue] at (z_2_mi_1) {};
    \draw (z_2_mi_1) node[left] {$\vec z_2^{Q}$};
    
    \node[fill=white,circle,inner sep=1pt, draw=blue] at (z_4_mi_1) {};
    \draw (z_4_mi_1) node[left] {$\vec z_4^{Q}$};

    \node[fill=white,circle,inner sep=1pt, draw=blue] at (p) {};
    \draw (p) node[right] {$\vec p$};
    
    \node[fill=white,circle,inner sep=1pt,draw=blue] at (p_pl_1) {};
    \draw (p_pl_1) node[above] {$\vec{p}^{Q}$};
    
    \node[fill=white,circle,inner sep=1pt,draw=blue] at (p_mi_1) {};
    \draw (p_mi_1) node[right] {$\vec{p}^{R}$};
    
    \draw[->] [red,very thin] (z_1) -- (z_2);
    \draw[->] [red,very thin] (z_1) -- (z_3);
    \draw[->] [red,very thin] (z_1) -- (z_4);

    \draw (5, 0, 0) node[right] {$x$};
    \draw (0, 7, 0) node[left] {$z$};
    \draw (0, 0, 4) node[left] {$y$};

    \end{scope}
  \end{tikzpicture}
  \caption{Расположение точек внутри гиперплоскости $S_{1}$ из случая \textbf{А.1.б} леммы \ref{main_lem_2_2}}
\end{figure}

\textbf{А.2}) $\Delta^{R}_{k} \cap \Z^{4} = \{\vec{z}_{2}, \vec{z}_{4}, \vec{p}^{R}\}$. Этот случай, в свою очередь, разбивается на два подсулчая:

\textbf{А.2.а}) Плоскость $\pi$ рациональная. Тогда, плоскости $Q$ и $R$ равноудалены от $\pi$. Покажем, что параллелограмм $\Delta^{\pi}_{k}$ не содержит точек решетки $\Z^{4}$. Пусть это не так. Можно считать, что существует точка $\vec{w} \in \Z^{4}$, лежащая либо в параллелограмме $\textup{conv}(\vec{z}^{\pi}_{1}, \frac{\vec{z}^{\pi}_{1} + \vec{z}^{\pi}_{2}}{2}, \vec{p}, \frac{\vec{z}^{\pi}_{4} + \vec{z}^{\pi}_{1}}{2})$, либо в параллелограмме $\textup{conv}(\vec{z}^{\pi}_{2}, \frac{\vec{z}^{\pi}_{2} + \vec{z}^{\pi}_{3}}{2}, \vec{p}, \frac{\vec{z}^{\pi}_{1} + \vec{z}^{\pi}_{2}}{2})$. Рассмотрим случай $\vec{w} \in \textup{conv}(\vec{z}^{\pi}_{1}, \frac{\vec{z}^{\pi}_{1} + \vec{z}^{\pi}_{2}}{2}, \vec{p}, \frac{\vec{z}^{\pi}_{4} + \vec{z}^{\pi}_{1}}{2})$. Если $\vec{w} \notin \{\frac{\vec{z}^{\pi}_{1} + \vec{z}^{\pi}_{2}}{2},  \frac{\vec{z}^{\pi}_{4} + \vec{z}^{\pi}_{1}}{2}, \frac{\vec{z}^{\pi}_{1} +  \vec{p}}{2}\}$, то $\vec{w} + (\vec{w} - \vec{z}_1) \in \Delta^{R}_{k} \cap \Z^{4}$ и $\vec{w} + (\vec{w} - \vec{z}_1) \notin \{\vec{z}_2, \vec{z}_4, \vec{p}^{R}\}$, чего не может быть. Если $\vec{w} =  \frac{\vec{z}^{\pi}_{1} +  \vec{p}}{2}$, то $ \frac{\vec{z}^{\pi}_{3} +  \vec{p}}{2} = F( \frac{\vec{z}^{\pi}_{1} +  \vec{p}}{2}) \in \Z^{4}$, а значит, $\vec{p}^{Q} = \vec{z}_1 + ( \frac{\vec{z}^{\pi}_{3} +  \vec{p}}{2} - \frac{\vec{z}^{\pi}_{1} +  \vec{p}}{2}) \in \Z^{4}$, чего не может быть. Если $\vec{w} =  \frac{\vec{z}^{\pi}_{1} + \vec{z}^{\pi}_{2}}{2}$, то $\frac{\vec{z}^{\pi}_{3} + \vec{z}^{\pi}_{4}}{2} = F(\frac{\vec{z}^{\pi}_{1} + \vec{z}^{\pi}_{2}}{2}) \in \Z^{4}$, а значит, $\vec{z}^{Q}_{4} = \vec{z}_1 + (\frac{\vec{z}^{\pi}_{3} + \vec{z}^{\pi}_{4}}{2} - \frac{\vec{z}^{\pi}_{1} + \vec{z}^{\pi}_{2}}{2}) \in \Z^{4}$, чего также не может быть. Аналогично показывается, что случай $\vec{w}  = \frac{\vec{z}^{\pi}_{1} + \vec{z}^{\pi}_{4}}{2}$ невозможен, а значит, случай $\vec{w} \in \textup{conv}(\vec{z}^{\pi}_{1}, \frac{\vec{z}^{\pi}_{1} + \vec{z}^{\pi}_{2}}{2}, \vec{p}, \frac{\vec{z}^{\pi}_{4} + \vec{z}^{\pi}_{1}}{2})$ невозможен. 
Теперь рассмотрим случай $\vec{w} \in \textup{conv}(\vec{z}^{\pi}_{2}, \frac{\vec{z}^{\pi}_{2} + \vec{z}^{\pi}_{3}}{2}, \vec{p}, \frac{\vec{z}^{\pi}_{1} + \vec{z}^{\pi}_{2}}{2})$. Если $\vec{w} \notin \{\frac{\vec{z}^{\pi}_{1} + \vec{z}^{\pi}_{2}}{2}, \frac{\vec{z}^{\pi}_{2} + \vec{z}^{\pi}_{3}}{2}\}$, то $\vec{w} + (\vec{w} - \vec{z}_2) \in \Delta^{Q}_{k} \cap \Z^{4}$ и $\vec{w} + (\vec{w} - \vec{z}_2) \notin \{\vec{z}_1, \vec{z}_3\}$, чего не может быть. По доказанному выше, случаи $\vec{w} =  \frac{\vec{z}^{\pi}_{1} + \vec{z}^{\pi}_{2}}{2}$ и $\vec{w} =  \frac{\vec{z}^{\pi}_{2} + \vec{z}^{\pi}_{3}}{2}$ также невозможны, а значит, $\Delta^{\pi}_{k}$ не содержит точек решетки $\Z^{4}$. Поскольку $\textup{conv}\big(\vec{z}_{1}, \vec{z}_{3},  \vec{z}_{3} + (\vec{z}_{4} - \vec{p}^{R})\big)  \cap \Z^{4} = \{\vec{z}_{1}, \vec{z}_{3},  \vec{z}_{3} + (\vec{z}_{4} - \vec{p}^{R})\}$, то в параллелограмме $\textup{conv}\big(\vec{z}^{\pi}_{3}, \vec{z}^{\pi}_{4},  \vec{z}^{\pi}_{1} + (\vec{z}^{\pi}_{1} -  \vec{z}^{\pi}_{4}), \vec{z}^{\pi}_{2} + (\vec{z}^{\pi}_{2} -  \vec{z}^{\pi}_{3})\big)$ должна существовать точка $\vec{y} \in \Z^{4}$, а значит, по доказанному выше, точка $\vec{y}$ лежит в параллелограмме $\textup{conv}\big(\vec{z}^{\pi}_{2}, \vec{z}^{\pi}_{1},  \vec{z}^{\pi}_{1} + (\vec{z}^{\pi}_{1} -  \vec{z}^{\pi}_{4}), \vec{z}^{\pi}_{2} + (\vec{z}^{\pi}_{2} -  \vec{z}^{\pi}_{3})\big)$. Но тогда $\vec{z}_{3} + (\vec{z}_3 - \vec{y}) \in \Z^{4}$ и $\vec{z}_{3} + (\vec{z}_3 - \vec{y}) \in \Delta^{\pi}_{k}$, чего, как мы показали, не может быть.

\textbf{А.2.б}) (будет соответствовать утверждению (8)) Плоскость $\pi$ не является рациональной плоскостью. Тогда 
\[(\Delta^{\pi}_{k} \cup \Delta^{Q}_{k} \cup \Delta^{R}_{k}) \cap \Z^{4} = \{\vec{z}_{1}, \, \vec{z}_{3}, \, \vec{z}_{2}, \, \vec{z}_{4}, \vec{p}^{R}\},\]
набор векторов $\vec{z}_{1}$, $\vec{z}_{2}$, $\vec{z}_{3}$, $\vec{p}^{R} = \frac{1}{2}(\vec{z}_{2} + \vec{z}_{4})$ образует базис решетки $\Z^{4}$, а значит, выполняется утверждение (8).

 \begin{figure}[h]
  \centering
  \begin{tikzpicture}[x=10mm, y=7mm, z=-5mm, scale=0.8]
    \begin{scope}[rotate around x=0]

    \draw[->] [very thin] (-5,0,0) -- (5,0,0);
    \draw[->] [very thin] (0,-3,0) -- (0,7,0);
    \draw[->] [very thin] (0,0,-4) -- (0,0,4);

    \node[fill=black,circle,inner sep=1pt,opacity=0.5] at (0,0,0) {};
    \node[fill=black,circle,inner sep=1pt,opacity=0.5] at (0,0,2) {};
    \node[fill=black,circle,inner sep=1pt,opacity=0.5] at (0,0,-2) {};
    \node[fill=black,circle,inner sep=1pt,opacity=0.5] at (2,0,0) {};
    \node[fill=black,circle,inner sep=1pt,opacity=0.5] at (2,0,2) {};
    \node[fill=black,circle,inner sep=1pt,opacity=0.5] at (2,0,-2) {};
    \node[fill=black,circle,inner sep=1pt,opacity=0.5] at (-2,0,0) {};
    \node[fill=black,circle,inner sep=1pt,opacity=0.5] at (-2,0,2) {};
    \node[fill=black,circle,inner sep=1pt,opacity=0.5] at (-2,0,-2) {};
    \node[fill=black,circle,inner sep=1pt,opacity=0.5] at (0,5,0) {};
    \node[fill=black,circle,inner sep=1pt,opacity=0.5] at (0,5,2) {};
    \node[fill=black,circle,inner sep=1pt,opacity=0.5] at (0,5,-2) {};
    \node[fill=black,circle,inner sep=1pt,opacity=0.5] at (2,5,0) {};
    \node[fill=black,circle,inner sep=1pt,opacity=0.5] at (2,5,2) {};
    \node[fill=black,circle,inner sep=1pt,opacity=0.5] at (2,5,-2) {};
    \node[fill=black,circle,inner sep=1pt,opacity=0.5] at (-2,5,0) {};
    \node[fill=black,circle,inner sep=1pt,opacity=0.5] at (-2,5,2) {};
    \node[fill=black,circle,inner sep=1pt,opacity=0.5] at (-2,5,-2) {};
    \node[fill=black,circle,inner sep=1pt,opacity=0.5] at (4,5,2) {};
    \node[fill=black,circle,inner sep=1pt,opacity=0.5] at (4,5,0) {};
    \node[fill=black,circle,inner sep=1pt,opacity=0.5] at (4,5,-2) {};
    
    \coordinate (z_1) at (0,0,0);
    \coordinate (z_3) at (0,0,2);
    \coordinate (z_2) at (0,5,0);
    \coordinate (z_4) at (4,5,0);
    
    \coordinate (p_pl_1) at ($ (z_1)!1/2!(z_3) $);
    \coordinate (p_mi_1) at ($ (z_2)!1/2!(z_4) $);
    \coordinate (z_1_z_2_half) at ($ (z_1)!1/2!(z_2) $);
    \coordinate (z_2_z_3_half) at ($ (z_2)!1/2!(z_3) $);
    \coordinate (z_3_z_4_half) at ($ (z_3)!1/2!(z_4) $);
    \coordinate (z_4_z_1_half) at ($ (z_4)!1/2!(z_1) $); 
    
    \coordinate (p) at ($ (p_pl_1)!1/2!(p_mi_1) $);
    \coordinate (z_1_pl_1) at ($(p_mi_1) + (z_1) - (p_pl_1) $);
    \coordinate (z_3_pl_1) at ($(p_mi_1) + (z_3) - (p_pl_1) $);
    \coordinate (z_2_mi_1) at ($(p_pl_1) + (z_2) - (p_mi_1) $);
    \coordinate (z_4_mi_1) at ($(p_pl_1) + (z_4) - (p_mi_1) $);
    
    \coordinate (z_1_p) at ($ (z_1)!1/2!(z_1_pl_1) $);
    \coordinate (z_3_p) at ($(z_3)!1/2!(z_3_pl_1)$);
    \coordinate (z_2_p) at ($(z_2)!1/2!(z_2_mi_1)$);
    \coordinate (z_4_p) at ($(z_4)!1/2!(z_4_mi_1)$);

    \fill[blue!05,opacity=0.3] (3,0,3) -- (3,0,-3) -- (-3,0,-3) -- (-3,0,3) -- cycle;
    \fill[blue!05,opacity=0.3] (5,5,3) -- (5,5,-3) -- (-3,5,-3) -- (-3,5,3) -- cycle;

    \fill[blue!25,opacity=0.6] (z_1) -- (z_2_mi_1) -- (z_3) -- (z_4_mi_1) -- cycle;
    \fill[blue!25,opacity=0.6] (z_1_pl_1) -- (z_2) -- (z_3_pl_1) -- (z_4) -- cycle;

    \node[fill=blue,circle,inner sep=1pt] at (z_1) {};
    \draw (z_1) node[below right] {$\vec z_1$};
    
    \node[fill=blue,circle,inner sep=1pt] at (z_3) {};
    \draw (z_3) node[below right] {$\vec z_3$};
    
    \node[fill=blue,circle,inner sep=1pt] at (z_2) {};
    \draw (z_2) node[left] {$\vec z_2$};
    
     \node[fill=blue,circle,inner sep=1pt] at (z_4) {};
    \draw (z_4) node[above right] {$\vec z_4$};
  
    \node[fill=white,circle,inner sep=1pt,draw=blue] at (z_1_pl_1) {};
    \draw (z_1_pl_1) node[above] {$\vec z_1^{R}$};
    
    \node[fill=white,circle,inner sep=1pt,draw=blue] at (z_3_pl_1) {};
    \draw (z_3_pl_1) node[left] {$\vec z_3^{R}$};
    
    \node[fill=white,circle,inner sep=1pt,draw=blue] at (z_2_mi_1) {};
    \draw (z_2_mi_1) node[left] {$\vec z_2^{Q}$};
    
    \node[fill=white,circle,inner sep=1pt, draw=blue] at (z_4_mi_1) {};
    \draw (z_4_mi_1) node[right] {$\vec z_4^{Q}$};

    \node[fill=white,circle,inner sep=1pt, draw=blue] at (p) {};
    \draw (p) node[left] {$\vec p$};
    
    \node[fill=white,circle,inner sep=1pt,draw=blue] at (p_pl_1) {};
    \draw (p_pl_1) node[left] {$\vec{p}^{Q}$};
    
    \node[fill=blue,circle,inner sep=1pt,draw=blue] at (p_mi_1) {};
    \draw (p_mi_1) node[right] {$\vec{p}^{R}$};
    
    \draw[->] [red,very thin] (z_1) -- (z_2);
    \draw[->] [red,very thin] (z_1) -- (z_3);
    \draw[->] [red,very thin] (z_1) -- (p_mi_1);

    \draw (5, 0, 0) node[right] {$x$};
    \draw (0, 7, 0) node[left] {$z$};
    \draw (0, 0, 4) node[left] {$y$};

    \end{scope}
  \end{tikzpicture}
  \caption{Расположение точек внутри гиперплоскости $S_{1}$ из случая \textbf{А.2.б} леммы \ref{main_lem_2_2}}
\end{figure}

\textbf{Б}) $\Delta^{Q}_{k} \cap \Z^{4} = \{\vec{z}_{1}, \vec{z}_{3}, \vec{z}^{Q}_{2}, \vec{z}^{Q}_{4}\}$. Заметим, что случай $\Delta^{R}_{k} \cap \Z^{4} = \{\vec{z}_{2}, \vec{z}_{4}\}$ с точностью до перестановки индексов невозможен в силу пункта \textbf{А}. Если $\vec{p}^{R} \in \Z^{4}$, то  $\vec{p}^{Q} =  \vec{z}^{Q}_{2} + (\vec{p}^{R} - \vec{z}_{2}) \in \Z^{4}$, что противоречит рассматриваемому случаю. Если $\vec{p}^{R}_{1} \in \Z^{4}$ и $\vec{p}^{R}_{3} \in \Z^{4}$, то $\vec{p}^{Q} = \vec{z}_{1} + (\vec{p}^{R}_{3} - \vec{p}^{R}_{1})  \in \Z^{4}$, что также противоречит рассматриваемому случаю. Таким образом, существует ровно один подслучай:

\textbf{Б.1}) $\Delta^{R}_{k} \cap \Z^{4} = \{\vec{z}_{2}, \vec{z}_{4}, \vec{z}^{R}_{1}, \vec{z}^{R}_{3}\}$. Этот случай, в свою очередь, разбивается на два подслучая:

\textbf{Б.1.а}) Плоскость $\pi$ рациональная. Тогда, плоскости $Q$ и $R$ равноудалены от $\pi$. Поскольку $\textup{conv}\big(\vec{z}_{1}, \vec{z}_{3},  \vec{z}_{3} + (\vec{z}_{4} - \vec{z}^{R}_{3})\big)  \cap \Z^{4} = \{\vec{z}_{1}, \vec{z}_{3},  \vec{z}_{3} + (\vec{z}_{4} - \vec{z}^{R}_{3})\}$, то в параллелограмме $\Delta^{\pi}_{k}$ должна существовать хотя бы одна точка решетки $\Z^{4}$. Пусть $\vec{w} \in \Delta^{\pi}_{k} \cap \Z^{4}$. Покажем, что $\vec{w} \in \{\vec{z}^{\pi}_{1}, \vec{z}^{\pi}_{2}, \vec{z}^{\pi}_{3}, \vec{z}^{\pi}_{4}, \vec{p}, \frac{\vec{z}^{\pi}_{1} + \vec{z}^{\pi}_{2}}{2}, \frac{\vec{z}^{\pi}_{2} + \vec{z}^{\pi}_{3}}{2}, \frac{\vec{z}^{\pi}_{3} + \vec{z}^{\pi}_{4}}{2}, \frac{\vec{z}^{\pi}_{4} + \vec{z}^{\pi}_{1}}{2} \}$. Предположим, что это не так. Можно считать, что либо 
\[\vec{w} \in \textup{conv}(\vec{z}^{\pi}_{1}, \frac{\vec{z}^{\pi}_{1} + \vec{z}^{\pi}_{2}}{2}, \vec{p}, \frac{\vec{z}^{\pi}_{4} + \vec{z}^{\pi}_{1}}{2}),\]
либо 
\[\vec{w} \in \textup{conv}(\vec{z}^{\pi}_{2}, \frac{\vec{z}^{\pi}_{2} + \vec{z}^{\pi}_{3}}{2}, \vec{p}, \frac{\vec{z}^{\pi}_{1} + \vec{z}^{\pi}_{2}}{2}).\] Если $\vec{w} \in \textup{conv}(\vec{z}^{\pi}_{1}, \frac{\vec{z}^{\pi}_{1} + \vec{z}^{\pi}_{2}}{2}, \vec{p}, \frac{\vec{z}^{\pi}_{4} + \vec{z}^{\pi}_{1}}{2})$, то $\vec{w} + (\vec{w} - \vec{z}_1) \in \Delta^{R}_{k} \cap \Z^{4}$ и $\vec{w} + (\vec{w} - \vec{z}_1) \notin \{\vec{z}_{2}, \vec{z}_{4}, \vec{z}^{R}_{1}, \vec{z}^{R}_{3}\}$, чего не может быть. Если $\vec{w} \in \textup{conv}(\vec{z}^{\pi}_{2}, \frac{\vec{z}^{\pi}_{2} + \vec{z}^{\pi}_{3}}{2}, \vec{p}, \frac{\vec{z}^{\pi}_{1} + \vec{z}^{\pi}_{2}}{2})$, то $\vec{w} + (\vec{w} - \vec{z}_2) \in \Delta^{Q}_{k} \cap \Z^{4}$ и $\vec{w} + (\vec{w} - \vec{z}_2) \notin \{\vec{z}_{1}, \vec{z}_{3}, \vec{z}^{Q}_{2}, \vec{z}^{Q}_{4}\}$, чего не может быть. Далее рассмотрим подслучаи:

\textbf{Б.1.а.1}) (будет соответствовать утверждению (7)) Пусть $\vec{w} \in \{\vec{z}^{\pi}_{1}, \vec{z}^{\pi}_{2}, \vec{z}^{\pi}_{3}, \vec{z}^{\pi}_{4} \}$. В этом случае каждая из точек множества $\{\vec{z}^{\pi}_{1}, \vec{z}^{\pi}_{2}, \vec{z}^{\pi}_{3}, \vec{z}^{\pi}_{4}\}$ принадлежит решетке $\Z^{4}$. При этом, так как $\Delta^{Q}_{k} \cap \Z^{4} = \{\vec{z}_{1}, \vec{z}_{3}, \vec{z}^{Q}_{2}, \vec{z}^{Q}_{4}\}$, то никакая из точек множества $\{\vec{p}, \frac{\vec{z}^{\pi}_{1} + \vec{z}^{\pi}_{2}}{2}, \frac{\vec{z}^{\pi}_{2} + \vec{z}^{\pi}_{3}}{2}, \frac{\vec{z}^{\pi}_{3} + \vec{z}^{\pi}_{4}}{2}, \frac{\vec{z}^{\pi}_{4} + \vec{z}^{\pi}_{1}}{2} \}$ не принадлежит решетке $\Z^{4}$. Тогда 
\[(\Delta^{\pi}_{k} \cup \Delta^{Q}_{k} \cup \Delta^{R}_{k}) \cap \Z^{4} = \{\vec{z}_{1}, \, \vec{z}^{Q}_{2}, \, \vec{z}_{3}, \, \vec{z}^{Q}_{4}, \, \vec{z}^{R}_{1}, \, \vec{z}_{2}, \, \vec{z}^{R}_{3}, \, \vec{z}_{4}, \, \vec{z}^{\pi}_{1}, \, \vec{z}^{\pi}_{2}, \, \vec{z}^{\pi}_{3}, \, \vec{z}^{\pi}_{4}\}\]
и, так как плоскости $Q$ и $R$ равноудалены от $\pi$, набор векторов $\vec{z}_{1}$, $\vec{z}_{2}$, $\vec{z}_{3}$, $\vec{z}^{\pi}_{1} = \frac{1}{2}(\vec{z}_{1}+\vec{z}_{2}) + \frac{1}{4}(\vec{z}_{1}+\vec{z}_{4} - \vec{z}_{3} - \vec{z}_{2})$ образует базис решетки $\Z^{4}$, а значит, выполняется утверждение (7).

  \begin{figure}[h]
  \centering
  \begin{tikzpicture}[x=10mm, y=7mm, z=-5mm, scale=0.8]
    \begin{scope}[rotate around x=0]

    \draw[->] [very thin] (-5,0,0) -- (5,0,0);
    \draw[->] [very thin] (0,-5,0) -- (0,13,0);
    \draw[->] [very thin] (0,0,-5) -- (0,0,7);

    \node[fill=black,circle,inner sep=1pt,opacity=0.5] at (0,0,0) {};
    \node[fill=black,circle,inner sep=1pt,opacity=0.5] at (0,0,2) {};
    \node[fill=black,circle,inner sep=1pt,opacity=0.5] at (0,0,-2) {};
    \node[fill=black,circle,inner sep=1pt,opacity=0.5] at (2,0,0) {};
    \node[fill=black,circle,inner sep=1pt,opacity=0.5] at (2,0,2) {};
    \node[fill=black,circle,inner sep=1pt,opacity=0.5] at (2,0,-2) {};
    \node[fill=black,circle,inner sep=1pt,opacity=0.5] at (-2,0,0) {};
    \node[fill=black,circle,inner sep=1pt,opacity=0.5] at (-2,0,2) {};
    \node[fill=black,circle,inner sep=1pt,opacity=0.5] at (-2,0,-2) {};
    \node[fill=black,circle,inner sep=1pt,opacity=0.5] at (2,0,4) {};
    \node[fill=black,circle,inner sep=1pt,opacity=0.5] at (0,5,0) {};
    \node[fill=black,circle,inner sep=1pt,opacity=0.5] at (0,5,2) {};
    \node[fill=black,circle,inner sep=1pt,opacity=0.5] at (0,5,-2) {};
    \node[fill=black,circle,inner sep=1pt,opacity=0.5] at (2,5,0) {};
    \node[fill=black,circle,inner sep=1pt,opacity=0.5] at (2,5,2) {};
    \node[fill=black,circle,inner sep=1pt,opacity=0.5] at (2,5,-2) {};
    \node[fill=black,circle,inner sep=1pt,opacity=0.5] at (-2,5,0) {};
    \node[fill=black,circle,inner sep=1pt,opacity=0.5] at (-2,5,2) {};
    \node[fill=black,circle,inner sep=1pt,opacity=0.5] at (-2,5,-2) {}; 
    \node[fill=black,circle,inner sep=1pt,opacity=0.5] at (0,10,0) {};
    \node[fill=black,circle,inner sep=1pt,opacity=0.5] at (0,10,2) {};
    \node[fill=black,circle,inner sep=1pt,opacity=0.5] at (0,10,-2) {};
    \node[fill=black,circle,inner sep=1pt,opacity=0.5] at (2,10,0) {};
    \node[fill=black,circle,inner sep=1pt,opacity=0.5] at (2,10,2) {};
    \node[fill=black,circle,inner sep=1pt,opacity=0.5] at (2,10,-2) {};
    \node[fill=black,circle,inner sep=1pt,opacity=0.5] at (-2,10,0) {};
    \node[fill=black,circle,inner sep=1pt,opacity=0.5] at (-2,10,2) {};
    \node[fill=black,circle,inner sep=1pt,opacity=0.5] at (-2,10,-2) {};    
    
    \coordinate (z_1) at (2,10,-2);
    \coordinate (z_3) at (4,10,0);
    \coordinate (z_2) at (-2,0,4);
    \coordinate (z_4) at (0,0,2);
    
    \coordinate (p_pl_1) at ($ (z_1)!1/2!(z_3) $);
    \coordinate (p_mi_1) at ($ (z_2)!1/2!(z_4) $);
    \coordinate (z_1_z_2_half) at ($ (z_1)!1/2!(z_2) $);
    \coordinate (z_2_z_3_half) at ($ (z_2)!1/2!(z_3) $);
    \coordinate (z_3_z_4_half) at ($ (z_3)!1/2!(z_4) $);
    \coordinate (z_4_z_1_half) at ($ (z_4)!1/2!(z_1) $); 
    
    \coordinate (p) at ($ (p_pl_1)!1/2!(p_mi_1) $);
    \coordinate (z_1_pl_1) at ($(p_mi_1) + (z_1) - (p_pl_1) $);
    \coordinate (z_3_pl_1) at ($(p_mi_1) + (z_3) - (p_pl_1) $);
    \coordinate (z_2_mi_1) at ($(p_pl_1) + (z_2) - (p_mi_1) $);
    \coordinate (z_4_mi_1) at ($(p_pl_1) + (z_4) - (p_mi_1) $);
    
    \coordinate (z_1_p) at ($ (z_1)!1/2!(z_1_pl_1) $);
    \coordinate (z_3_p) at ($(z_3)!1/2!(z_3_pl_1)$);
    \coordinate (z_2_p) at ($(z_2)!1/2!(z_2_mi_1)$);
    \coordinate (z_4_p) at ($(z_4)!1/2!(z_4_mi_1)$);
    
    \fill[blue!05,opacity=0.3] (3,0,5) -- (3,0,-3) -- (-3,0,-3) -- (-3,0,5) -- cycle;
    \fill[blue!05,opacity=0.3] (3,5,3) -- (3,5,-3) -- (-3,5,-3) -- (-3,5,3) -- cycle;
    \fill[blue!05,opacity=0.3] (3,10,3) -- (3,10,-3) -- (-3,10,-3) -- (-3,10,3) -- cycle;
    
    \fill[blue!25,opacity=0.6] (z_1) -- (z_2_mi_1) -- (z_3) -- (z_4_mi_1) -- cycle; 
    \fill[blue!25,opacity=0.6] (z_1_p) -- (z_2_p) -- (z_3_p) -- (z_4_p) -- cycle;     
    \fill[blue!25,opacity=0.6] (z_1_pl_1) -- (z_2) -- (z_3_pl_1) -- (z_4) -- cycle;
         
    \node[fill=blue,circle,inner sep=1pt] at (z_1) {};
    \draw (z_1) node[above] {$\vec z_1$};
    
    \node[fill=blue,circle,inner sep=1pt] at (z_3) {};
    \draw (z_3) node[right] {$\vec z_3$};
    
    \node[fill=blue,circle,inner sep=1pt] at (z_2) {};
    \draw (z_2) node[left] {$\vec z_2$};
    
     \node[fill=blue,circle,inner sep=1pt] at (z_4) {};
    \draw (z_4) node[right] {$\vec z_4$};
  
    \node[fill=blue,circle,inner sep=1pt] at (z_1_pl_1) {};
    \draw (z_1_pl_1) node[above] {$\vec z_1^{R}$};
    
    \node[fill=blue,circle,inner sep=1pt] at (z_3_pl_1) {};
    \draw (z_3_pl_1) node[right] {$\vec z_3^{R}$};
    
    \node[fill=blue,circle,inner sep=1pt] at (z_2_mi_1) {};
    \draw (z_2_mi_1) node[left] {$\vec z_2^{Q}$};
    
    \node[fill=blue,circle,inner sep=1pt] at (z_4_mi_1) {};
    \draw (z_4_mi_1) node[right] {$\vec z_4^{Q}$};
    
     \node[fill=blue,circle,inner sep=1pt] at (z_1_p) {};
    \draw (z_1_p) node[left] {$\vec z_1^{\pi}$};
    
    \node[fill=blue,circle,inner sep=1pt] at (z_3_p) {};
    \draw (z_3_p) node[below right] {$\vec z_3^{\pi}$};
    
     \node[fill=blue,circle,inner sep=1pt] at (z_2_p) {};
    \draw (z_2_p) node[below left] {$\vec z_2^{\pi}$};
    
     \node[fill=blue,circle,inner sep=1pt] at (z_4_p) {};
    \draw (z_4_p) node[above right] {$\vec z_4^{\pi}$};

    \node[fill=white,circle,inner sep=1pt,draw=blue] at (p) {};
    \draw (p) node[right] {$\vec p$};

    \node[fill=white,circle,inner sep=1pt,draw=blue] at (p_pl_1) {};
    \draw (p_pl_1) node[right] {$\vec{p}^{Q}$};
    
    \node[fill=white,circle,inner sep=1pt,draw=blue] at (p_mi_1) {};
    \draw (p_mi_1) node[above] {$\vec{p}^{R}$};
    
    \node[fill=white,circle,inner sep=1pt,draw=blue] at (z_1_z_2_half) {};
    \draw (z_1_z_2_half) node[left] {$\frac{\vec{z}_1^{\pi} + \vec{z}_2^{\pi}}{2}$};
    
    \node[fill=white,circle,inner sep=1pt,draw=blue] at (z_2_z_3_half) {};
    \draw (z_2_z_3_half) node[below] {$\frac{\vec{z}_2^{\pi} + \vec{z}_3^{\pi}}{2}$};
    
    \node[fill=white,circle,inner sep=1pt,draw=blue] at (z_3_z_4_half) {};
    \draw (z_3_z_4_half) node[right] {$\frac{\vec{z}_3^{\pi} + \vec{z}_4^{\pi}}{2}$};
    
    \node[fill=white,circle,inner sep=1pt,draw=blue] at (z_4_z_1_half) {};
    \draw (z_4_z_1_half) node[above] {$\frac{\vec{z}_4^{\pi} + \vec{z}_1^{\pi}}{2}$};
    
    \draw[->] [red,very thin] (z_1) -- (z_3);
    \draw[->] [red,very thin] (z_1) -- (z_2);
    \draw[->] [red,very thin] (z_1) -- (z_1_p);

    \draw (5, 0, 0) node[right] {$x$};
    \draw (0, 13, 0) node[left] {$z$};
    \draw (0, 0, 7) node[left] {$y$};

    \end{scope}
  \end{tikzpicture}
  \caption{Расположение точек внутри гиперплоскости $S_{1}$ из случая \textbf{Б.1.а.1} леммы \ref{main_lem_2_2}}
\end{figure}

\textbf{Б.1.а.2}) (будет соответствовать утверждению (1)) Пусть $\vec{w} = \vec{p}$. В силу доказательства случая \textbf{Б.1.а.1} никакая из точек множества $\{\vec{z}^{\pi}_{1}, \vec{z}^{\pi}_{2}, \vec{z}^{\pi}_{3}, \vec{z}^{\pi}_{4} \}$ не принадлежит решетке $\Z^{4}$. Кроме того, так как $\Delta^{Q}_{k} \cap \Z^{4} = \{\vec{z}_{1}, \vec{z}_{3}, \vec{z}^{Q}_{2}, \vec{z}^{Q}_{4}\}$, то никакая из точек множества $\{ \frac{\vec{z}^{\pi}_{1} + \vec{z}^{\pi}_{2}}{2}, \frac{\vec{z}^{\pi}_{2} + \vec{z}^{\pi}_{3}}{2}, \frac{\vec{z}^{\pi}_{3} + \vec{z}^{\pi}_{4}}{2}, \frac{\vec{z}^{\pi}_{4} + \vec{z}^{\pi}_{1}}{2} \}$ не принадлежит решетке $\Z^{4}$. Тогда
\[(\Delta^{\pi}_{k} \cup \Delta^{Q}_{k} \cup \Delta^{R}_{k}) \cap \Z^{4} = \{\vec{z}_{1}, \, \vec{z}^{Q}_{2}, \, \vec{z}_{3}, \, \vec{z}^{Q}_{4}, \, \vec{z}^{R}_{1}, \, \vec{z}_{2}, \, \vec{z}^{R}_{3}, \, \vec{z}_{4}, \, \vec{p}\}\]
и, так как плоскости $Q$ и $R$ равноудалены от $\pi$, набор векторов $\vec{z}_{1}$, $\vec{z}_{2}$, $\vec{z}_{3}$, $\vec{p} = \frac{1}{4}(\vec{z}_{1}+\vec{z}_{2}+\vec{z}_{3}+\vec{z}_{4})$ образует базис решетки $\Z^{4}$, а значит, выполняется утверждение (1).

 \begin{figure}[h]
  \centering
  \begin{tikzpicture}[x=10mm, y=7mm, z=-5mm]
    \begin{scope}[rotate around x=0]
    
    \draw[->] [very thin] (-5,0,0) -- (5,0,0);
    \draw[->] [very thin] (0,-7,0) -- (0,7,0);
    \draw[->] [very thin] (0,0,-7) -- (0,0,7);

    \node[fill=black,circle,inner sep=1pt,opacity=0.5] at (0,0,0) {};
    \node[fill=black,circle,inner sep=1pt,opacity=0.5] at (0,0,2) {};
    \node[fill=black,circle,inner sep=1pt,opacity=0.5] at (0,0,-2) {};
    \node[fill=black,circle,inner sep=1pt,opacity=0.5] at (2,0,0) {};
    \node[fill=black,circle,inner sep=1pt,opacity=0.5] at (2,0,-2) {};
    \node[fill=black,circle,inner sep=1pt,opacity=0.5] at (-2,0,0) {};
    \node[fill=black,circle,inner sep=1pt,opacity=0.5] at (-2,0,2) {};
    \node[fill=black,circle,inner sep=1pt,opacity=0.5] at (0,-5,0) {};
    \node[fill=black,circle,inner sep=1pt,opacity=0.5] at (0,-5,2) {};
    \node[fill=black,circle,inner sep=1pt,opacity=0.5] at (0,-5,-2) {};
    \node[fill=black,circle,inner sep=1pt,opacity=0.5] at (2,-5,0) {};
    \node[fill=black,circle,inner sep=1pt,opacity=0.5] at (2,-5,2) {};
    \node[fill=black,circle,inner sep=1pt,opacity=0.5] at (2,-5,-2) {};
    \node[fill=black,circle,inner sep=1pt,opacity=0.5] at (-2,-5,0) {};
    \node[fill=black,circle,inner sep=1pt,opacity=0.5] at (-2,-5,2) {};
    \node[fill=black,circle,inner sep=1pt,opacity=0.5] at (-2,-5,-2) {};
    \node[fill=black,circle,inner sep=1pt,opacity=0.5] at (0,5,0) {};
    \node[fill=black,circle,inner sep=1pt,opacity=0.5] at (0,5,2) {};
    \node[fill=black,circle,inner sep=1pt,opacity=0.5] at (0,5,-2) {};
    \node[fill=black,circle,inner sep=1pt,opacity=0.5] at (2,5,0) {};
    \node[fill=black,circle,inner sep=1pt,opacity=0.5] at (2,5,2) {};
    \node[fill=black,circle,inner sep=1pt,opacity=0.5] at (2,5,-2) {};
    \node[fill=black,circle,inner sep=1pt,opacity=0.5] at (-2,5,0) {};
    \node[fill=black,circle,inner sep=1pt,opacity=0.5] at (-2,5,2) {};
    \node[fill=black,circle,inner sep=1pt,opacity=0.5] at (-2,5,-2) {};
       
    \coordinate (z_1) at (-2,0,0);
    \coordinate (z_3) at (0,-5,2);
    \coordinate (z_2) at (0,5,0);
    \coordinate (z_4) at (2,0,-2);
    
    \coordinate (p_pl_1) at ($ (z_1)!1/2!(z_3) $);
    \coordinate (p_mi_1) at ($ (z_2)!1/2!(z_4) $);
    \coordinate (z_1_z_2_half) at ($ (z_1)!1/2!(z_2) $);
    \coordinate (z_2_z_3_half) at ($ (z_2)!1/2!(z_3) $);
    \coordinate (z_3_z_4_half) at ($ (z_3)!1/2!(z_4) $);
    \coordinate (z_4_z_1_half) at ($ (z_4)!1/2!(z_1) $); 
    
    \coordinate (p) at ($ (p_pl_1)!1/2!(p_mi_1) $);
    \coordinate (z_1_pl_1) at ($(p_mi_1) + (z_1) - (p_pl_1) $);
    \coordinate (z_3_pl_1) at ($(p_mi_1) + (z_3) - (p_pl_1) $);
    \coordinate (z_2_mi_1) at ($(p_pl_1) + (z_2) - (p_mi_1) $);
    \coordinate (z_4_mi_1) at ($(p_pl_1) + (z_4) - (p_mi_1) $);
    
    \coordinate (z_1_p) at ($ (z_1)!1/2!(z_1_pl_1) $);
    \coordinate (z_3_p) at ($(z_3)!1/2!(z_3_pl_1)$);
    \coordinate (z_2_p) at ($(z_2)!1/2!(z_2_mi_1)$);
    \coordinate (z_4_p) at ($(z_4)!1/2!(z_4_mi_1)$);
        
    \fill[blue!05,opacity=0.3] (3,0,3) -- (3,0,-3) -- (-3,0,-3) -- (-3,0,3) -- cycle;
    \fill[blue!05,opacity=0.3] (3,-5,3) -- (3,-5,-3) -- (-3,-5,-3) -- (-3,-5,3) -- cycle;   
    \fill[blue!05,opacity=0.3] (3,5,3) -- (3,5,-3) -- (-3,5,-3) -- (-3,5,3) -- cycle;  
    \fill[blue!25,opacity=0.4] (-2,5,-2) -- (-2,5,2) -- (2,-5,2) -- (2,-5,-2) -- cycle;
    
    \fill[blue!25,opacity=0.7] (z_1) -- (z_2_mi_1) -- (z_3) -- (z_4_mi_1) -- cycle;  
    \fill[blue!25,opacity=0.7] (z_1_pl_1) -- (z_2) -- (z_3_pl_1) -- (z_4) -- cycle;     
    \fill[blue!25,opacity=0.7] (z_1_p) -- (z_2_p) -- (z_3_p) -- (z_4_p) -- cycle;     
    
    \node[fill=blue,circle,inner sep=1pt] at (z_1) {};
    \draw (z_1) node[left] {$\vec z_1$};
    
    \node[fill=blue,circle,inner sep=1pt] at (z_3) {};
    \draw (z_3) node[below] {$\vec z_3$};
    
    \node[fill=blue,circle,inner sep=1pt] at (z_2) {};
    \draw (z_2) node[above] {$\vec z_2$};
    
    \node[fill=blue,circle,inner sep=1pt] at (z_4) {};
    \draw (z_4) node[right] {$\vec z_4$};
    
    \node[fill=blue,circle,inner sep=1pt] at (z_1_pl_1) {};
    \draw (z_1_pl_1) node[right] {$\vec z_1^{R}$};
    
    \node[fill=blue,circle,inner sep=1pt] at (z_3_pl_1) {};
    \draw (z_3_pl_1) node[left] {$\vec z_3^{R}$};
    
    \node[fill=blue,circle,inner sep=1pt] at (z_2_mi_1) {};
    \draw (z_2_mi_1) node[left] {$\vec z_2^{Q}$};
    
    \node[fill=blue,circle,inner sep=1pt] at (z_4_mi_1) {};
    \draw (z_4_mi_1) node[right] {$\vec z_4^{Q}$};
    
    \node[fill=white,circle,inner sep=1pt,draw=blue] at (z_1_p) {};
    \draw (z_1_p) node[below] {$\vec z_1^{\pi}$};
    
    \node[fill=white,circle,inner sep=1pt,draw=blue] at (z_3_p) {};
    \draw (z_3_p) node[left] {$\vec z_3^{\pi}$};
    
    \node[fill=white,circle,inner sep=1pt,draw=blue] at (z_2_p) {};
    \draw (z_2_p) node[below left] {$\vec z_2^{\pi}$};
    
    \node[fill=white,circle,inner sep=1pt,draw=blue] at (z_4_p) {};
    \draw (z_4_p) node[right] {$\vec z_4^{\pi}$};
     
    \node[fill=blue,circle,inner sep=1pt] at (p) {};
    \draw (p) node[right] {$\vec p$};
    
    \node[fill=white,circle,inner sep=1pt,draw=blue] at (p_pl_1) {};
    \draw (p_pl_1) node[right] {$\vec{p}^{Q}$};
    
    \node[fill=white,circle,inner sep=1pt,draw=blue] at (p_mi_1) {};
    \draw (p_mi_1) node[above right] {$\vec{p}^{R}$};
    
    \node[fill=white,circle,inner sep=1pt,draw=blue] at (z_1_z_2_half) {};
    \draw (z_1_z_2_half) node[left] {$\frac{\vec{z}_1^{\pi} + \vec{z}_2^{\pi}}{2}$};
    
    \node[fill=white,circle,inner sep=1pt,draw=blue] at (z_2_z_3_half) {};
    \draw (z_2_z_3_half) node[below right] {$\frac{\vec{z}_2^{\pi} + \vec{z}_3^{\pi}}{2}$};
    
    \node[fill=white,circle,inner sep=1pt,draw=blue] at (z_3_z_4_half) {};
    \draw (z_3_z_4_half) node[below right] {$\frac{\vec{z}_3^{\pi} + \vec{z}_4^{\pi}}{2}$};
    
    \node[fill=white,circle,inner sep=1pt,draw=blue] at (z_4_z_1_half) {};
    \draw (z_4_z_1_half) node[above] {$\frac{\vec{z}_4^{\pi} + \vec{z}_1^{\pi}}{2}$};
    
    \draw[->] [red,very thin] (p) -- (z_1);
    \draw[->] [red,very thin] (p) -- (z_2);
    \draw[->] [red,very thin] (p) -- (z_3);
    
    \draw (5, 0, 0) node[right] {$x$};
    \draw (0, 7, 0) node[left] {$z$};
    \draw (0, 0, 7) node[left] {$y$};

    \end{scope}
  \end{tikzpicture}
  \caption{Расположение точек внутри гиперплоскости $S_{1}$ из случая \textbf{Б.1.а.2} леммы \ref{main_lem_2_2}}
\end{figure}

\textbf{Б.1.а.3}) (будет соответствовать утверждению (11)) Пусть $\vec{w} \in \{\frac{\vec{z}^{\pi}_{1} + \vec{z}^{\pi}_{2}}{2}, \frac{\vec{z}^{\pi}_{3} + \vec{z}^{\pi}_{4}}{2} \}$. В этом случае каждая из точек множества $\{\frac{\vec{z}^{\pi}_{1} + \vec{z}^{\pi}_{2}}{2}, \frac{\vec{z}^{\pi}_{3} + \vec{z}^{\pi}_{4}}{2} \}$ принадлежит решетке $\Z^{4}$. В силу доказательства случаев \textbf{Б.1.а.1} и \textbf{Б.1.а.2} никакая из точек множества $\{\vec{z}^{\pi}_{1}, \vec{z}^{\pi}_{2}, \vec{z}^{\pi}_{3}, \vec{z}^{\pi}_{4}, \vec{p} \}$ не принадлежит решетке $\Z^{4}$. Кроме того, так как $\Delta^{Q}_{k} \cap \Z^{4} = \{\vec{z}_{1}, \vec{z}_{3}, \vec{z}^{Q}_{2}, \vec{z}^{Q}_{4}\}$, то никакая из точек множества $\{\frac{\vec{z}^{\pi}_{2} + \vec{z}^{\pi}_{3}}{2}, \frac{\vec{z}^{\pi}_{4} + \vec{z}^{\pi}_{1}}{2} \}$ не принадлежит решетке $\Z^{4}$. Тогда
\[(\Delta^{\pi}_{k} \cup \Delta^{Q}_{k} \cup \Delta^{R}_{k}) \cap \Z^{4} = \{ \vec{z}_{1}, \, \vec{z}^{Q}_{2}, \, \vec{z}_{3}, \, \vec{z}^{Q}_{4}, \, \vec{z}^{R}_{1}, \, \vec{z}_{2}, \, \vec{z}^{R}_{3}, \, \vec{z}_{4}, \, \frac{\vec{z}^{\pi}_{1} + \vec{z}^{\pi}_{2}}{2}, \, \frac{\vec{z}^{\pi}_{3} + \vec{z}^{\pi}_{4}}{2} \}\]
и, так как плоскости $Q$ и $R$ равноудалены от $\pi$, набор векторов $\vec{z}_{1}$, $\frac{1}{2}(\vec{z}^{\pi}_{1} + \vec{z}^{\pi}_{2}) = \frac{1}{2}(\vec{z}_{1} + \vec{z}_{2})$, $\vec{z}_{3}$, $\vec{z}^{Q}_{4} = \frac{1}{2}(\vec{z}_{1} + \vec{z}_{3} + \vec{z}_{4} - \vec{z}_{2})$ образует базис решетки $\Z^{4}$, а значит, выполняется утверждение (11).

\begin{figure}[h]
  \centering
  \begin{tikzpicture}[x=10mm, y=7mm, z=-5mm, scale=0.8]
    \begin{scope}[rotate around x=0]
  
    \draw[->] [very thin] (-5,0,0) -- (5,0,0);
    \draw[->] [very thin] (0,-7,0) -- (0,11,0);
    \draw[->] [very thin] (0,0,-7) -- (0,0,7);
    
    \node[fill=black,circle,inner sep=1pt,opacity=0.5] at (0,0,0) {};
    \node[fill=black,circle,inner sep=1pt,opacity=0.5] at (0,0,2) {};
    \node[fill=black,circle,inner sep=1pt,opacity=0.5] at (0,0,-2) {};
    \node[fill=black,circle,inner sep=1pt,opacity=0.5] at (2,0,0) {};
    \node[fill=black,circle,inner sep=1pt,opacity=0.5] at (2,0,-2) {};
    \node[fill=black,circle,inner sep=1pt,opacity=0.5] at (-2,0,0) {};
    \node[fill=black,circle,inner sep=1pt,opacity=0.5] at (-2,0,2) {}; 
    \node[fill=black,circle,inner sep=1pt,opacity=0.5] at (0,-5,0) {};
    \node[fill=black,circle,inner sep=1pt,opacity=0.5] at (0,-5,2) {};
    \node[fill=black,circle,inner sep=1pt,opacity=0.5] at (0,-5,-2) {};
    \node[fill=black,circle,inner sep=1pt,opacity=0.5] at (2,-5,0) {};
    \node[fill=black,circle,inner sep=1pt,opacity=0.5] at (2,-5,2) {};
    \node[fill=black,circle,inner sep=1pt,opacity=0.5] at (2,-5,-2) {};
    \node[fill=black,circle,inner sep=1pt,opacity=0.5] at (-2,-5,0) {};
    \node[fill=black,circle,inner sep=1pt,opacity=0.5] at (-2,-5,2) {};
    \node[fill=black,circle,inner sep=1pt,opacity=0.5] at (-2,-5,-2) {};
    \node[fill=black,circle,inner sep=1pt,opacity=0.5] at (0,5,0) {};
    \node[fill=black,circle,inner sep=1pt,opacity=0.5] at (0,5,2) {};
    \node[fill=black,circle,inner sep=1pt,opacity=0.5] at (0,5,-2) {};
    \node[fill=black,circle,inner sep=1pt,opacity=0.5] at (2,5,0) {};
    \node[fill=black,circle,inner sep=1pt,opacity=0.5] at (2,5,2) {};
    \node[fill=black,circle,inner sep=1pt,opacity=0.5] at (2,5,-2) {};
    \node[fill=black,circle,inner sep=1pt,opacity=0.5] at (-2,5,0) {};
    \node[fill=black,circle,inner sep=1pt,opacity=0.5] at (-2,5,2) {};
    \node[fill=black,circle,inner sep=1pt,opacity=0.5] at (-2,5,-2) {};
    \node[fill=black,circle,inner sep=1pt,opacity=0.5] at (0,10,0) {};
    \node[fill=black,circle,inner sep=1pt,opacity=0.5] at (0,10,2) {};
    \node[fill=black,circle,inner sep=1pt,opacity=0.5] at (0,10,-2) {};
    \node[fill=black,circle,inner sep=1pt,opacity=0.5] at (2,10,0) {};
    \node[fill=black,circle,inner sep=1pt,opacity=0.5] at (2,10,2) {};
    \node[fill=black,circle,inner sep=1pt,opacity=0.5] at (2,10,-2) {};
    \node[fill=black,circle,inner sep=1pt,opacity=0.5] at (-2,10,0) {};
    \node[fill=black,circle,inner sep=1pt,opacity=0.5] at (-2,10,2) {};
    \node[fill=black,circle,inner sep=1pt,opacity=0.5] at (-2,10,-2) {};    
    
    \coordinate (z_1) at (0,0,0);
    \coordinate (z_3) at (0,0,2);
    \coordinate (z_2) at (0,10,0);
    \coordinate (z_4) at (4,0,-2);
    
    \coordinate (p_pl_1) at ($ (z_1)!1/2!(z_3) $);
    \coordinate (p_mi_1) at ($ (z_2)!1/2!(z_4) $);
    \coordinate (z_1_z_2_half) at ($ (z_1)!1/2!(z_2) $);
    \coordinate (z_2_z_3_half) at ($ (z_2)!1/2!(z_3) $);
    \coordinate (z_3_z_4_half) at ($ (z_3)!1/2!(z_4) $);
    \coordinate (z_4_z_1_half) at ($ (z_4)!1/2!(z_1) $); 
    
    \coordinate (p) at ($ (p_pl_1)!1/2!(p_mi_1) $);
    \coordinate (z_1_pl_1) at ($(p_mi_1) + (z_1) - (p_pl_1) $);
    \coordinate (z_3_pl_1) at ($(p_mi_1) + (z_3) - (p_pl_1) $);
    \coordinate (z_2_mi_1) at ($(p_pl_1) + (z_2) - (p_mi_1) $);
    \coordinate (z_4_mi_1) at ($(p_pl_1) + (z_4) - (p_mi_1) $);
    
    \coordinate (z_1_p) at ($ (z_1)!1/2!(z_1_pl_1) $);
    \coordinate (z_3_p) at ($(z_3)!1/2!(z_3_pl_1)$);
    \coordinate (z_2_p) at ($(z_2)!1/2!(z_2_mi_1)$);
    \coordinate (z_4_p) at ($(z_4)!1/2!(z_4_mi_1)$);
    
    \fill[blue!05,opacity=0.3] (3,0,3) -- (3,0,-3) -- (-3,0,-3) -- (-3,0,3) -- cycle;
    \fill[blue!05,opacity=0.3] (3,-5,3) -- (3,-5,-3) -- (-3,-5,-3) -- (-3,-5,3) -- cycle;   
    \fill[blue!05,opacity=0.3] (3,5,3) -- (3,5,-3) -- (-3,5,-3) -- (-3,5,3) -- cycle;
    \fill[blue!05,opacity=0.3] (3,10,3) -- (3,10,-3) -- (-3,10,-3) -- (-3,10,3) -- cycle;

    \fill[blue!25,opacity=0.6] (z_1) -- (z_2_mi_1) -- (z_3) -- (z_4_mi_1) -- cycle; 
    \fill[blue!25,opacity=0.6] (z_1_p) -- (z_2_p) -- (z_3_p) -- (z_4_p) -- cycle; 
    \fill[blue!25,opacity=0.6] (z_1_pl_1) -- (z_2) -- (z_3_pl_1) -- (z_4) -- cycle;
      
    \node[fill=blue,circle,inner sep=1pt] at (z_1) {};
    \draw (z_1) node[below right] {$\vec z_1$};
    
    \node[fill=blue,circle,inner sep=1pt] at (z_3) {};
    \draw (z_3) node[left] {$\vec z_3$};
    
    \node[fill=blue,circle,inner sep=1pt] at (z_2) {};
    \draw (z_2) node[above] {$\vec z_2$};
    
    \node[fill=blue,circle,inner sep=1pt] at (z_4) {};
    \draw (z_4) node[above right] {$\vec z_4$};
    
    \node[fill=blue,circle,inner sep=1pt] at (z_1_pl_1) {};
    \draw (z_1_pl_1) node[right] {$\vec z_1^{R}$};
    
    \node[fill=blue,circle,inner sep=1pt] at (z_3_pl_1) {};
    \draw (z_3_pl_1) node[below right] {$\vec z_3^{R}$};
    
    \node[fill=blue,circle,inner sep=1pt] at (z_2_mi_1) {};
    \draw (z_2_mi_1) node[left] {$\vec z_2^{Q}$};
    
    \node[fill=blue,circle,inner sep=1pt] at (z_4_mi_1) {};
    \draw (z_4_mi_1) node[right] {$\vec z_4^{Q}$};
    
    \node[fill=white,circle,inner sep=1pt,draw=blue] at (z_1_p) {};
    \draw (z_1_p) node[below right] {$\vec z_1^{\pi}$};
    
    \node[fill=white,circle,inner sep=1pt,draw=blue] at (z_3_p) {};
    \draw (z_3_p) node[left] {$\vec z_3^{\pi}$};
    
    \node[fill=white,circle,inner sep=1pt,draw=blue] at (z_2_p) {};
    \draw (z_2_p) node[below left] {$\vec z_2^{\pi}$};
    
    \node[fill=white,circle,inner sep=1pt,draw=blue] at (z_4_p) {};
    \draw (z_4_p) node[above right] {$\vec z_4^{\pi}$};
    
    \node[fill=white,circle,inner sep=1pt,draw=blue] at (p) {};
    \draw (p) node[right] {$\vec p$};
    
    \node[fill=white,circle,inner sep=1pt,draw=blue] at (p_pl_1) {};
    \draw (p_pl_1) node[right] {$\vec{p}^{Q}$};
    
    \node[fill=white,circle,inner sep=1pt,draw=blue] at (p_mi_1) {};
    \draw (p_mi_1) node[above left] {$\vec{p}^{R}$};
    
    \node[fill=blue,circle,inner sep=1pt] at (z_1_z_2_half) {};
    \draw (z_1_z_2_half) node[right] {$\frac{\vec{z}_1^{\pi} + \vec{z}_2^{\pi}}{2}$};
    
    \node[fill=white,circle,inner sep=1pt,draw=blue] at (z_2_z_3_half) {};
    \draw (z_2_z_3_half) node[left] {$\frac{\vec{z}_2^{\pi} + \vec{z}_3^{\pi}}{2}$};
    
    \node[fill=blue,circle,inner sep=1pt] at (z_3_z_4_half) {};
    \draw (z_3_z_4_half) node[below] {$\frac{\vec{z}_3^{\pi} + \vec{z}_4^{\pi}}{2}$};
    
    \node[fill=white,circle,inner sep=1pt,draw=blue] at (z_4_z_1_half) {};
    \draw (z_4_z_1_half) node[above right] {$\frac{\vec{z}_4^{\pi} + \vec{z}_1^{\pi}}{2}$};
    
    \draw[->] [red,very thin] (z_1) -- (z_1_z_2_half);
    \draw[->] [red,very thin] (z_1) -- (z_3);
    \draw[->] [red,very thin] (z_1) -- (z_4_mi_1);
    
    \draw (5, 0, 0) node[right] {$x$};
    \draw (0, 11, 0) node[left] {$z$};
    \draw (0, 0, 7) node[left] {$y$};

    \end{scope}
  \end{tikzpicture}
  \caption{Расположение точек внутри гиперплоскости $S_{1}$ из случая \textbf{Б.1.а.3} леммы \ref{main_lem_2_2}}
\end{figure}

\textbf{Б.1.б}) (будет соответствовать утверждению (6)) Плоскость $\pi$ не является рациональной плоскостью. Тогда 
\[(\Delta^{\pi}_{k} \cup \Delta^{Q}_{k} \cup \Delta^{R}_{k}) \cap \Z^{4} = \{\vec{z}_{1}, \, \vec{z}^{Q}_{2}, \, \vec{z}_{3}, \, \vec{z}^{Q}_{4}, \, \vec{z}^{R}_{1}, \, \vec{z}_{2}, \, \vec{z}^{R}_{3}, \, \vec{z}_{4}\},\]
набор векторов $\vec{z}_{1}$, $\vec{z}_{2}$, $\vec{z}_{3}$, $\vec{z}^{Q}_{4} = \frac{1}{2}(\vec{z}_{1} + \vec{z}_{3} + \vec{z}_{4} - \vec{z}_{2})$ образует базис решетки $\Z^{4}$, а значит, выполняется утверждение (6). 

 \begin{figure}[h]
  \centering
  \begin{tikzpicture}[x=10mm, y=7mm, z=-5mm]
    \begin{scope}[rotate around x=0]
    
    \draw[->] [very thin] (-5,0,0) -- (5,0,0);
    \draw[->] [very thin] (0,-3,0) -- (0,7,0);
    \draw[->] [very thin] (0,0,-4) -- (0,0,4);
    
    \node[fill=black,circle,inner sep=1pt,opacity=0.5] at (0,0,0) {};
    \node[fill=black,circle,inner sep=1pt,opacity=0.5] at (0,0,2) {};
    \node[fill=black,circle,inner sep=1pt,opacity=0.5] at (0,0,-2) {};
    \node[fill=black,circle,inner sep=1pt,opacity=0.5] at (2,0,0) {};
    \node[fill=black,circle,inner sep=1pt,opacity=0.5] at (2,0,2) {};
    \node[fill=black,circle,inner sep=1pt,opacity=0.5] at (2,0,-2) {};
    \node[fill=black,circle,inner sep=1pt,opacity=0.5] at (-2,0,0) {};
    \node[fill=black,circle,inner sep=1pt,opacity=0.5] at (-2,0,2) {};
    \node[fill=black,circle,inner sep=1pt,opacity=0.5] at (-2,0,-2) {};    
    \node[fill=black,circle,inner sep=1pt,opacity=0.5] at (0,5,0) {};
    \node[fill=black,circle,inner sep=1pt,opacity=0.5] at (0,5,2) {};
    \node[fill=black,circle,inner sep=1pt,opacity=0.5] at (0,5,-2) {};
    \node[fill=black,circle,inner sep=1pt,opacity=0.5] at (2,5,0) {};
    \node[fill=black,circle,inner sep=1pt,opacity=0.5] at (2,5,2) {};
    \node[fill=black,circle,inner sep=1pt,opacity=0.5] at (2,5,-2) {};
    \node[fill=black,circle,inner sep=1pt,opacity=0.5] at (-2,5,0) {};
    \node[fill=black,circle,inner sep=1pt,opacity=0.5] at (-2,5,2) {};
    \node[fill=black,circle,inner sep=1pt,opacity=0.5] at (-2,5,-2) {};
    
    \coordinate (z_1) at (2,5,0);
    \coordinate (z_3) at (0,5,-2);
    \coordinate (z_2) at (0,0,0);
    \coordinate (z_4) at (-2,0,2);
    
    \coordinate (p_pl_1) at ($ (z_1)!1/2!(z_3) $);
    \coordinate (p_mi_1) at ($ (z_2)!1/2!(z_4) $);
    \coordinate (z_1_z_2_half) at ($ (z_1)!1/2!(z_2) $);
    \coordinate (z_2_z_3_half) at ($ (z_2)!1/2!(z_3) $);
    \coordinate (z_3_z_4_half) at ($ (z_3)!1/2!(z_4) $);
    \coordinate (z_4_z_1_half) at ($ (z_4)!1/2!(z_1) $); 
    
    \coordinate (p) at ($ (p_pl_1)!1/2!(p_mi_1) $);
    \coordinate (z_1_pl_1) at ($(p_mi_1) + (z_1) - (p_pl_1) $);
    \coordinate (z_3_pl_1) at ($(p_mi_1) + (z_3) - (p_pl_1) $);
    \coordinate (z_2_mi_1) at ($(p_pl_1) + (z_2) - (p_mi_1) $);
    \coordinate (z_4_mi_1) at ($(p_pl_1) + (z_4) - (p_mi_1) $);
    
    \coordinate (z_1_p) at ($ (z_1)!1/2!(z_1_pl_1) $);
    \coordinate (z_3_p) at ($(z_3)!1/2!(z_3_pl_1)$);
    \coordinate (z_2_p) at ($(z_2)!1/2!(z_2_mi_1)$);
    \coordinate (z_4_p) at ($(z_4)!1/2!(z_4_mi_1)$);
    
    \fill[blue!05,opacity=0.3] (3,0,3) -- (3,0,-3) -- (-3,0,-3) -- (-3,0,3) -- cycle;
    \fill[blue!05,opacity=0.3] (3,5,3) -- (3,5,-3) -- (-3,5,-3) -- (-3,5,3) -- cycle;
    
    \fill[blue!25,opacity=0.6] (z_1) -- (z_2_mi_1) -- (z_3) -- (z_4_mi_1) -- cycle;
    \fill[blue!25,opacity=0.6] (z_1_pl_1) -- (z_2) -- (z_3_pl_1) -- (z_4) -- cycle;
    
    \node[fill=blue,circle,inner sep=1pt] at (z_1) {};
    \draw (z_1) node[below right] {$\vec z_1$};
    
    \node[fill=blue,circle,inner sep=1pt] at (z_3) {};
    \draw (z_3) node[above] {$\vec z_3$};
    
    \node[fill=blue,circle,inner sep=1pt] at (z_2) {};
    \draw (z_2) node[below right] {$\vec z_2$};
    
    \node[fill=blue,circle,inner sep=1pt] at (z_4) {};
    \draw (z_4) node[left] {$\vec z_4$};
    
    \node[fill=blue,circle,inner sep=1pt] at (z_1_pl_1) {};
    \draw (z_1_pl_1) node[right] {$\vec z_1^{R}$};
    
    \node[fill=blue,circle,inner sep=1pt] at (z_3_pl_1) {};
    \draw (z_3_pl_1) node[above right] {$\vec z_3^{R}$};
    
    \node[fill=blue,circle,inner sep=1pt] at (z_2_mi_1) {};
    \draw (z_2_mi_1) node[above right] {$\vec z_2^{Q}$};
    
    \node[fill=blue,circle,inner sep=1pt] at (z_4_mi_1) {};
    \draw (z_4_mi_1) node[left] {$\vec z_4^{Q}$};
    
    \node[fill=white,circle,inner sep=1pt,draw=blue] at (p) {};
    \draw (p) node[right] {$\vec p$};
    
    \node[fill=white,circle,inner sep=1pt,draw=blue] at (p_pl_1) {};
    \draw (p_pl_1) node[above right] {$\vec{p}^{Q}$};
    
    \node[fill=white,circle,inner sep=1pt,draw=blue] at (p_mi_1) {};
    \draw (p_mi_1) node[right] {$\vec{p}^{R}$};
    
    \draw[->] [red,very thin] (z_1) -- (z_2);
    \draw[->] [red,very thin] (z_1) -- (z_3);
    \draw[->] [red,very thin] (z_1) -- (z_2_mi_1);
    
    \draw (5, 0, 0) node[right] {$x$};
    \draw (0, 7, 0) node[left] {$z$};
    \draw (0, 0, 4) node[left] {$y$};
    
    \end{scope}
  \end{tikzpicture}
  \caption{Расположение точек внутри гиперплоскости $S_{1}$ из случая \textbf{Б.1.б} леммы \ref{main_lem_2_2}}
\end{figure}

\textbf{В}) $\Delta^{Q}_{k} \cap \Z^{4} = \{\vec{z}_{1}, \vec{z}_{3}, \vec{p}^{Q}\}$. Заметим, что случай $\Delta^{R}_{k} \cap \Z^{4} = \{\vec{z}_{2}, \vec{z}_{4}\}$ с точностью до перестановки индексов полностью эквивалентен пункту \textbf{А.2}. Кроме того, случай $\Delta^{R}_{k} \cap \Z^{4} = \{\vec{z}_{2}, \vec{z}_{4}, \vec{z}^{R}_{1}, \vec{z}^{R}_{3}\}$ с точностью до перестановки индексов невозможен в силу пункта \textbf{Б}. Если $\vec{p}^{R} \in \Z^{4}$, то  $\vec{z}^{Q}_{2} = \vec{p}^{Q} + (\vec{z}_{2} - \vec{p}^{R}) \in \Z^{4}$, что противоречит рассматриваемому случаю. Таким образом, существует ровно один подслучай:

\textbf{В.1}) $\Delta^{R}_{k} \cap \Z^{4} = \{\vec{z}_{2}, \vec{z}_{4}, \vec{p}^{R}_{1},\vec{p}^{R}_{3}\}$. Этот случай, в свою очередь, разбивается на два подслучая:

\textbf{В.1.а}) Плоскость $\pi$ рациональная. Тогда, плоскости $Q$ и $R$ равноудалены от $\pi$. Покажем, что параллелограмм $\Delta^{\pi}_{k}$ не содержит точек решетки $\Z^{4}$. Пусть это не так. Можно считать, что существует точка $\vec{w} \in \Z^{4}$, лежащая либо в параллелограмме $\textup{conv}(\vec{z}^{\pi}_{1}, \frac{\vec{z}^{\pi}_{1} + \vec{z}^{\pi}_{2}}{2}, \vec{p}, \frac{\vec{z}^{\pi}_{4} + \vec{z}^{\pi}_{1}}{2})$, либо в параллелограмме $\textup{conv}(\vec{z}^{\pi}_{2}, \frac{\vec{z}^{\pi}_{2} + \vec{z}^{\pi}_{3}}{2}, \vec{p}, \frac{\vec{z}^{\pi}_{1} + \vec{z}^{\pi}_{2}}{2})$. Рассмотрим случай $\vec{w} \in \textup{conv}(\vec{z}^{\pi}_{1}, \frac{\vec{z}^{\pi}_{1} + \vec{z}^{\pi}_{2}}{2}, \vec{p}, \frac{\vec{z}^{\pi}_{4} + \vec{z}^{\pi}_{1}}{2})$. Если $\vec{w} \notin \{\frac{\vec{z}^{\pi}_{1} + \vec{z}^{\pi}_{2}}{2},  \frac{\vec{z}^{\pi}_{4} + \vec{z}^{\pi}_{1}}{2}, \frac{3\vec{z}^{\pi}_{1} +  \vec{p}}{4}, \frac{\vec{z}^{\pi}_{1} +  3\vec{p}}{4}\}$, то $\vec{w} + (\vec{w} - \vec{z}_1) \in \Delta^{R}_{k} \cap \Z^{4}$ и $\vec{w} + (\vec{w} - \vec{z}_1) \notin \{\vec{z}_{2}, \vec{z}_{4}, \vec{p}^{R}_{1},\vec{p}^{R}_{3}\}$, чего не может быть. Если $\vec{w} =  \frac{\vec{z}^{\pi}_{1} + \vec{z}^{\pi}_{2}}{2}$, то $\frac{\vec{z}^{\pi}_{3} + \vec{z}^{\pi}_{4}}{2} = F(\frac{\vec{z}^{\pi}_{1} + \vec{z}^{\pi}_{2}}{2}) \in \Z^{4}$, а значит, $\vec{z}^{Q}_{4} = \vec{z}_1 + (\frac{\vec{z}^{\pi}_{3} + \vec{z}^{\pi}_{4}}{2} - \frac{\vec{z}^{\pi}_{1} + \vec{z}^{\pi}_{2}}{2}) \in \Z^{4}$, чего также не может быть. Аналогично показывается, что случай $\vec{w}  = \frac{\vec{z}^{\pi}_{1} + \vec{z}^{\pi}_{4}}{2}$ невозможен. Если $\vec{w} =  \frac{3\vec{z}^{\pi}_{1} +  \vec{p}}{4}$, то $\frac{\vec{z}^{\pi}_{3} +  3\vec{p}}{4} =  \frac{3\vec{z}^{\pi}_{1} +  \vec{p}}{4} + (\vec{p}^{R}_{3} - \vec{p}^{R}_{3}) \in \Z^{4}$, а значит, $\frac{\vec{z}^{\pi}_{1} +  3\vec{p}}{4} = F(\vec{w}) \in \Z^{4}$. Тогда $\vec{p}^{R} = \vec{p}^{R}_{1} + (\frac{\vec{z}^{\pi}_{3} +  3\vec{p}}{4} - \frac{\vec{z}^{\pi}_{1} +  3\vec{p}}{4}) \in \Z^{4}$, чего не может быть. Если $\vec{w} =  \frac{\vec{z}^{\pi}_{1} +  3\vec{p}}{4}$, то $\frac{\vec{z}^{\pi}_{3} +  3\vec{p}}{4} = F(\vec{w}) \in \Z^{4}$, и вновь получаем $\vec{p}^{R} = \vec{p}^{R}_{1} + (\frac{\vec{z}^{\pi}_{3} +  3\vec{p}}{4} - \frac{\vec{z}^{\pi}_{1} +  3\vec{p}}{4}) \in \Z^{4}$, чего не может быть.
Теперь рассмотрим случай $\vec{w} \in \textup{conv}(\vec{z}^{\pi}_{2}, \frac{\vec{z}^{\pi}_{2} + \vec{z}^{\pi}_{3}}{2}, \vec{p}, \frac{\vec{z}^{\pi}_{1} + \vec{z}^{\pi}_{2}}{2})$. Если $\vec{w} \notin \{\frac{\vec{z}^{\pi}_{1} + \vec{z}^{\pi}_{2}}{2}, \frac{\vec{z}^{\pi}_{2} + \vec{z}^{\pi}_{3}}{2}, \frac{\vec{z}^{\pi}_{2} +  \vec{p}}{2}\}$, то $\vec{w} + (\vec{w} - \vec{z}_2) \in \Delta^{Q}_{k} \cap \Z^{4}$ и $\vec{w} + (\vec{w} - \vec{z}_2) \notin \{\vec{z}_1, \vec{z}_3, \vec{p}^{Q}\}$, чего не может быть. По доказанному выше, случаи $\vec{w} =  \frac{\vec{z}^{\pi}_{1} + \vec{z}^{\pi}_{2}}{2}$ и $\vec{w} =  \frac{\vec{z}^{\pi}_{2} + \vec{z}^{\pi}_{3}}{2}$ невозможны. Если $\vec{w} = \frac{\vec{z}^{\pi}_{2} +  \vec{p}}{2}$, то $\frac{\vec{z}^{\pi}_{4} +  \vec{p}}{2} = F(\vec{w}) \in \Z^{4}$, а значит, $\vec{p}^{R} = \vec{z}_{2} + (\frac{\vec{z}^{\pi}_{4} +  \vec{p}}{2} - \frac{\vec{z}^{\pi}_{2} +  \vec{p}}{2}) \in \Z^{4}$, чего не может быть. И так, параллелограмм $\Delta^{\pi}_{k}$ не содержит точек решетки $\Z^{4}$. С другой стороны, параллелограмм $\textup{conv}(\vec{z}^{\pi}_{2}, \frac{\vec{z}^{\pi}_{1} +  \vec{p}}{2}, \vec{z}^{\pi}_{4}, \frac{\vec{z}^{\pi}_{1} +  \vec{p}}{2}) \subset \Delta^{\pi}_{k}$ должен содержать хотя бы одну точку решетки $\Z^{4}$, поскольку $\textup{conv}(\vec{z}_{2}, \vec{p}^{R}_{1},  \vec{z}_{4}, \vec{p}^{R}_{3})  \cap \Z^{4} = \{\vec{z}_{2}, \vec{p}^{R}_{1},  \vec{z}_{4}, \vec{p}^{R}_{3})\}$. Таким образом, данный случай невозможен.

\textbf{В.1.б}) (будет соответствовать утверждению (9)) Плоскость $\pi$ не является рациональной плоскостью. Тогда 
\[(\Delta^{\pi}_{k} \cup \Delta^{Q}_{k} \cup \Delta^{R}_{k}) \cap \Z^{4} = \{\vec{z}_{1}, \, \vec{z}_{3}, \, \vec{p}^{Q}, \, \vec{z}_{2},  \, \vec{z}_{4}, \,  \vec{p}^{R}_{1}, \, \vec{p}^{R}_{3}\},\]
набор векторов $\vec{z}_{1}$, $\vec{z}_{2}$, $\vec{p}^{Q} = \frac{1}{2}(\vec{z}_{1} + \vec{z}_{3})$, $\vec{p}^{R}_{1} = \vec{p}^{R}  + (\frac{1}{2}(\vec{p}^{Q} + \vec{z}_{3}) - \vec{z}_{3})= \frac{1}{2}(\vec{z}_{2} + \vec{z}_{4}) + (\frac{1}{2}(\frac{1}{2}(\vec{z}_{1} + \vec{z}_{3}) +  \vec{z}_{3}) - \vec{z}_{3}) =  \frac{1}{4}\vec{z}_{1} +  \frac{1}{2}\vec{z}_{2} - \frac{1}{4}\vec{z}_{3} + \frac{1}{2}\vec{z}_{4}$ образует базис решетки $\Z^{4}$, а значит, выполняется утверждение (9). 

 \begin{figure}[h]
  \centering
  \begin{tikzpicture}[x=10mm, y=7mm, z=-5mm]
    \begin{scope}[rotate around x=0]
  
    \draw[->] [very thin] (-5,0,0) -- (5,0,0);
    \draw[->] [very thin] (0,-3,0) -- (0,7,0);
    \draw[->] [very thin] (0,0,-4) -- (0,0,5);
    
    \node[fill=black,circle,inner sep=1pt,opacity=0.5] at (0,0,0) {};
    \node[fill=black,circle,inner sep=1pt,opacity=0.5] at (0,0,2) {};
    \node[fill=black,circle,inner sep=1pt,opacity=0.5] at (0,0,-2) {};
    \node[fill=black,circle,inner sep=1pt,opacity=0.5] at (2,0,0) {};
    \node[fill=black,circle,inner sep=1pt,opacity=0.5] at (2,0,2) {};
    \node[fill=black,circle,inner sep=1pt,opacity=0.5] at (2,0,-2) {};
    \node[fill=black,circle,inner sep=1pt,opacity=0.5] at (-2,0,0) {};
    \node[fill=black,circle,inner sep=1pt,opacity=0.5] at (-2,0,2) {};
    \node[fill=black,circle,inner sep=1pt,opacity=0.5] at (-2,0,-2) {};
    \node[fill=black,circle,inner sep=1pt,opacity=0.5] at (0,5,0) {};
    \node[fill=black,circle,inner sep=1pt,opacity=0.5] at (0,5,2) {};
    \node[fill=black,circle,inner sep=1pt,opacity=0.5] at (0,5,-2) {};
    \node[fill=black,circle,inner sep=1pt,opacity=0.5] at (2,5,0) {};
    \node[fill=black,circle,inner sep=1pt,opacity=0.5] at (2,5,2) {};
    \node[fill=black,circle,inner sep=1pt,opacity=0.5] at (2,5,-2) {};
    \node[fill=black,circle,inner sep=1pt,opacity=0.5] at (-2,5,0) {};
    \node[fill=black,circle,inner sep=1pt,opacity=0.5] at (-2,5,2) {};
    \node[fill=black,circle,inner sep=1pt,opacity=0.5] at (-2,5,-2) {};
    \node[fill=black,circle,inner sep=1pt,opacity=0.5] at (4,5,2) {};
    \node[fill=black,circle,inner sep=1pt,opacity=0.5] at (4,5,0) {};
    \node[fill=black,circle,inner sep=1pt,opacity=0.5] at (4,5,-2) {};
    
    \coordinate (z_1) at (0,0,0);
    \coordinate (z_3) at (0,0,4);
    \coordinate (z_2) at (0,5,0);
    \coordinate (z_4) at (4,5,2);
    
    \coordinate (p_pl_1) at ($ (z_1)!1/2!(z_3) $);
    \coordinate (p_mi_1) at ($ (z_2)!1/2!(z_4) $);
    \coordinate (z_1_z_2_half) at ($ (z_1)!1/2!(z_2) $);
    \coordinate (z_2_z_3_half) at ($ (z_2)!1/2!(z_3) $);
    \coordinate (z_3_z_4_half) at ($ (z_3)!1/2!(z_4) $);
    \coordinate (z_4_z_1_half) at ($ (z_4)!1/2!(z_1) $); 
    
    \coordinate (p) at ($ (p_pl_1)!1/2!(p_mi_1) $);
    \coordinate (z_1_pl_1) at ($(p_mi_1) + (z_1) - (p_pl_1) $);
    \coordinate (z_3_pl_1) at ($(p_mi_1) + (z_3) - (p_pl_1) $);
    \coordinate (z_2_mi_1) at ($(p_pl_1) + (z_2) - (p_mi_1) $);
    \coordinate (z_4_mi_1) at ($(p_pl_1) + (z_4) - (p_mi_1) $);
    
    \coordinate (z_1_p) at ($ (z_1)!1/2!(z_1_pl_1) $);
    \coordinate (z_3_p) at ($(z_3)!1/2!(z_3_pl_1)$);
    \coordinate (z_2_p) at ($(z_2)!1/2!(z_2_mi_1)$);
    \coordinate (z_4_p) at ($(z_4)!1/2!(z_4_mi_1)$);
    
    \coordinate (p_1_p) at ($(z_1_pl_1)!1/2!(p_mi_1)$);
    \coordinate (p_3_p) at ($(z_3_pl_1)!1/2!(p_mi_1)$);
    
    \fill[blue!05,opacity=0.3] (3,0,5) -- (3,0,-3) -- (-3,0,-3) -- (-3,0,5) -- cycle;
    \fill[blue!05,opacity=0.3] (5,5,4) -- (5,5,-3) -- (-3,5,-3) -- (-3,5,4) -- cycle;
    
    \fill[blue!25,opacity=0.6] (z_1) -- (z_2_mi_1) -- (z_3) -- (z_4_mi_1) -- cycle;
    \fill[blue!25,opacity=0.6] (z_1_pl_1) -- (z_2) -- (z_3_pl_1) -- (z_4) -- cycle;
    
    \node[fill=blue,circle,inner sep=1pt] at (z_1) {};
    \draw (z_1) node[below right] {$\vec z_1$};
    
    \node[fill=blue,circle,inner sep=1pt] at (z_3) {};
    \draw (z_3) node[below right] {$\vec z_3$};
    
    \node[fill=blue,circle,inner sep=1pt] at (z_2) {};
    \draw (z_2) node[left] {$\vec z_2$};
    
    \node[fill=blue,circle,inner sep=1pt] at (z_4) {};
    \draw (z_4) node[above right] {$\vec z_4$};
    
    \node[fill=white,circle,inner sep=1pt,draw=blue] at (z_1_pl_1) {};
    \draw (z_1_pl_1) node[above] {$\vec z_1^{R}$};
    
    \node[fill=white,circle,inner sep=1pt,draw=blue] at (z_3_pl_1) {};
    \draw (z_3_pl_1) node[left] {$\vec z_3^{R}$};
    
    \node[fill=white,circle,inner sep=1pt,draw=blue] at (z_2_mi_1) {};
    \draw (z_2_mi_1) node[left] {$\vec z_2^{Q}$};
    
    \node[fill=white,circle,inner sep=1pt, draw=blue] at (z_4_mi_1) {};
    \draw (z_4_mi_1) node[right] {$\vec z_4^{Q}$};
    
    \node[fill=blue,circle,inner sep=1pt] at (p_1_p) {};
    \draw (p_1_p) node[above left] {$\vec{p}^{R}_{1}$};
    
    \node[fill=blue,circle,inner sep=1pt] at (p_3_p) {};
    \draw (p_3_p) node[above left] {$\vec{p}^{R}_{3}$};
    
    \node[fill=white,circle,inner sep=1pt, draw=blue] at (p) {};
    \draw (p) node[left] {$\vec p$};
    
    \node[fill=blue,circle,inner sep=1pt,draw=blue] at (p_pl_1) {};
    \draw (p_pl_1) node[left] {$\vec{p}^{Q}$};
    
    \node[fill=white,circle,inner sep=1pt,draw=blue] at (p_mi_1) {};
    \draw (p_mi_1) node[right] {$\vec{p}^{R}$};
    
    \draw[->] [red,very thin] (z_1) -- (z_2);
    \draw[->] [red,very thin] (z_1) -- (p_pl_1);
    \draw[->] [red,very thin] (z_1) -- (p_1_p);
    
    \draw (5, 0, 0) node[right] {$x$};
    \draw (0, 7, 0) node[left] {$z$};
    \draw (0, 0, 5) node[left] {$y$};
    
    \end{scope}
  \end{tikzpicture}
  \caption{Расположение точек внутри гиперплоскости $S_{1}$ из случая \textbf{В.1.б} леммы \ref{main_lem_2_2}}
\end{figure}

\textbf{Г}) $\Delta^{Q}_{k} \cap \Z^{4} = \{\vec{z}_{1}, \vec{z}_{3}, \vec{z}^{Q}_{2}, \vec{z}^{Q}_{4},  \vec{p}^{Q}\}$. Заметим, что случай $\Delta^{R}_{k} \cap \Z^{4} = \{\vec{z}_{2}, \vec{z}_{4}\}$ с точностью до перестановки индексов невозможен в силу пункта \textbf{А}. Также случай $\Delta^{R}_{k} \cap \Z^{4} = \{\vec{z}_{2}, \vec{z}_{4}, \vec{z}^{R}_{1}, \vec{z}^{R}_{3}\}$ с точностью до перестановки индексов невозможен в силу пункта \textbf{Б}. Кроме того, случай $\Delta^{R}_{k} \cap \Z^{4} = \{\vec{z}_{2}, \vec{z}_{4}, \vec{p}^{R}\}$ с точностью до перестановки индексов невозможен в силу пункта \textbf{В}. При этом $\vec{z}^{R}_{1} = \vec{z}_{2} + (\vec{z}_{1} - \vec{z}^{Q}_{2}) \in \Z^{4}$. Таким образом, существует ровно один подслучай:

\textbf{Г.1}) $\Delta^{R}_{k} \cap \Z^{4} = \{\vec{z}_{2}, \vec{z}_{4}, \vec{z}^{R}_{1}, \vec{z}^{R}_{3},  \vec{p}^{R}\}$. Этот случай, в свою очередь, разбивается на два подслучая:

\textbf{Г.1.а}) Плоскость $\pi$ рациональная. Тогда, плоскости $Q$ и $R$ равноудалены от $\pi$. Поскольку $\{\vec{z}_{1}, \vec{z}_{3}, \vec{z}^{Q}_{2}, \vec{z}^{Q}_{4}\} \subset \Delta^{Q}_{k}$, то в параллелограмме $\Delta^{\pi}_{k}$ должна существовать хотя бы одна точка решетки $\Z^{4}$. Пусть $\vec{w} \in \Delta^{\pi}_{k} \cap \Z^{4}$. Покажем, что 
\[\vec{w} \in \{\vec{z}^{\pi}_{1}, \vec{z}^{\pi}_{2}, \vec{z}^{\pi}_{3}, \vec{z}^{\pi}_{4}, \vec{p}, \frac{\vec{z}^{\pi}_{1} + \vec{z}^{\pi}_{2}}{2}, \frac{\vec{z}^{\pi}_{2} + \vec{z}^{\pi}_{3}}{2}, \frac{\vec{z}^{\pi}_{3} + \vec{z}^{\pi}_{4}}{2}, \frac{\vec{z}^{\pi}_{4} + \vec{z}^{\pi}_{1}}{2},  \frac{\vec{z}^{\pi}_{1} +  \vec{p}}{2}, \frac{\vec{z}^{\pi}_{2} +  \vec{p}}{2}, \frac{\vec{z}^{\pi}_{3} +  \vec{p}}{2}, \frac{\vec{z}^{\pi}_{4} +  \vec{p}}{2}\}.\]
Предположим, что это не так. Можно считать, что $\vec{w} \in \textup{conv}(\vec{z}^{\pi}_{1}, \frac{\vec{z}^{\pi}_{1} + \vec{z}^{\pi}_{2}}{2}, \vec{p}, \frac{\vec{z}^{\pi}_{4} + \vec{z}^{\pi}_{1}}{2})$. Тогда $\vec{w} + (\vec{w} - \vec{z}_1) \in \Delta^{R}_{k} \cap \Z^{4}$ и $\vec{w} + (\vec{w} - \vec{z}_1) \notin \{\vec{z}_{2}, \vec{z}_{4}, \vec{z}^{R}_{1}, \vec{z}^{R}_{3},  \vec{p}^{R}\}$, чего не может быть. Далее рассмотрим подслучаи:

\textbf{Г.1.а.1})  (будет соответствовать утверждению (3)) Предположим, что
 \[\vec{w} \in \{\frac{\vec{z}^{\pi}_{1} + \vec{z}^{\pi}_{2}}{2}, \frac{\vec{z}^{\pi}_{2} + \vec{z}^{\pi}_{3}}{2}, \frac{\vec{z}^{\pi}_{3} + \vec{z}^{\pi}_{4}}{2}, \frac{\vec{z}^{\pi}_{4} + \vec{z}^{\pi}_{1}}{2} \}.\]
 В этом случае каждая из точек множества 
\[\{\frac{\vec{z}^{\pi}_{1} + \vec{z}^{\pi}_{2}}{2}, \frac{\vec{z}^{\pi}_{2} + \vec{z}^{\pi}_{3}}{2}, \frac{\vec{z}^{\pi}_{3} + \vec{z}^{\pi}_{4}}{2}, \frac{\vec{z}^{\pi}_{4} + \vec{z}^{\pi}_{1}}{2} \}\]
 принадлежит решетке $\Z^{4}$. При этом, так как $\Delta^{Q}_{k} \cap \Z^{4} = \{\vec{z}_{1}, \vec{z}_{3}, \vec{z}^{Q}_{2}, \vec{z}^{Q}_{4},  \vec{p}^{Q}\}$, то никакая из точек множества 
\[\{\vec{z}^{\pi}_{1}, \vec{z}^{\pi}_{2}, \vec{z}^{\pi}_{3}, \vec{z}^{\pi}_{4}, \vec{p}, \frac{\vec{z}^{\pi}_{1} +  \vec{p}}{2}, \frac{\vec{z}^{\pi}_{2} +  \vec{p}}{2}, \frac{\vec{z}^{\pi}_{3} +  \vec{p}}{2}, \frac{\vec{z}^{\pi}_{4} +  \vec{p}}{2}\}\] не принадлежит решетке $\Z^{4}$. Тогда
\[(\Delta^{\pi}_{k} \cup \Delta^{Q}_{k} \cup \Delta^{R}_{k}) \cap \Z^{4} = \]
\[\{\vec{z}_{1}, \, \vec{z}^{Q}_{2}, \, \vec{z}_{3}, \, \vec{z}^{Q}_{4}, \, \vec{z}^{R}_{1}, \, \vec{z}_{2}, \, \vec{z}^{R}_{3}, \, \vec{z}_{4}, \, \vec{p}^{Q}, \, \vec{p}^{R}, \, \frac{\vec{z}^{\pi}_{1} + \vec{z}^{\pi}_{2}}{2}, \, \frac{\vec{z}^{\pi}_{2} + \vec{z}^{\pi}_{3}}{2}, \, \frac{\vec{z}^{\pi}_{3} + \vec{z}^{\pi}_{4}}{2}, \, \frac{\vec{z}^{\pi}_{4} + \vec{z}^{\pi}_{1}}{2}\}\]
и, так как плоскости $Q$ и $R$ равноудалены от $\pi$, набор векторов $\vec{z}_{1}$, $\frac{\vec{z}^{\pi}_{1} + \vec{z}^{\pi}_{2}}{2} = \frac{1}{2}(\vec{z}_{1} + \vec{z}_{2})$, $\vec{p}^{Q} = \frac{1}{2}(\vec{z}_{1}+\vec{z}_{3})$, $\frac{\vec{z}^{\pi}_{4} + \vec{z}^{\pi}_{1}}{2} = \frac{1}{2}(\vec{z}_{1} + \vec{z}_{4})$ образует базис решетки $\Z^{4}$, а значит, выполняется утверждение (3).

\begin{figure}[h]
  \centering
  \begin{tikzpicture}[x=10mm, y=7mm, z=-5mm, scale=0.7]
    \begin{scope}[rotate around x=0]
    
    \draw[->] [very thin] (-5,0,0) -- (5,0,0);
    \draw[->] [very thin] (0,-7,0) -- (0,13,0);
    \draw[->] [very thin] (0,0,-7) -- (0,0,7);
    
    \node[fill=black,circle,inner sep=1pt,opacity=0.5] at (0,0,0) {};
    \node[fill=black,circle,inner sep=1pt,opacity=0.5] at (0,0,2) {};
    \node[fill=black,circle,inner sep=1pt,opacity=0.5] at (0,0,-2) {};
    \node[fill=black,circle,inner sep=1pt,opacity=0.5] at (2,0,0) {};
    \node[fill=black,circle,inner sep=1pt,opacity=0.5] at (2,0,-2) {};
    \node[fill=black,circle,inner sep=1pt,opacity=0.5] at (-2,0,0) {};
    \node[fill=black,circle,inner sep=1pt,opacity=0.5] at (-2,0,2) {};
    \node[fill=black,circle,inner sep=1pt,opacity=0.5] at (0,-5,0) {};
    \node[fill=black,circle,inner sep=1pt,opacity=0.5] at (0,-5,2) {};
    \node[fill=black,circle,inner sep=1pt,opacity=0.5] at (0,-5,-2) {};
    \node[fill=black,circle,inner sep=1pt,opacity=0.5] at (2,-5,0) {};
    \node[fill=black,circle,inner sep=1pt,opacity=0.5] at (2,-5,2) {};
    \node[fill=black,circle,inner sep=1pt,opacity=0.5] at (2,-5,-2) {};
    \node[fill=black,circle,inner sep=1pt,opacity=0.5] at (-2,-5,0) {};
    \node[fill=black,circle,inner sep=1pt,opacity=0.5] at (-2,-5,2) {};
    \node[fill=black,circle,inner sep=1pt,opacity=0.5] at (-2,-5,-2) {};  
    \node[fill=black,circle,inner sep=1pt,opacity=0.5] at (0,5,0) {};
    \node[fill=black,circle,inner sep=1pt,opacity=0.5] at (0,5,2) {};
    \node[fill=black,circle,inner sep=1pt,opacity=0.5] at (0,5,-2) {};
    \node[fill=black,circle,inner sep=1pt,opacity=0.5] at (2,5,0) {};
    \node[fill=black,circle,inner sep=1pt,opacity=0.5] at (2,5,2) {};
    \node[fill=black,circle,inner sep=1pt,opacity=0.5] at (2,5,-2) {};
    \node[fill=black,circle,inner sep=1pt,opacity=0.5] at (-2,5,0) {};
    \node[fill=black,circle,inner sep=1pt,opacity=0.5] at (-2,5,2) {};
    \node[fill=black,circle,inner sep=1pt,opacity=0.5] at (-2,5,-2) {};  
    \node[fill=black,circle,inner sep=1pt,opacity=0.5] at (0,10,0) {};
    \node[fill=black,circle,inner sep=1pt,opacity=0.5] at (0,10,2) {};
    \node[fill=black,circle,inner sep=1pt,opacity=0.5] at (0,10,-2) {};
    \node[fill=black,circle,inner sep=1pt,opacity=0.5] at (2,10,0) {};
    \node[fill=black,circle,inner sep=1pt,opacity=0.5] at (2,10,2) {};
    \node[fill=black,circle,inner sep=1pt,opacity=0.5] at (2,10,-2) {};
    \node[fill=black,circle,inner sep=1pt,opacity=0.5] at (-2,10,0) {};
    \node[fill=black,circle,inner sep=1pt,opacity=0.5] at (-2,10,2) {};
    \node[fill=black,circle,inner sep=1pt,opacity=0.5] at (-2,10,-2) {};    
    
    \coordinate (z_1) at (0,0,0);
    \coordinate (z_3) at (0,0,4);
    \coordinate (z_2) at (0,10,0);
    \coordinate (z_4) at (4,0,0);
    
    \coordinate (p_pl_1) at ($ (z_1)!1/2!(z_3) $);
    \coordinate (p_mi_1) at ($ (z_2)!1/2!(z_4) $);
    \coordinate (z_1_z_2_half) at ($ (z_1)!1/2!(z_2) $);
    \coordinate (z_2_z_3_half) at ($ (z_2)!1/2!(z_3) $);
    \coordinate (z_3_z_4_half) at ($ (z_3)!1/2!(z_4) $);
    \coordinate (z_4_z_1_half) at ($ (z_4)!1/2!(z_1) $); 
    
    \coordinate (p) at ($ (p_pl_1)!1/2!(p_mi_1) $);
    \coordinate (z_1_pl_1) at ($(p_mi_1) + (z_1) - (p_pl_1) $);
    \coordinate (z_3_pl_1) at ($(p_mi_1) + (z_3) - (p_pl_1) $);
    \coordinate (z_2_mi_1) at ($(p_pl_1) + (z_2) - (p_mi_1) $);
    \coordinate (z_4_mi_1) at ($(p_pl_1) + (z_4) - (p_mi_1) $);
    
    \coordinate (z_1_p) at ($ (z_1)!1/2!(z_1_pl_1) $);
    \coordinate (z_3_p) at ($(z_3)!1/2!(z_3_pl_1)$);
    \coordinate (z_2_p) at ($(z_2)!1/2!(z_2_mi_1)$);
    \coordinate (z_4_p) at ($(z_4)!1/2!(z_4_mi_1)$);
    
    \fill[blue!05,opacity=0.3] (3,0,3) -- (3,0,-3) -- (-3,0,-3) -- (-3,0,3) -- cycle;
    \fill[blue!05,opacity=0.3] (3,-5,3) -- (3,-5,-3) -- (-3,-5,-3) -- (-3,-5,3) -- cycle;   
    \fill[blue!05,opacity=0.3] (3,5,3) -- (3,5,-3) -- (-3,5,-3) -- (-3,5,3) -- cycle;
    \fill[blue!05,opacity=0.3] (3,10,3) -- (3,10,-3) -- (-3,10,-3) -- (-3,10,3) -- cycle;
    
    \fill[blue!25,opacity=0.3] (z_1) -- (z_2_mi_1) -- (z_3) -- (z_4_mi_1) -- cycle;
    \fill[blue!25,opacity=0.6] (z_1_p) -- (z_2_p) -- (z_3_p) -- (z_4_p) -- cycle;
    \fill[blue!25,opacity=0.9] (z_1_pl_1) -- (z_2) -- (z_3_pl_1) -- (z_4) -- cycle;
    
    \node[fill=blue,circle,inner sep=1pt] at (z_1) {};
    \draw (z_1) node[below right] {$\vec z_1$};
    
    \node[fill=blue,circle,inner sep=1pt] at (z_3) {};
    \draw (z_3) node[left] {$\vec z_3$};
    
    \node[fill=blue,circle,inner sep=1pt] at (z_2) {};
    \draw (z_2) node[above] {$\vec z_2$};
    
    \node[fill=blue,circle,inner sep=1pt] at (z_4) {};
    \draw (z_4) node[above right] {$\vec z_4$};
    
    \node[fill=blue,circle,inner sep=1pt] at (z_1_pl_1) {};
    \draw (z_1_pl_1) node[right] {$\vec z_1^{R}$};
    
    \node[fill=blue,circle,inner sep=1pt] at (z_3_pl_1) {};
    \draw (z_3_pl_1) node[above right] {$\vec z_3^{R}$};
    
    \node[fill=blue,circle,inner sep=1pt] at (z_2_mi_1) {};
    \draw (z_2_mi_1) node[left] {$\vec z_2^{Q}$};
    
    \node[fill=blue,circle,inner sep=1pt] at (z_4_mi_1) {};
    \draw (z_4_mi_1) node[right] {$\vec z_4^{Q}$};
    
    \node[fill=white,circle,inner sep=1pt,draw=blue] at (z_1_p) {};
    \draw (z_1_p) node[below right] {$\vec z_1^{\pi}$};
    
    \node[fill=white,circle,inner sep=1pt,draw=blue] at (z_3_p) {};
    \draw (z_3_p) node[left] {$\vec z_3^{\pi}$};
    
    \node[fill=white,circle,inner sep=1pt,draw=blue] at (z_2_p) {};
    \draw (z_2_p) node[below left] {$\vec z_2^{\pi}$};
    
    \node[fill=white,circle,inner sep=1pt,draw=blue] at (z_4_p) {};
    \draw (z_4_p) node[above right] {$\vec z_4^{\pi}$};
    
    \node[fill=white,circle,inner sep=1pt,draw=blue] at (p) {};
    \draw (p) node[right] {$\vec p$};
    
    \node[fill=blue,circle,inner sep=1pt] at (p_pl_1) {};
    \draw (p_pl_1) node[right] {$\vec{p}^{Q}$};
    
    \node[fill=blue,circle,inner sep=1pt] at (p_mi_1) {};
    \draw (p_mi_1) node[above right] {$\vec{p}^{R}$};
    
    \node[fill=blue,circle,inner sep=1pt] at (z_1_z_2_half) {};
    \draw (z_1_z_2_half) node[right] {$\frac{\vec{z}_1^{\pi} + \vec{z}_2^{\pi}}{2}$};
    
    \node[fill=blue,circle,inner sep=1pt] at (z_2_z_3_half) {};
    \draw (z_2_z_3_half) node[left] {$\frac{\vec{z}_2^{\pi} + \vec{z}_3^{\pi}}{2}$};
    
    \node[fill=blue,circle,inner sep=1pt] at (z_3_z_4_half) {};
    \draw (z_3_z_4_half) node[below] {$\frac{\vec{z}_3^{\pi} + \vec{z}_4^{\pi}}{2}$};
    
    \node[fill=blue,circle,inner sep=1pt] at (z_4_z_1_half) {};
    \draw (z_4_z_1_half) node[above right] {$\frac{\vec{z}_4^{\pi} + \vec{z}_1^{\pi}}{2}$};
    
    \draw[->] [red,very thin] (z_1) -- (z_1_z_2_half);
    \draw[->] [red,very thin] (z_1) -- (z_2_z_3_half);
    \draw[->] [red,very thin] (z_1) -- (z_3_z_4_half);
    
    \draw (5, 0, 0) node[right] {$x$};
    \draw (0, 13, 0) node[left] {$z$};
    \draw (0, 0, 7) node[left] {$y$};
    
    \end{scope}
  \end{tikzpicture}
  \caption{Расположение точек внутри гиперплоскости $S_{1}$ из случая \textbf{Г.1.а.1} леммы \ref{main_lem_2_2}}
\end{figure}

\textbf{Г.1.а.2}) (будет соответствовать утверждению (4)) Предположим, что 
\[\vec{w} \in \{\vec{z}^{\pi}_{1}, \vec{z}^{\pi}_{2}, \vec{z}^{\pi}_{3}, \vec{z}^{\pi}_{4}, \vec{p}\}.\]
В этом случае каждая из точек множества $\{\vec{z}^{\pi}_{1}, \vec{z}^{\pi}_{2}, \vec{z}^{\pi}_{3}, \vec{z}^{\pi}_{4}, \vec{p}\}$ принадлежит решетке $\Z^{4}$. В силу доказательства случая \textbf{Г.1.а.1} никакая из точек множества $\{\frac{\vec{z}^{\pi}_{1} + \vec{z}^{\pi}_{2}}{2}, \frac{\vec{z}^{\pi}_{2} + \vec{z}^{\pi}_{3}}{2}, \frac{\vec{z}^{\pi}_{3} + \vec{z}^{\pi}_{4}}{2}, \frac{\vec{z}^{\pi}_{4} + \vec{z}^{\pi}_{1}}{2}\}$ не принадлежит решетке $\Z^{4}$. Кроме того, так как $\Delta^{Q}_{k} \cap \Z^{4} = \{\vec{z}_{1}, \vec{z}_{3}, \vec{z}^{Q}_{2}, \vec{z}^{Q}_{4},  \vec{p}^{Q}\}$, то никакая из точек множества $\{\frac{\vec{z}^{\pi}_{1} +  \vec{p}}{2}, \frac{\vec{z}^{\pi}_{2} +  \vec{p}}{2}, \frac{\vec{z}^{\pi}_{3} +  \vec{p}}{2}, \frac{\vec{z}^{\pi}_{4} +  \vec{p}}{2} \}$ не принадлежит решетке $\Z^{4}$. Тогда
\[(\Delta^{\pi}_{k} \cup \Delta^{Q}_{k} \cup \Delta^{R}_{k}) \cap \Z^{4} = \{\vec{p}, \, \vec{z}_{1}, \, \vec{z}^{Q}_{2}, \, \vec{z}_{3}, \, \vec{z}^{Q}_{4}, \, \vec{z}^{R}_{1}, \, \vec{z}_{2}, \, \vec{z}^{R}_{3}, \, \vec{z}_{4}, \, \vec{p}^{Q}, \, \vec{p}^{R}, \, \vec{z}^{\pi}_{1}, \, \vec{z}^{\pi}_{2}, \, \vec{z}^{\pi}_{3}, \, \vec{z}^{\pi}_{4} \}\]
и, так как плоскости $Q$ и $R$ равноудалены от $\pi$, набор векторов $\vec{z}_{1}$, $\vec{z}_{2}$, $\vec{p}^{Q} = \frac{1}{2}(\vec{z}_{1}+\vec{z}_{3})$, $\vec{p} = \frac{1}{4}(\vec{z}_{1}+\vec{z}_{2}+\vec{z}_{3}+\vec{z}_{4})$ образует базис решетки $\Z^{4}$, а значит, выполняется утверждение (4).

 \begin{figure}[h]
  \centering
  \begin{tikzpicture}[x=10mm, y=7mm, z=-5mm, scale=0.8]
    \begin{scope}[rotate around x=0]
    
    \draw[->] [very thin] (-5,0,0) -- (5,0,0);
    \draw[->] [very thin] (0,-7,0) -- (0,7,0);
    \draw[->] [very thin] (0,0,-7) -- (0,0,7);
    
    \node[fill=black,circle,inner sep=1pt,opacity=0.5] at (0,0,0) {};
    \node[fill=black,circle,inner sep=1pt,opacity=0.5] at (0,0,2) {};
    \node[fill=black,circle,inner sep=1pt,opacity=0.5] at (0,0,-2) {};
    \node[fill=black,circle,inner sep=1pt,opacity=0.5] at (2,0,0) {};
    \node[fill=black,circle,inner sep=1pt,opacity=0.5] at (2,0,2) {};
    \node[fill=black,circle,inner sep=1pt,opacity=0.5] at (2,0,-2) {};
    \node[fill=black,circle,inner sep=1pt,opacity=0.5] at (-2,0,0) {};
    \node[fill=black,circle,inner sep=1pt,opacity=0.5] at (-2,0,2) {};
    \node[fill=black,circle,inner sep=1pt,opacity=0.5] at (-2,0,-2) {};
    \node[fill=black,circle,inner sep=1pt,opacity=0.5] at (0,-5,0) {};
    \node[fill=black,circle,inner sep=1pt,opacity=0.5] at (0,-5,2) {};
    \node[fill=black,circle,inner sep=1pt,opacity=0.5] at (0,-5,-2) {};
    \node[fill=black,circle,inner sep=1pt,opacity=0.5] at (2,-5,0) {};
    \node[fill=black,circle,inner sep=1pt,opacity=0.5] at (2,-5,2) {};
    \node[fill=black,circle,inner sep=1pt,opacity=0.5] at (2,-5,-2) {};
    \node[fill=black,circle,inner sep=1pt,opacity=0.5] at (-2,-5,0) {};
    \node[fill=black,circle,inner sep=1pt,opacity=0.5] at (-2,-5,2) {};
    \node[fill=black,circle,inner sep=1pt,opacity=0.5] at (-2,-5,-2) {};
    \node[fill=black,circle,inner sep=1pt,opacity=0.5] at (2,-5,4) {};
    \node[fill=black,circle,inner sep=1pt,opacity=0.5] at (0,-5,4) {};
    \node[fill=black,circle,inner sep=1pt,opacity=0.5] at (-2,-5,4) {};
    \node[fill=black,circle,inner sep=1pt,opacity=0.5] at (-4,-5,4) {};
    \node[fill=black,circle,inner sep=1pt,opacity=0.5] at (-4,-5,2) {};
    \node[fill=black,circle,inner sep=1pt,opacity=0.5] at (-4,-5,0) {};
    \node[fill=black,circle,inner sep=1pt,opacity=0.5] at (-4,-5,-2) {};
    \node[fill=black,circle,inner sep=1pt,opacity=0.5] at (0,5,0) {};
    \node[fill=black,circle,inner sep=1pt,opacity=0.5] at (0,5,2) {};
    \node[fill=black,circle,inner sep=1pt,opacity=0.5] at (0,5,-2) {};
    \node[fill=black,circle,inner sep=1pt,opacity=0.5] at (2,5,0) {};
    \node[fill=black,circle,inner sep=1pt,opacity=0.5] at (2,5,2) {};
    \node[fill=black,circle,inner sep=1pt,opacity=0.5] at (2,5,-2) {};
    \node[fill=black,circle,inner sep=1pt,opacity=0.5] at (-2,5,0) {};
    \node[fill=black,circle,inner sep=1pt,opacity=0.5] at (-2,5,2) {};
    \node[fill=black,circle,inner sep=1pt,opacity=0.5] at (-2,5,-2) {};
    \node[fill=black,circle,inner sep=1pt,opacity=0.5] at (4,5,-2) {};
    \node[fill=black,circle,inner sep=1pt,opacity=0.5] at (4,5,2) {};
    \node[fill=black,circle,inner sep=1pt,opacity=0.5] at (4,5,0) {};
    \node[fill=black,circle,inner sep=1pt,opacity=0.5] at (4,5,-4) {};
    \node[fill=black,circle,inner sep=1pt,opacity=0.5] at (0,5,-4) {};
    \node[fill=black,circle,inner sep=1pt,opacity=0.5] at (-2,5,-4) {};
    
    \coordinate (z_1) at (0,5,-2);
    \coordinate (z_3) at (4,5,-2);
    \coordinate (z_2) at (-2,-5,4);
    \coordinate (z_4) at (-2,-5,0);
    
    \coordinate (p_pl_1) at ($ (z_1)!1/2!(z_3) $);
    \coordinate (p_mi_1) at ($ (z_2)!1/2!(z_4) $);
    \coordinate (z_1_z_2_half) at ($ (z_1)!1/2!(z_2) $);
    \coordinate (z_2_z_3_half) at ($ (z_2)!1/2!(z_3) $);
    \coordinate (z_3_z_4_half) at ($ (z_3)!1/2!(z_4) $);
    \coordinate (z_4_z_1_half) at ($ (z_4)!1/2!(z_1) $); 
    
    \coordinate (p) at ($ (p_pl_1)!1/2!(p_mi_1) $);
    \coordinate (z_1_pl_1) at ($(p_mi_1) + (z_1) - (p_pl_1) $);
    \coordinate (z_3_pl_1) at ($(p_mi_1) + (z_3) - (p_pl_1) $);
    \coordinate (z_2_mi_1) at ($(p_pl_1) + (z_2) - (p_mi_1) $);
    \coordinate (z_4_mi_1) at ($(p_pl_1) + (z_4) - (p_mi_1) $);
    
    \coordinate (z_1_p) at ($ (z_1)!1/2!(z_1_pl_1) $);
    \coordinate (z_3_p) at ($(z_3)!1/2!(z_3_pl_1)$);
    \coordinate (z_2_p) at ($(z_2)!1/2!(z_2_mi_1)$);
    \coordinate (z_4_p) at ($(z_4)!1/2!(z_4_mi_1)$);
    
    \fill[blue!05,opacity=0.3] (3,0,3) -- (3,0,-3) -- (-3,0,-3) -- (-3,0,3) -- cycle;
    \fill[blue!05,opacity=0.3] (3,-5,5) -- (3,-5,-3) -- (-5,-5,-3) -- (-5,-5,5) -- cycle;   
    \fill[blue!05,opacity=0.3] (3,5,3) -- (3,5,-3) -- (-3,5,-3) -- (-3,5,3) -- cycle;
    
    \fill[blue!25,opacity=0.7] (z_1) -- (z_2_mi_1) -- (z_3) -- (z_4_mi_1) -- cycle; 
    \fill[blue!25,opacity=0.7] (z_1_pl_1) -- (z_2) -- (z_3_pl_1) -- (z_4) -- cycle;
    \fill[blue!25,opacity=0.7] (z_1_p) -- (z_2_p) -- (z_3_p) -- (z_4_p) -- cycle;
    
    \node[fill=blue,circle,inner sep=1pt] at (z_1) {};
    \draw (z_1) node[left] {$\vec z_1$};
    
    \node[fill=blue,circle,inner sep=1pt] at (z_3) {};
    \draw (z_3) node[right] {$\vec z_3$};
    
    \node[fill=blue,circle,inner sep=1pt] at (z_2) {};
    \draw (z_2) node[left] {$\vec z_2$};
    
    \node[fill=blue,circle,inner sep=1pt] at (z_4) {};
    \draw (z_4) node[right] {$\vec z_4$};
    
    \node[fill=blue,circle,inner sep=1pt] at (z_1_pl_1) {};
    \draw (z_1_pl_1) node[left] {$\vec z_1^{R}$};
    
    \node[fill=blue,circle,inner sep=1pt] at (z_3_pl_1) {};
    \draw (z_3_pl_1) node[below right] {$\vec z_3^{R}$};
    
    \node[fill=blue,circle,inner sep=1pt] at (z_2_mi_1) {};
    \draw (z_2_mi_1) node[left] {$\vec z_2^{Q}$};
    
    \node[fill=blue,circle,inner sep=1pt] at (z_4_mi_1) {};
    \draw (z_4_mi_1) node[right] {$\vec z_4^{Q}$};
    
    \node[fill=blue,circle,inner sep=1pt] at (z_1_p) {};
    \draw (z_1_p) node[above left] {$\vec z_1^{\pi}$};
    
    \node[fill=blue,circle,inner sep=1pt] at (z_3_p) {};
    \draw (z_3_p) node[below right] {$\vec z_3^{\pi}$};
    
    \node[fill=blue,circle,inner sep=1pt] at (z_2_p) {};
    \draw (z_2_p) node[below left] {$\vec z_2^{\pi}$};
    
    \node[fill=blue,circle,inner sep=1pt] at (z_4_p) {};
    \draw (z_4_p) node[above] {$\vec z_4^{\pi}$};
    
    \node[fill=blue,circle,inner sep=1pt] at (p) {};
    \draw (p) node[right] {$\vec p$};
    
    \node[fill=blue,circle,inner sep=1pt] at (p_pl_1) {};
    \draw (p_pl_1) node[above] {$\vec{p}^{Q}$};
    
    \node[fill=blue,circle,inner sep=1pt,draw=blue] at (p_mi_1) {};
    \draw (p_mi_1) node[right] {$\vec{p}^{R}$};
    
    \node[fill=white,circle,inner sep=1pt,draw=blue] at (z_1_z_2_half) {};
    \draw (z_1_z_2_half) node[left] {$\frac{\vec{z}_1^{\pi} + \vec{z}_2^{\pi}}{2}$};
    
    \node[fill=white,circle,inner sep=1pt,draw=blue] at (z_2_z_3_half) {};
    \draw (z_2_z_3_half) node[below] {$\frac{\vec{z}_2^{\pi} + \vec{z}_3^{\pi}}{2}$};
    
    \node[fill=white,circle,inner sep=1pt,draw=blue] at (z_3_z_4_half) {};
    \draw (z_3_z_4_half) node[right] {$\frac{\vec{z}_3^{\pi} + \vec{z}_4^{\pi}}{2}$};
    
    \node[fill=white,circle,inner sep=1pt,draw=blue] at (z_4_z_1_half) {};
    \draw (z_4_z_1_half) node[left] {$\frac{\vec{z}_4^{\pi} + \vec{z}_1^{\pi}}{2}$};
    
    \draw[->] [red,very thin] (p) -- (z_1);
    \draw[->] [red,very thin] (p) -- (z_2);
    \draw[->] [red,very thin] (p) -- (p_pl_1);
    
    \draw (5, 0, 0) node[right] {$x$};
    \draw (0, 7, 0) node[left] {$z$};
    \draw (0, 0, 7) node[left] {$y$};
    
    \end{scope}
  \end{tikzpicture}
  \caption{Расположение точек внутри гиперплоскости $S_{1}$ из случая \textbf{Г.1.а.2} леммы \ref{main_lem_2_2}}
\end{figure}

\textbf{Г.1.а.3})  (будет соответствовать утверждению (10)) Предположим, что $\vec{w} \in \{\frac{\vec{z}^{\pi}_{1} +  \vec{p}}{2}, \frac{\vec{z}^{\pi}_{3} +  \vec{p}}{2}\}$. В силу доказательства случаев \textbf{Г.1.а.1} и \textbf{Г.1.а.2} никакая из точек множества 
\[\{\frac{\vec{z}^{\pi}_{1} + \vec{z}^{\pi}_{2}}{2}, \frac{\vec{z}^{\pi}_{2} + \vec{z}^{\pi}_{3}}{2}, \frac{\vec{z}^{\pi}_{3} + \vec{z}^{\pi}_{4}}{2}, \frac{\vec{z}^{\pi}_{4} + \vec{z}^{\pi}_{1}}{2}, \vec{z}^{\pi}_{1}, \vec{z}^{\pi}_{2}, \vec{z}^{\pi}_{3}, \vec{z}^{\pi}_{4}, \vec{p}\}\]
не принадлежит решетке $\Z^{4}$. Кроме того, так как $\Delta^{Q}_{k} \cap \Z^{4} = \{\vec{z}_{1}, \vec{z}_{3}, \vec{z}^{Q}_{2}, \vec{z}^{Q}_{4},  \vec{p}^{Q}\}$, то никакая из точек множества $\{\frac{\vec{z}^{\pi}_{2} +  \vec{p}}{2}, \frac{\vec{z}^{\pi}_{4} +  \vec{p}}{2}\}$ не принадлежит решетке $\Z^{4}$. Тогда
\[(\Delta^{\pi}_{k} \cup \Delta^{Q}_{k} \cup \Delta^{R}_{k}) \cap \Z^{4} = \{\vec{z}_{1}, \, \vec{z}^{Q}_{2}, \, \vec{z}_{3}, \, \vec{z}^{Q}_{4}, \, \vec{z}^{R}_{1}, \, \vec{z}_{2}, \, \vec{z}^{R}_{3}, \, \vec{z}_{4}, \, \vec{p}^{Q}, \, \vec{p}^{R}, \, \frac{\vec{z}^{\pi}_{1} +  \vec{p}}{2}, \, \frac{\vec{z}^{\pi}_{3} +  \vec{p}}{2} \}\]
и, так как плоскости $Q$ и $R$ равноудалены от $\pi$, набор векторов $\vec{z}_{1}$, $\vec{z}_{2}$, $\vec{p}^{Q} = \frac{1}{2}(\vec{z}_{1}+\vec{z}_{3})$, $\frac{\vec{z}^{\pi}_{1} +  \vec{p}}{2} = \frac{1}{2}\vec{z}_{1} +  \frac{1}{4}\vec{z}_{2} +  \frac{1}{4}\vec{z}_{4}$ образует базис решетки $\Z^{4}$, а значит, выполняется утверждение (10).

 \begin{figure}[h]
  \centering
  \begin{tikzpicture}[x=10mm, y=7mm, z=-5mm, scale=0.7]
    \begin{scope}[rotate around x=0]
    
    \draw[->] [very thin] (-5,0,0) -- (5,0,0);
    \draw[->] [very thin] (0,-7,0) -- (0,7,0);
    \draw[->] [very thin] (0,0,-7) -- (0,0,7);
    
    \node[fill=black,circle,inner sep=1pt,opacity=0.5] at (0,0,0) {};
    \node[fill=black,circle,inner sep=1pt,opacity=0.5] at (0,0,2) {};
    \node[fill=black,circle,inner sep=1pt,opacity=0.5] at (0,0,-2) {};
    \node[fill=black,circle,inner sep=1pt,opacity=0.5] at (2,0,0) {};
    \node[fill=black,circle,inner sep=1pt,opacity=0.5] at (2,0,2) {};
    \node[fill=black,circle,inner sep=1pt,opacity=0.5] at (2,0,-2) {};
    \node[fill=black,circle,inner sep=1pt,opacity=0.5] at (-2,0,0) {};
    \node[fill=black,circle,inner sep=1pt,opacity=0.5] at (-2,0,2) {};
    \node[fill=black,circle,inner sep=1pt,opacity=0.5] at (-2,0,-2) {};
    \node[fill=black,circle,inner sep=1pt,opacity=0.5] at (0,-5,0) {};
    \node[fill=black,circle,inner sep=1pt,opacity=0.5] at (0,-5,2) {};
    \node[fill=black,circle,inner sep=1pt,opacity=0.5] at (0,-5,-2) {};
    \node[fill=black,circle,inner sep=1pt,opacity=0.5] at (2,-5,0) {};
    \node[fill=black,circle,inner sep=1pt,opacity=0.5] at (2,-5,2) {};
    \node[fill=black,circle,inner sep=1pt,opacity=0.5] at (2,-5,-2) {};
    \node[fill=black,circle,inner sep=1pt,opacity=0.5] at (-2,-5,0) {};
    \node[fill=black,circle,inner sep=1pt,opacity=0.5] at (-2,-5,2) {};
    \node[fill=black,circle,inner sep=1pt,opacity=0.5] at (-2,-5,-2) {};
    \node[fill=black,circle,inner sep=1pt,opacity=0.5] at (2,-5,4) {};
    \node[fill=black,circle,inner sep=1pt,opacity=0.5] at (0,-5,4) {};
    \node[fill=black,circle,inner sep=1pt,opacity=0.5] at (-2,-5,4) {};
    \node[fill=black,circle,inner sep=1pt,opacity=0.5] at (-4,-5,4) {};
    \node[fill=black,circle,inner sep=1pt,opacity=0.5] at (-4,-5,2) {};
    \node[fill=black,circle,inner sep=1pt,opacity=0.5] at (-4,-5,0) {};
    \node[fill=black,circle,inner sep=1pt,opacity=0.5] at (-4,-5,-2) {};
    \node[fill=black,circle,inner sep=1pt,opacity=0.5] at (0,5,0) {};
    \node[fill=black,circle,inner sep=1pt,opacity=0.5] at (0,5,2) {};
    \node[fill=black,circle,inner sep=1pt,opacity=0.5] at (0,5,-2) {};
    \node[fill=black,circle,inner sep=1pt,opacity=0.5] at (2,5,0) {};
    \node[fill=black,circle,inner sep=1pt,opacity=0.5] at (2,5,2) {};
    \node[fill=black,circle,inner sep=1pt,opacity=0.5] at (2,5,-2) {};
    \node[fill=black,circle,inner sep=1pt,opacity=0.5] at (-2,5,0) {};
    \node[fill=black,circle,inner sep=1pt,opacity=0.5] at (-2,5,2) {};
    \node[fill=black,circle,inner sep=1pt,opacity=0.5] at (-2,5,-2) {};
    \node[fill=black,circle,inner sep=1pt,opacity=0.5] at (4,5,-2) {};
    \node[fill=black,circle,inner sep=1pt,opacity=0.5] at (4,5,2) {};
    \node[fill=black,circle,inner sep=1pt,opacity=0.5] at (4,5,0) {};
    \node[fill=black,circle,inner sep=1pt,opacity=0.5] at (4,5,-4) {};
    \node[fill=black,circle,inner sep=1pt,opacity=0.5] at (0,5,-4) {};
    \node[fill=black,circle,inner sep=1pt,opacity=0.5] at (-2,5,-4) {};
    
    \coordinate (z_1) at (0,0,0);
    \coordinate (z_3) at (0,0,4);
    \coordinate (z_2) at (0,5,0);
    \coordinate (z_4) at (8,-5,0);
    
    \coordinate (p_pl_1) at ($ (z_1)!1/2!(z_3) $);
    \coordinate (p_mi_1) at ($ (z_2)!1/2!(z_4) $);
    \coordinate (z_1_z_2_half) at ($ (z_1)!1/2!(z_2) $);
    \coordinate (z_2_z_3_half) at ($ (z_2)!1/2!(z_3) $);
    \coordinate (z_3_z_4_half) at ($ (z_3)!1/2!(z_4) $);
    \coordinate (z_4_z_1_half) at ($ (z_4)!1/2!(z_1) $); 
    
    \coordinate (p) at ($ (p_pl_1)!1/2!(p_mi_1) $);
    \coordinate (z_1_pl_1) at ($(p_mi_1) + (z_1) - (p_pl_1) $);
    \coordinate (z_3_pl_1) at ($(p_mi_1) + (z_3) - (p_pl_1) $);
    \coordinate (z_2_mi_1) at ($(p_pl_1) + (z_2) - (p_mi_1) $);
    \coordinate (z_4_mi_1) at ($(p_pl_1) + (z_4) - (p_mi_1) $);
    
    \coordinate (z_1_p) at ($ (z_1)!1/2!(z_1_pl_1) $);
    \coordinate (z_3_p) at ($(z_3)!1/2!(z_3_pl_1)$);
    \coordinate (z_2_p) at ($(z_2)!1/2!(z_2_mi_1)$);
    \coordinate (z_4_p) at ($(z_4)!1/2!(z_4_mi_1)$);
    
    \coordinate (z_1_p_1_2) at ($(p)!1/2!(z_1_p)$);
    \coordinate (z_3_p_1_2) at ($(p)!1/2!(z_3_p)$);
    
    \fill[blue!05,opacity=0.6] (3,0,3) -- (3,0,-3) -- (-3,0,-3) -- (-3,0,3) -- cycle;
    \fill[blue!05,opacity=0.3] (3,-5,5) -- (3,-5,-3) -- (-5,-5,-3) -- (-5,-5,5) -- cycle;   
    \fill[blue!05,opacity=0.3] (3,5,3) -- (3,5,-3) -- (-3,5,-3) -- (-3,5,3) -- cycle;
    
    \fill[blue!25,opacity=0.7] (z_1) -- (z_2_mi_1) -- (z_3) -- (z_4_mi_1) -- cycle;
    \fill[blue!25,opacity=0.7] (z_1_pl_1) -- (z_2) -- (z_3_pl_1) -- (z_4) -- cycle;
    \fill[blue!25,opacity=0.7] (z_1_p) -- (z_2_p) -- (z_3_p) -- (z_4_p) -- cycle;
    
    \node[fill=blue,circle,inner sep=1pt] at (z_1) {};
    \draw (z_1) node[left] {$\vec z_1$};
    
    \node[fill=blue,circle,inner sep=1pt] at (z_3) {};
    \draw (z_3) node[right] {$\vec z_3$};
    
    \node[fill=blue,circle,inner sep=1pt] at (z_2) {};
    \draw (z_2) node[left] {$\vec z_2$};
    
    \node[fill=blue,circle,inner sep=1pt] at (z_4) {};
    \draw (z_4) node[right] {$\vec z_4$};
    
    \node[fill=blue,circle,inner sep=1pt] at (z_1_pl_1) {};
    \draw (z_1_pl_1) node[left] {$\vec z_1^{R}$};
    
    \node[fill=blue,circle,inner sep=1pt] at (z_3_pl_1) {};
    \draw (z_3_pl_1) node[below right] {$\vec z_3^{R}$};
    
    \node[fill=blue,circle,inner sep=1pt] at (z_2_mi_1) {};
    \draw (z_2_mi_1) node[left] {$\vec z_2^{Q}$};
    
    \node[fill=blue,circle,inner sep=1pt] at (z_4_mi_1) {};
    \draw (z_4_mi_1) node[right] {$\vec z_4^{Q}$};
    
    \node[fill=white,circle,inner sep=1pt,draw=blue] at (z_1_p) {};
    \draw (z_1_p) node[above left] {$\vec z_1^{\pi}$};
    
    \node[fill=white,circle,inner sep=1pt,draw=blue] at (z_3_p) {};
    \draw (z_3_p) node[below right] {$\vec z_3^{\pi}$};
    
    \node[fill=white,circle,inner sep=1pt,draw=blue] at (z_2_p) {};
    \draw (z_2_p) node[below left] {$\vec z_2^{\pi}$};
    
    \node[fill=white,circle,inner sep=1pt,draw=blue] at (z_4_p) {};
    \draw (z_4_p) node[above] {$\vec z_4^{\pi}$};
    
    \node[fill=white,circle,inner sep=1pt,draw=blue] at (p) {};
    \draw (p) node[right] {$\vec p$};
    
    \node[fill=blue,circle,inner sep=1pt] at (z_1_p_1_2) {};
    \draw (z_1_p_1_2) node[right] {$\frac{\vec{z}^{\pi}_{1} +  \vec{p}}{2}$};
    
    \node[fill=blue,circle,inner sep=1pt] at (z_3_p_1_2) {};
    \draw (z_3_p_1_2) node[right] {$\frac{\vec{z}^{\pi}_{3} +  \vec{p}}{2}$};
    
    \node[fill=blue,circle,inner sep=1pt] at (p_pl_1) {};
    \draw (p_pl_1) node[above] {$\vec{p}^{Q}$};
    
    \node[fill=blue,circle,inner sep=1pt,draw=blue] at (p_mi_1) {};
    \draw (p_mi_1) node[right] {$\vec{p}^{R}$};
    
    \node[fill=white,circle,inner sep=1pt,draw=blue] at (z_1_z_2_half) {};
    \draw (z_1_z_2_half) node[left] {$\frac{\vec{z}_1^{\pi} + \vec{z}_2^{\pi}}{2}$};
    
    \node[fill=white,circle,inner sep=1pt,draw=blue] at (z_2_z_3_half) {};
    \draw (z_2_z_3_half) node[left] {$\frac{\vec{z}_2^{\pi} + \vec{z}_3^{\pi}}{2}$};
    
    \node[fill=white,circle,inner sep=1pt,draw=blue] at (z_3_z_4_half) {};
    \draw (z_3_z_4_half) node[right] {$\frac{\vec{z}_3^{\pi} + \vec{z}_4^{\pi}}{2}$};
    
    \node[fill=white,circle,inner sep=1pt,draw=blue] at (z_4_z_1_half) {};
    \draw (z_4_z_1_half) node[right] {$\frac{\vec{z}_4^{\pi} + \vec{z}_1^{\pi}}{2}$};
    
    \draw[->] [red,very thin] (z_1) -- (z_2);
    \draw[->] [red,very thin] (z_1) -- (p_pl_1);
    \draw[->] [red,very thin] (z_1) -- (z_1_p_1_2);
    
    \draw (5, 0, 0) node[right] {$x$};
    \draw (0, 7, 0) node[left] {$z$};
    \draw (0, 0, 7) node[left] {$y$};
    
    \end{scope}
  \end{tikzpicture}
  \caption{Расположение точек внутри гиперплоскости $S_{1}$ из случая \textbf{Г.1.а.3} леммы \ref{main_lem_2_2}}
\end{figure}

\textbf{Г.1.а.4}) Пусть $\vec{w} \in \{\frac{\vec{z}^{\pi}_{2} +  \vec{p}}{2}, \frac{\vec{z}^{\pi}_{4} +  \vec{p}}{2}\}$. Заметим, что этот случай с точностью до перестановки индексов полностью эквивалентен пункту \textbf{Г.1.а.3}. 

\textbf{Г.1.б}) (будет соответствовать утверждению (5)) Плоскость $\pi$ не является рациональной плоскостью. Тогда 
\[(\Delta^{\pi}_{k} \cup \Delta^{Q}_{k} \cup \Delta^{R}_{k}) \cap \Z^{4} = \{\vec{z}_{1}, \, \vec{z}^{Q}_{2}, \, \vec{z}_{3}, \, \vec{z}^{Q}_{4}, \, \vec{z}^{R}_{1}, \, \vec{z}_{2}, \, \vec{z}^{R}_{3}, \, \vec{z}_{4}, \, \vec{p}^{Q}, \, \vec{p}^{R}\},\]
набор векторов $\vec{z}_{1}$, $\vec{z}_{2}$, $\vec{p}^{Q} = \frac{1}{2}(\vec{z}_{1}+\vec{z}_{3})$, $\vec{p}^{R} = \frac{1}{2}(\vec{z}_{2}+\vec{z}_{4})$ образует базис решетки $\Z^{4}$, а значит, выполняется утверждение (5).

  \begin{figure}[h]
  \centering
  \begin{tikzpicture}[x=10mm, y=7mm, z=-5mm, scale=0.6]
    \begin{scope}[rotate around x=0]
    
    \draw[->] [very thin] (-5,0,0) -- (5,0,0);
    \draw[->] [very thin] (0,-4,0) -- (0,7,0);
    \draw[->] [very thin] (0,0,-5) -- (0,0,6);
    
    \node[fill=black,circle,inner sep=1pt,opacity=0.5] at (0,0,0) {};
    \node[fill=black,circle,inner sep=1pt,opacity=0.5] at (0,0,2) {};
    \node[fill=black,circle,inner sep=1pt,opacity=0.5] at (0,0,-2) {};
    \node[fill=black,circle,inner sep=1pt,opacity=0.5] at (2,0,0) {};
    \node[fill=black,circle,inner sep=1pt,opacity=0.5] at (2,0,2) {};
    \node[fill=black,circle,inner sep=1pt,opacity=0.5] at (2,0,-2) {};
    \node[fill=black,circle,inner sep=1pt,opacity=0.5] at (-2,0,0) {};
    \node[fill=black,circle,inner sep=1pt,opacity=0.5] at (-2,0,2) {};
    \node[fill=black,circle,inner sep=1pt,opacity=0.5] at (-2,0,-2) {};
    \node[fill=black,circle,inner sep=1pt,opacity=0.5] at (2,0,4) {};
    \node[fill=black,circle,inner sep=1pt,opacity=0.5] at (0,5,0) {};
    \node[fill=black,circle,inner sep=1pt,opacity=0.5] at (0,5,2) {};
    \node[fill=black,circle,inner sep=1pt,opacity=0.5] at (0,5,-2) {};
    \node[fill=black,circle,inner sep=1pt,opacity=0.5] at (2,5,0) {};
    \node[fill=black,circle,inner sep=1pt,opacity=0.5] at (2,5,2) {};
    \node[fill=black,circle,inner sep=1pt,opacity=0.5] at (2,5,-2) {};
    \node[fill=black,circle,inner sep=1pt,opacity=0.5] at (-2,5,0) {};
    \node[fill=black,circle,inner sep=1pt,opacity=0.5] at (-2,5,2) {};
    \node[fill=black,circle,inner sep=1pt,opacity=0.5] at (-2,5,-2) {};
    \node[fill=black,circle,inner sep=1pt,opacity=0.5] at (4,5,2) {};
    \node[fill=black,circle,inner sep=1pt,opacity=0.5] at (4,5,0) {};
    \node[fill=black,circle,inner sep=1pt,opacity=0.5] at (4,5,-2) {};
    
    \coordinate (z_1) at (0,0,0);
    \coordinate (z_3) at (0,0,4);
    \coordinate (z_2) at (0,5,0);
    \coordinate (z_4) at (4,5,0);
    
    \coordinate (p_pl_1) at ($ (z_1)!1/2!(z_3) $);
    \coordinate (p_mi_1) at ($ (z_2)!1/2!(z_4) $);
    \coordinate (z_1_z_2_half) at ($ (z_1)!1/2!(z_2) $);
    \coordinate (z_2_z_3_half) at ($ (z_2)!1/2!(z_3) $);
    \coordinate (z_3_z_4_half) at ($ (z_3)!1/2!(z_4) $);
    \coordinate (z_4_z_1_half) at ($ (z_4)!1/2!(z_1) $); 
    
    \coordinate (p) at ($ (p_pl_1)!1/2!(p_mi_1) $);
    \coordinate (z_1_pl_1) at ($(p_mi_1) + (z_1) - (p_pl_1) $);
    \coordinate (z_3_pl_1) at ($(p_mi_1) + (z_3) - (p_pl_1) $);
    \coordinate (z_2_mi_1) at ($(p_pl_1) + (z_2) - (p_mi_1) $);
    \coordinate (z_4_mi_1) at ($(p_pl_1) + (z_4) - (p_mi_1) $);
    
    \coordinate (z_1_p) at ($ (z_1)!1/2!(z_1_pl_1) $);
    \coordinate (z_3_p) at ($(z_3)!1/2!(z_3_pl_1)$);
    \coordinate (z_2_p) at ($(z_2)!1/2!(z_2_mi_1)$);
    \coordinate (z_4_p) at ($(z_4)!1/2!(z_4_mi_1)$);
    
    \fill[blue!05,opacity=0.3] (3,0,5) -- (3,0,-3) -- (-3,0,-3) -- (-3,0,5) -- cycle;
    \fill[blue!05,opacity=0.3] (5,5,3) -- (5,5,-3) -- (-3,5,-3) -- (-3,5,3) -- cycle;
    
    \fill[blue!25,opacity=0.3] (z_1) -- (z_2_mi_1) -- (z_3) -- (z_4_mi_1) -- cycle;
    \fill[blue!25,opacity=0.9] (z_1_pl_1) -- (z_2) -- (z_3_pl_1) -- (z_4) -- cycle;
    
    \node[fill=blue,circle,inner sep=1pt] at (z_1) {};
    \draw (z_1) node[above left] {$\vec z_1$};
    
    \node[fill=blue,circle,inner sep=1pt] at (z_3) {};
    \draw (z_3) node[right] {$\vec z_3$};
    
    \node[fill=blue,circle,inner sep=1pt] at (z_2) {};
    \draw (z_2) node[left] {$\vec z_2$};
    
    \node[fill=blue,circle,inner sep=1pt] at (z_4) {};
    \draw (z_4) node[right] {$\vec z_4$};
    
    \node[fill=blue,circle,inner sep=1pt] at (z_1_pl_1) {};
    \draw (z_1_pl_1) node[above] {$\vec z_1^{R}$};
    
    \node[fill=blue,circle,inner sep=1pt] at (z_3_pl_1) {};
    \draw (z_3_pl_1) node[left] {$\vec z_3^{R}$};
    
    \node[fill=blue,circle,inner sep=1pt] at (z_2_mi_1) {};
    \draw (z_2_mi_1) node[left] {$\vec z_2^{Q}$};
    
    \node[fill=blue,circle,inner sep=1pt] at (z_4_mi_1) {};
    \draw (z_4_mi_1) node[right] {$\vec z_4^{Q}$};
    
    \node[fill=white,circle,inner sep=1pt,draw=blue] at (p) {};
    \draw (p) node[left] {$\vec p$};
    
    \node[fill=blue,circle,inner sep=1pt,draw=blue] at (p_pl_1) {};
    \draw (p_pl_1) node[below right] {$\vec{p}^{Q}$};
    
    \node[fill=blue,circle,inner sep=1pt,draw=blue] at (p_mi_1) {};
    \draw (p_mi_1) node[above] {$\vec{p}^{R}$};
    
    \draw[->] [red,very thin] (z_1) -- (z_2);
    \draw[->] [red,very thin] (z_1) -- (p_pl_1);
    \draw[->] [red,very thin] (z_1) -- (p_mi_1);
    
    \draw (5, 0, 0) node[right] {$x$};
    \draw (0, 7, 0) node[left] {$z$};
    \draw (0, 0, 6) node[left] {$y$};
    
    \end{scope}
  \end{tikzpicture}
  \caption{Расположение точек внутри гиперплоскости $S_{1}$ из случая \textbf{Г.1.б} леммы \ref{main_lem_2_2}}
\end{figure}

\textbf{Д}) $\Delta^{Q}_{k} \cap \Z^{4} = \{\vec{z}_{1}, \vec{z}_{3}, \vec{p}^{Q}_{2},\vec{p}^{Q}_{4}\}$. Заметим, что случай $\Delta^{R}_{k} \cap \Z^{4} = \{\vec{z}_{2}, \vec{z}_{4}\}$ с точностью до перестановки индексов невозможен в силу пункта \textbf{А}. Также случай $\Delta^{R}_{k} \cap \Z^{4} = \{\vec{z}_{2}, \vec{z}_{4}, \vec{z}^{R}_{1}, \vec{z}^{R}_{3}\}$ с точностью до перестановки индексов невозможен в силу пункта \textbf{Б}. Случай $\Delta^{R}_{k} \cap \Z^{4} = \{\vec{z}_{2}, \vec{z}_{4}, \vec{p}^{R}\}$ с точностью до перестановки индексов полностью эквивалентен пункту \textbf{В.1}. Кроме того, случай $\Delta^{R}_{k} \cap \Z^{4} = \{\vec{z}_{2}, \vec{z}_{4}, \vec{z}^{R}_{1}, \vec{z}^{R}_{3},  \vec{p}^{R}\}$ с точностью до перестановки индексов невозможен в силу пункта \textbf{Г}. Осталось заметить, что если точки $\vec{p}^{R}_{1}$ и $\vec{p}^{R}_{3}$ принадлежат решетке $\Z^{4}$, то $\vec{p}^{Q} \in \Z^{4}$, чего не может быть. Значит случай  $\Delta^{R}_{k} \cap \Z^{4} = \{\vec{z}_{2}, \vec{z}_{4}, \vec{p}^{R}_{1},\vec{p}^{R}_{3}\}$ также невозможен, что завершает доказательство леммы. 

\end{proof}

\begin{lemma}\label{main_lem_4}
  Пусть $G$ --- собственная циклическая симметрия $\cf(l_1,l_2,l_3, l_4)\in\gA_3$. Тогда существуют $\vec{z}_1$, $\vec{z}_2$, $\vec{z}_3$, $\vec{z}_4$ $\in$ $\Z^4$, такие что
\[G(\vec{z}_{1}) = \vec{z}_{2}, \, G(\vec{z}_{2}) = \vec{z}_{3}, \, G(\vec{z}_{3}) = \vec{z}_{4}, \, G(\vec{z}_{4}) = \vec{z}_{1}\] 
и выполняется хотя бы одно из семи утверждений:

\textup{(1)} утверждение \textup{(4)} леммы \ref{main_lem_2_2};

\textup{(2)} утверждение \textup{(8)} леммы \ref{main_lem_2_2};

\textup{(3)} утверждение \textup{(7)} леммы \ref{main_lem_2_2};

\textup{(4)} утверждение \textup{(3)} леммы \ref{main_lem_2_2};

\textup{(5)} утверждение \textup{(10)} леммы \ref{main_lem_2_2};

\textup{(6)} утверждение \textup{(5)} леммы \ref{main_lem_2_2};

\textup{(7)} утверждение \textup{(1)} леммы \ref{main_lem_2_2}.
\end{lemma}

\begin{proof}
Рассмотрим подпространства $l_{+}$, $l_{-}$ и $L$ из леммы \ref{rational_subspace_4} и положим $S = l_{-} + L$. Обозначим через $S_1$ ближайшую к $S$ рациональную гиперплоскость, параллельную $S$ и не совпадающую с $S$ (любую из двух). Тогда $G(S_1) = S_1$. Также обозначим через $\vec{p}$ точку пересечения гиперплоскости  $S_1$ и $l_{+}$, а через $l$ и $\pi$ прямую и плоскость, проходящие через точку $\vec{p}$ и параллельные $l_{-}$ и $L$ соответственно. Тогда,  $G(\vec{p}) = \vec{p}$, $G(\vec{v} - \vec{p}) = \vec{p} - \vec{v}$  для любого вектора $\vec v \in l$ и  $G^{2}(\vec{v} - \vec{p}) = \vec{p} - \vec{v}$  для любого вектора $\vec v \in \pi$ в силу леммы \ref{rational_subspace_4}. 

Поскольку подпространство $L$ не содержит собственных для $G$ одномерных подпространств, то для произвольной точки $\vec{v} \in \pi \setminus l$ четырехугольник $\textup{conv}\big(\vec{v}, G(\vec{v}), G^{2}(\vec{v}), G^{3}(\vec{v})\big)$ является параллелограммом, диагонали которого пересекаются в точке $\vec{p} = \frac{1}{2}\big(\vec{v} + G^{2}(\vec{v})\big) = \frac{1}{2}\big(G(\vec{v}) + G^{3}(\vec{v})\big)$.

Обозначим через $Q$ рациональную плоскость ближайшую к $\pi$, лежащую в гиперплоскости $S_1$, параллельную $\pi$ и не совпадающую с $\pi$. Поскольку $G(\vec{v} - \vec{p}) = \vec{p} - \vec{v}$  для любого вектора $\vec v \in l$, то $R = G(Q)$ и $Q$ --- рациональные плоскости ближайшие к $\pi$ и равноудаленные от нее, лежащие в гиперплоскости $S_1$ по разные стороны от $\pi$, параллельные $\pi$ и не совпадающие с $\pi$. Положим $\vec{p}^{Q} = Q \cap l$ и $\vec{p}^{R} = R \cap l$. Построим точки $\vec{z}_{1}$, $\vec{z}_{2}$, $\vec{z}_{3}$, $\vec{z}_{4}$ при помощи следующей итерационной процедуры. Возьмем произвольную целочисленную точку $\vec{v}_{1,1} \in Q \setminus l$. Введем обозначения $\vec{v}_{1,2} = G(\vec{v}_{1,1})$, $\vec{v}_{1,3} = G^{2}(\vec{v}_{1,1})$, $\vec{v}_{1,4} = G^{3}(\vec{v}_{1,1})$. Пусть точки  $\vec{v}^{\pi}_{1,1}$, $\vec{v}^{\pi}_{1,2}$, $\vec{v}^{\pi}_{1,3}$, $\vec{v}^{\pi}_{1,4}$ --- проекции параллельные $l$ на плоскость $\pi$ точек $\vec{v}_{1,1}$, $\vec{v}_{1,2}$, $\vec{v}_{1,3}$, $\vec{v}_{1,4}$ соответственно. Также обозначим через $\vec{v}^{R}_{1,1}$ и $\vec{v}^{R}_{1,3}$ --- проекции параллельные $l$ на плоскость $R$ точек $\vec{v}_{1,1}$ и $\vec{v}_{1,3}$, а через $\vec{v}^{Q}_{1,2}$ и $\vec{v}^{Q}_{1,4}$ --- проекции параллельные $l$ на плоскость $Q$ точек $\vec{v}_{1,2}$ и $\vec{v}_{1,4}$. По доказанному выше множество $\Delta^{\pi}_{1} = \textup{conv}(\vec{v}^{\pi}_{1,1}, \vec{v}^{\pi}_{1,2}, \vec{v}^{\pi}_{1,3}, \vec{v}^{\pi}_{1,4})$ является параллелограммом, диагонали которого пересекаются в точке $\vec{p} = \frac{1}{4}(\vec{v}_{1,1} + \vec{v}_{1,2} + \vec{v}_{1,3} + \vec{v}_{1,4})$. Таким образом, множества $\Delta^{Q}_{1} = \textup{conv}(\vec{v}_{1,1}, \vec{v}^{Q}_{1,2}, \vec{v}_{1,3}, \vec{v}^{Q}_{1,4})$ и $\Delta^{R}_{1} = \textup{conv}(\vec{v}^{R}_{1,1}, \vec{v}_{1,2}, \vec{v}^{R}_{1,3}, \vec{v}_{1,4})$ также являются параллелограммами. Заметим, что $\vec{p}^{Q} = \frac{1}{2}(\vec{v}_{1,1} + \vec{v}_{1,3})$ и $\vec{p}^{R} = \frac{1}{2}(\vec{v}_{1,2} + \vec{v}_{1,4})$. 

Предположим, мы построили параллелограммы $\Delta^{\pi}_{j}$, $\Delta^{Q}_{j}$ и $\Delta^{R}_{j}$. Если на плоскостях $Q$ и $R$ существует целая точка, не совпадающая с точками $\vec{p}^{Q}$, $\vec{p}^{R}$ и ни с какой из вершин параллелограммов $\Delta^{Q}_{j}$ и $\Delta^{R}_{j}$, при этом лежащая в одном из этих параллелограммов (без ограничения общности, внутри $\Delta^{Q}_{j}$), то обозначим её через $\vec{v}_{j+1,1}$. Введем обозначения $\vec{v}_{j+1,2} = G(\vec{v}_{j+1,1})$, $\vec{v}_{j+1,3} = G^{2}(\vec{v}_{j+1,1})$, $\vec{v}_{j+1,4} = G^{3}(\vec{v}_{j+1,1})$. Пусть точки  $\vec{v}^{\pi}_{j+1,1}$, $\vec{v}^{\pi}_{j+1,2}$, $\vec{v}^{\pi}_{j+1,3}$, $\vec{v}^{\pi}_{j+1,4}$ --- проекции параллельные $l$ на плоскость $\pi$ точек $\vec{v}_{j+1,1}$, $\vec{v}_{j+1,2}$, $\vec{v}_{j+1,3}$, $\vec{v}_{j+1,4}$ соответственно. Также обозначим через $\vec{v}^{R}_{j+1,1}$ и $\vec{v}^{R}_{j+1,3}$ --- проекции параллельные $l$  на плоскость $R$ точек $\vec{v}_{j+1,1}$ и $\vec{v}_{j+1,3}$, а через $\vec{v}^{Q}_{j+1,2}$ и $\vec{v}^{Q}_{j+1,4}$ --- проекции параллельные $l$ на плоскость $Q$ точек $\vec{v}_{j+1,2}$ и $\vec{v}_{j+1,4}$. Определим параллелограммы 
\[\Delta^{\pi}_{j+1} = \textup{conv}(\vec{v}^{\pi}_{j+1,1}, \vec{v}^{\pi}_{j+1,2}, \vec{v}^{\pi}_{j+1,3}, \vec{v}^{\pi}_{j+1,4}),\]
\[\Delta^{Q}_{j+1} = \textup{conv}(\vec{v}_{j+1,1}, \vec{v}^{Q}_{j+1,2}, \vec{v}_{j+1,3}, \vec{v}^{Q}_{j+1,4}),\] 
\[\Delta^{R}_{j+1} = \textup{conv}(\vec{v}^{R}_{j+1,1}, \vec{v}_{j+1,2}, \vec{v}^{R}_{j+1,3}, \vec{v}_{j+1,4}).\] 
При этом $\vec{p} = \frac{1}{4}(\vec{v}_{j+1,1} + \vec{v}_{j+1,2} + \vec{v}_{j+1,3} + \vec{v}_{j+1,4})$, $\vec{p}^{Q} = \frac{1}{2}(\vec{v}_{j+1,1} + \vec{v}_{j+1,3})$ и $\vec{p}^{R} = \frac{1}{2}(\vec{v}_{j+1,2} + \vec{v}_{j+1,4})$. 

Последовательность троек $(\Delta^{\pi}_{j}, \Delta^{Q}_{j}, \Delta^{R}_{j})$ конечна. Пусть $(\Delta^{\pi}_{k}, \Delta^{Q}_{k}, \Delta^{R}_{k})$ --- последний её элемент. Положим $\vec{z}_{1} = \vec{v}_{k,1}$, $\vec{z}_{2} = \vec{v}_{k,2}$, $\vec{z}_{3} = \vec{v}_{k,3}$, $\vec{z}_{4} = \vec{v}_{k,4}$, $\vec{z}^{\pi}_{1} = \vec{v}^{\pi}_{k,1}$, $\vec{z}^{\pi}_{2} = \vec{v}^{\pi}_{k,2}$, $\vec{z}^{\pi}_{3} = \vec{v}^{\pi}_{k,3}$, $\vec{z}^{\pi}_{4} = \vec{v}^{\pi}_{k,4}$, $\vec{z}^{R}_{1} = \vec{v}^{R}_{k,1}$, $\vec{z}^{R}_{3} = \vec{v}^{R}_{k,3}$, $\vec{z}^{Q}_{2} = \vec{v}^{Q}_{k,2}$, $\vec{z}^{Q}_{4} = \vec{v}^{Q}_{k,4}$.

Обозначим через $\hat{l}^{1}$ и $\hat{l}^{2}$ такие одномерные рациональные подпространства, что $\hat{l}^{1} + \hat{l}^{2} = L$. Заметим, что подпространства $l^{1}_{+} = l_{+}$, $l^{2}_{+} = l_{-}$, $l^{1}_{-} = \hat{l}^{1}$, $l^{2}_{-} = \hat{l}^{2}$ удовлетворяют условиям леммы \ref{rational_subspace_2_2} для оператора $G^{2}$. Таким образом, пара точек $(\vec{z}_{1}, \vec{z}_{2})$ является допустимой (см. доказательство леммы \ref{main_lem_2_2}) для оператора $G^{2}$ и цепной дроби $\cf(l_1,l_2,l_3, l_4)\in\gA_3$, поскольку 
\[(\Delta^{Q}_{k} \cup \Delta^{R}_{k}) \cap \Z^{4} \subset \{\vec{z}_{1}, \vec{z}^{Q}_{2}, \vec{z}_{3}, \vec{z}^{Q}_{4}, \vec{z}^{R}_{1} , \vec{z}_{2}, \vec{z}^{R}_{3}, \vec{z}_{4}, \vec{p}_{Q}, \vec{p}_{R}\}.\]
Из этого следует, что должно выполняться хотя бы одно из одиннадцати утверждений \textup{(1)} - \textup{(11)} леммы \ref{main_lem_2_2}.

Заметим, что случай \textbf{А.2.б} из доказательства леммы \ref{main_lem_2_2} невозможен, поскольку, в этом случае, $G(\vec{p}^{R}) =  \vec{p}^{Q}$, а значит, $\vec{p}^{Q} \in \Z^{4}$, что противоречит расположению точек решетки $\Z^{4}$ в этом случае. Случай \textbf{Б.1.а.3} из доказательства леммы \ref{main_lem_2_2} невозможен, поскольку, в этом случае, $G(\frac{\vec{z}^{\pi}_{1} + \vec{z}^{\pi}_{2}}{2}) =  \frac{\vec{z}^{\pi}_{2} + \vec{z}^{\pi}_{3}}{2}$, а значит, $\frac{\vec{z}^{\pi}_{2} + \vec{z}^{\pi}_{3}}{2} \in \Z^{4}$, что противоречит расположению точек решетки $\Z^{4}$ в этом случае. Случай \textbf{В.1.б} из доказательства леммы \ref{main_lem_2_2} невозможен, поскольку, в этом случае, $G(\frac{\vec{z}^{R}_{1} + \vec{p}_{R}}{2}) =  \frac{\vec{z}^{Q}_{2} + \vec{p}_{Q}}{2}$, а значит, $\frac{\vec{z}^{Q}_{2} + \vec{p}_{Q}}{2} \in \Z^{4}$, что противоречит расположению точек решетки $\Z^{4}$ в этом случае. Случай \textbf{Г.1.а.3} из доказательства леммы \ref{main_lem_2_2} невозможен, поскольку, в этом случае, $G(\frac{\vec{z}^{\pi}_{1} + \vec{p}}{2}) =  \frac{\vec{z}^{\pi}_{2} + \vec{p}}{2}$, а значит, $\frac{\vec{z}^{\pi}_{2} + \vec{p}}{2} \in \Z^{4}$ и $\frac{\vec{z}_{1} + \vec{z}^{Q}_{2}}{2} = \vec{z}_{1} + (\frac{\vec{z}^{\pi}_{2} + \vec{p}}{2} - \frac{\vec{z}^{\pi}_{1} + \vec{p}}{2}) \in \Z^{4}$, что противоречит расположению точек решетки $\Z^{4}$ в этом случае. Итак, мы показали, что должно выполняться хотя бы одно из семи утверждений \textup{(4)}, \textup{(8)}, \textup{(7)}, \textup{(3)}, \textup{(10)}, \textup{(5)}, \textup{(1)} из формулировки леммы \ref{main_lem_2_2}. 

\end{proof}

\section{Матрицы собственных симметрий}\label{matrix_and_algebraicity}

Напомним, что если задана дробь $\cf(l_1,l_2,l_3,l_4)=\cf(A)\in\gA_3$, будем считать, что подпространство $l_1$ порождается вектором $\vec l_1=(1,\alpha,\beta,\gamma)$. Тогда из предложения \ref{prop:more_than_pelle_n_dim} следует, что числа $1,\alpha,\beta, \gamma$ образуют базис поля $K=\Q(\alpha,\beta, \gamma)$ над $\Q$ и каждое $l_i$ порождается вектором $\vec l_i=(1,\sigma_i(\alpha),\sigma_i(\beta), \sigma_i(\gamma))$, где $\sigma_1(=\id),\sigma_2,\sigma_3, \sigma_4$ --- все вложения $K$ в $\R$. То есть, если через $\big(\vec l_1,\vec l_2,\vec l_3,\vec l_4\big)$ обозначить матрицу со столбцами $\vec l_1,\vec l_2,\vec l_3,\vec l_4$, получим
 \[
  \big(\vec l_1,\vec l_2,\vec l_3,\vec l_4\big)=
  \begin{pmatrix}
    1 & 1 & 1 & 1 \\
    \alpha & \sigma_2(\alpha) & \sigma_3(\alpha) & \sigma_4(\alpha) \\
    \beta & \sigma_2(\beta) & \sigma_3(\beta) & \sigma_4(\beta) \\
    \gamma & \sigma_2(\gamma) & \sigma_3(\gamma) & \sigma_4(\gamma)
  \end{pmatrix}.
\]

Мы будем обозначать через $\widetilde{\gA_{3}}$ множество всех трехмерных алгебраических цепных дробей, для которых 
\[\sigma_3(K) = K, \quad \sigma_4(K) = \sigma_2(K), \quad \sigma^{2}_{3} = \textup{id}, \quad \sigma_{4} = \sigma_{2}\sigma_{3}.\]
 
Для каждого $i=1, 2, \ldots, 10$ определим $\widetilde{\mathbf{CF}_{i}}$ как класс дробей из $\widetilde{\gA_{3}}$, удовлетворяющих паре соотношений $\gR_{i}$, где
  
  $\gR_{1}$: $\beta + \sigma_{3}(\beta)= -\big(\alpha + \sigma_{3}(\alpha)\big),\ \gamma =   \sigma_{3}(\alpha)$;
  
  $\gR_{2}$: $\beta + \sigma_{3}(\beta)= 1 - \big(\alpha + \sigma_{3}(\alpha)\big) ,\ \gamma = \sigma_{3}(\alpha)$;
  
  $\gR_{3}$: $\beta + \sigma_{3}(\beta)= 2 - \big(\alpha + \sigma_{3}(\alpha)\big) ,\ \gamma = \sigma_{3}(\alpha)$;
  
  $\gR_{4}$: $\beta + \sigma_{3}(\beta)= -\big(\alpha + \sigma_{3}(\alpha)\big),\ \gamma =  \frac{\alpha + \sigma_{3}(\alpha)}{2}$;
  
  $\gR_{5}$: $\beta + \sigma_{3}(\beta)= 2 -\big(\alpha + \sigma_{3}(\alpha)\big),\ \gamma =   \frac{\alpha + \sigma_{3}(\alpha)}{2}$;
  
  $\gR_{6}$: $\beta + \sigma_{3}(\beta)= -\big(\alpha + \sigma_{3}(\alpha)\big),\ \gamma =  \frac{\alpha + \sigma_{3}(\alpha) + 1}{2}$;
  
  $\gR_{7}$: $\beta + \sigma_{3}(\beta)= 2 -\big(\alpha + \sigma_{3}(\alpha)\big),\ \gamma =  \frac{\alpha + \sigma_{3}(\alpha) + 1}{2}$;
  
  $\gR_{8}$: $\beta + \sigma_{3}(\beta)= 1 - \frac{\alpha + \sigma_{3}(\alpha)}{2},\ \gamma =  \frac{\alpha + \sigma_{3}(\alpha)}{2}$;
  
  $\gR_{9}$: $\beta + \sigma_{3}(\beta)= 2 - \frac{\alpha + \sigma_{3}(\alpha)}{2},\ \gamma =  \frac{\alpha + \sigma_{3}(\alpha)}{2}$;
  
  $\gR_{10}$: $\beta + \sigma_{3}(\beta)= 2 - \frac{\alpha + \sigma_{3}(\alpha)}{2},\ \gamma =  \frac{\sigma_{3}(\alpha) - \alpha}{4}$.

Покажем, что все дроби из классов $ \widetilde{\mathbf{CF}_{i}}$, палиндромичны для каждого $i=1, 2, \ldots, 10$. Положим $\widetilde{G_{1}}, \widetilde{G_{2}}, \ldots, \widetilde{G_{10}}$ равными соответственно матрицам 

\[
  \left(\begin{smallmatrix}
    \phantom{-}1 & \phantom{-}0 & \phantom{-}0 & \phantom{-}0 \\
    \phantom{-}0 & \phantom{-}0 & \phantom{-}0 & \phantom{-}1 \\
    \phantom{-}0 &                  -1 &                   -1 &                  -1 \\
    \phantom{-}0 & \phantom{-}1 & \phantom{-}0 & \phantom{-}0
  \end{smallmatrix}\right),
  \left(\begin{smallmatrix}
    \phantom{-}1 & \phantom{-}0 & \phantom{-}0 & \phantom{-}0 \\
    \phantom{-}0 & \phantom{-}0 & \phantom{-}0 & \phantom{-}1 \\
    \phantom{-}1 &                  -1 &                   -1 &                  -1 \\
    \phantom{-}0 & \phantom{-}1 & \phantom{-}0 & \phantom{-}0
  \end{smallmatrix}\right),
  \left(\begin{smallmatrix}
    1 & \phantom{-}0 & \phantom{-}0 & \phantom{-}0 \\
    0 & \phantom{-}0 & \phantom{-}0 & \phantom{-}1 \\
    2 &                  -1 &                   -1 &                  -1 \\
    0 & \phantom{-}1 & \phantom{-}0 & \phantom{-}0
  \end{smallmatrix}\right),
  \left(\begin{smallmatrix}
    1 & \phantom{-}0 & \phantom{-}0 & \phantom{-}0 \\
    0 &                  -1 & \phantom{-}0 & \phantom{-}2 \\
    0 & \phantom{-}0 &                   -1 &                  -2 \\
    0 & \phantom{-}0 & \phantom{-}0 & \phantom{-}1
  \end{smallmatrix}\right),
  \left(\begin{smallmatrix}
    1 & \phantom{-}0 & \phantom{-}0 & \phantom{-}0 \\
    0 &                  -1 & \phantom{-}0 & \phantom{-}2 \\
    2 & \phantom{-}0 &                   -1 &                  -2 \\
    0 & \phantom{-}0 & \phantom{-}0 & \phantom{-}1
  \end{smallmatrix}\right),
 \]
 \[
  \left(\begin{smallmatrix}
    \phantom{-}1 & \phantom{-}0 & \phantom{-}0 & \phantom{-}0 \\
                     -1 &                   -1 & \phantom{-}0 & \phantom{-}2 \\
    \phantom{-}1 & \phantom{-}0 &                   -1 &                  -2 \\
    \phantom{-}0 & \phantom{-}0 & \phantom{-}0 & \phantom{-}1
  \end{smallmatrix}\right),
  \left(\begin{smallmatrix}
    \phantom{-}1 & \phantom{-}0 & \phantom{-}0 & \phantom{-}0 \\
                     -1 &                  -1 & \phantom{-}0 & \phantom{-}2 \\
    \phantom{-}3 & \phantom{-}0 &                   -1 &                  -2 \\
    \phantom{-}0 & \phantom{-}0 & \phantom{-}0 & \phantom{-}1
  \end{smallmatrix}\right),
  \left(\begin{smallmatrix}
    1 & \phantom{-}0 & \phantom{-}0 & \phantom{-}0 \\
    0 &                  -1 & \phantom{-}0 & \phantom{-}2 \\
    1 & \phantom{-}0 &                   -1 &                  -1 \\
    0 & \phantom{-}0 & \phantom{-}0 & \phantom{-}1
  \end{smallmatrix}\right),
  \left(\begin{smallmatrix}
    1 & \phantom{-}0 & \phantom{-}0 & \phantom{-}0 \\
    0 &                  -1 & \phantom{-}0 & \phantom{-}2 \\
    2 & \phantom{-}0 &                   -1 &                  -1 \\
    0 & \phantom{-}0 & \phantom{-}0 & \phantom{-}1
  \end{smallmatrix}\right),
  \left(\begin{smallmatrix}
    1 & \phantom{-}0 & \phantom{-}0 & \phantom{-}0 \\
    0 & \phantom{-}1 & \phantom{-}0 & \phantom{-}4 \\
    2 &                  -1 &                   -1 &                  -2 \\
    0 & \phantom{-}0 & \phantom{-}0 &                  -1
  \end{smallmatrix}\right).
\]

\begin{lemma}\label{oper_eq_4d_2}
  Пусть $\cf(l_1,l_2,l_3,l_4)\in\gA_3$ и $i\in\{1,2, \ldots,10\}$. Тогда цепная дробь $\cf(l_1,l_2,l_3,l_4)$ принадлежит классу $\widetilde{\mathbf{CF}_{i}}$ в том и только в том случае, если $\widetilde{G_{i}}$ --- её собственная симметрия и $\textup{ord}({\sigma_{\widetilde{G_{i}}}}) = 2$.
\end{lemma}

\begin{proof}
  В силу леммы \ref{prod_lemm} оператор $G\in\Gl_4(\Z)$ является собственной симметрией дроби $\cf(l_1,l_2,l_3,l_4)$ и $\textup{ord}({\sigma_{G}}) = 2$ тогда и только тогда, когда с точностью до перестановки индексов существуют такие действительные  числа $\mu_1,\mu_2,\mu_3,\mu_4$, что $G\big(\vec l_1,\vec l_2,\vec l_3,\vec l_4\big)=\big(\mu_3\vec l_3,\mu_4\vec l_4,\mu_1\vec l_1,\mu_2\vec l_2\big)$ и $\mu_1\mu_3 = \mu_2\mu_4 = 1$.
  
Пусть $\cf(l_1, l_2, l_3, l_4) \in \widetilde{\mathbf{CF}_{1}}$. Заметим, что $\sigma_{2} = \sigma_{2}\sigma^{2}_{3} = \sigma_{4}\sigma_{3}$, $\sigma_2(\gamma) = \sigma_{2}\sigma_{3}(\alpha)$, $\sigma_3(\gamma) = \sigma^{2}_{3}(\alpha) = \alpha$, $\sigma_4(\gamma) = \sigma_{4}\sigma_{3}(\alpha) = \sigma_{2}(\alpha)$, $\sigma_{2}(\beta) + \sigma_{2}\sigma_{3}(\beta) = \sigma_{2}\big(\beta + \sigma_{3}(\beta)\big) =  \sigma_{2}\Big( - \big(\alpha + \sigma_{3}(\alpha)\big)\Big) = - \big(\sigma_{2}(\alpha) + \sigma_{2}\sigma_{3}(\alpha)\big)$ и $\sigma_{4}(\beta) = \sigma_{2}\sigma_{3}(\beta)$. Тогда

\[\widetilde{G_{1}}\vec{l}_1 =
  \begin{pmatrix}
  1\\
  \sigma_{3}(\alpha)\\
  - \beta - \big(\alpha + \sigma_{3}(\alpha)\big)\\
  \alpha
   \end{pmatrix},
    \quad
    \widetilde{G_{1}}\vec{l}_2 =
  \begin{pmatrix}
  1\\
  \sigma_{2}\sigma_{3}(\alpha)\\
  - \sigma_{2}(\beta) - \big(\sigma_{2}(\alpha) + \sigma_{2}\sigma_{3}(\alpha)\big)\\
  \sigma_{2}(\alpha)
   \end{pmatrix},
\]
\[\widetilde{G_{1}}\vec{l}_3 =
  \begin{pmatrix}
  1\\
  \alpha\\
  - \sigma_{3}(\beta) - \big(\sigma_{3}(\alpha) + \alpha\big)\\
  \sigma_{3}(\alpha)
   \end{pmatrix},
    \quad
    \widetilde{G_{1}}\vec{l}_4 =
  \begin{pmatrix}
  1\\
  \sigma_{2}(\alpha)\\
  - \sigma_{2}\sigma_{3}(\beta) - \big(\sigma_{2}\sigma_{3}(\alpha) + \sigma_{2}(\alpha)\big)\\
  \sigma_{2}\sigma_{3}(\alpha)
   \end{pmatrix},
\]
то есть $\widetilde{G_{1}} \big(\vec{l}_1, \vec{l}_2, \vec{l}_3, \vec{l}_4\big) = \big(\vec{l}_3, \vec{l}_4, \vec{l}_1, \vec{l}_2\big)$. Следовательно, $\widetilde{G_{1}}$ ---  собственная симметрия цепной дроби $\cf(l_1,l_2,l_3,l_4)$ и $\textup{ord}({\sigma_{\widetilde{G_{1}}}}) = 2$. Обратно, предположим, что $\widetilde{G_{1}}$ ---  собственная симметрия $\cf(l_1,l_2,l_3,l_4)$ и $\textup{ord}({\sigma_{\widetilde{G_{1}}}}) = 2$. Тогда существует такое $\mu_3$, что с точностью до перестановки индексов
  \[
    \widetilde{G_{1}}\vec{l}_1=
    \begin{pmatrix}
      1  \\
      \gamma \\
      - \beta - (\alpha + \gamma) \\
      \alpha
    \end{pmatrix}
    =\mu_3
    \begin{pmatrix}
      1  \\
      \sigma_{3}(\alpha) \\
      \sigma_{3}(\beta) \\
      \sigma_{3}(\gamma)
    \end{pmatrix}, 
  \]
  откуда $\mu_3 = 1$,  $\gamma =\sigma_{3}(\alpha)$, $\sigma^{2}_3(\alpha)=\sigma_{3}(\gamma) = \alpha$, $\sigma^{2}_{3}(\gamma) = \sigma_{3}(\alpha) = \gamma$, $\beta + \sigma_{3}(\beta) = - \big(\alpha + \sigma_{3}(\alpha)\big)$, $\sigma^{2}_3(\beta) = - \sigma_3(\beta)  - \big(\sigma_3(\alpha) + \sigma^{2}_{3}(\alpha)\big) = - \sigma_3(\beta) - \big(\alpha + \sigma_3(\alpha)\big) = \beta$. Существует такое $\mu_4$, что
  \[ 
    \widetilde{G_{1}}\vec{l}_2=
    \begin{pmatrix}
      1  \\
      \sigma_{2}(\gamma) \\
      -  \sigma_{2}(\beta) - \big(\sigma_{2}(\alpha) +  \sigma_{2}(\gamma)\big) \\
       \sigma_{2}(\alpha)
    \end{pmatrix}
    =\mu_4
    \begin{pmatrix}
      1  \\
      \sigma_{4}(\alpha) \\
      \sigma_{4}(\beta) \\
      \sigma_{4}(\gamma)
    \end{pmatrix},
  \]
   откуда $\mu_4= 1$, $\sigma_{4}(\alpha) = \sigma_{2}(\gamma) = \sigma_{2}\sigma_{3}(\alpha)$, $\sigma_{4}(\gamma) = \sigma_{2}(\alpha) = \sigma_{2}\sigma_{3}(\gamma)$, $\sigma_{4}(\beta) = -  \sigma_{2}(\beta) - \big(\sigma_{2}(\alpha) +  \sigma_{2}(\gamma)\big) = \sigma_{2}\big(-  \beta - (\alpha + \gamma)\big) = \sigma_{2}\sigma_{3}(\beta)$. Стало быть, $\cf(l_1, l_2, l_3, l_4) \in  \widetilde{\mathbf{CF}_{1}}$, так как числа $1,\alpha,\beta, \gamma$ образуют базис поля $K=\Q(\alpha,\beta, \gamma)$.

  Пусть $\cf(l_1, l_2, l_3, l_4) \in  \widetilde{\mathbf{CF}_{2}}$. Заметим, что $\sigma_{2} = \sigma_{2}\sigma^{2}_{3} = \sigma_{4}\sigma_{3}$, $\sigma_2(\gamma) = \sigma_{2}\sigma_{3}(\alpha)$, $\sigma_3(\gamma) = \sigma^{2}_{3}(\alpha) = \alpha$, $\sigma_4(\gamma) = \sigma_{4}\sigma_{3}(\alpha) = \sigma_{2}(\alpha)$, $\sigma_{2}(\beta) + \sigma_{2}\sigma_{3}(\beta) = \sigma_{2}\big(\beta + \sigma_{3}(\beta)\big) =  \sigma_{2}\Big(1 - \big(\alpha + \sigma_{3}(\alpha)\big)\Big) = 1 - \big(\sigma_{2}(\alpha) + \sigma_{2}\sigma_{3}(\alpha)\big)$ и $\sigma_{4}(\beta) = \sigma_{2}\sigma_{3}(\beta)$. Тогда

\[\widetilde{G_{2}}\vec{l}_1 =
  \begin{pmatrix}
  1\\
  \sigma_{3}(\alpha)\\
  1 - \beta - \big(\alpha + \sigma_{3}(\alpha)\big)\\
  \alpha
   \end{pmatrix},
    \quad
    \widetilde{G_{2}}\vec{l}_2 =
  \begin{pmatrix}
  1\\
  \sigma_{2}\sigma_{3}(\alpha)\\
  1 - \sigma_{2}(\beta) - \big(\sigma_{2}(\alpha) + \sigma_{2}\sigma_{3}(\alpha)\big)\\
  \sigma_{2}(\alpha)
   \end{pmatrix},
\]
\[\widetilde{G_{2}}\vec{l}_3 =
  \begin{pmatrix}
  1\\
  \alpha\\
  1 - \sigma_{3}(\beta) - \big(\sigma_{3}(\alpha) + \alpha\big)\\
  \sigma_{3}(\alpha)
   \end{pmatrix},
    \quad
    \widetilde{G_{2}}\vec{l}_4 =
  \begin{pmatrix}
  1\\
  \sigma_{2}(\alpha)\\
  1 - \sigma_{2}\sigma_{3}(\beta) - \big(\sigma_{2}\sigma_{3}(\alpha) + \sigma_{2}(\alpha)\big)\\
  \sigma_{2}\sigma_{3}(\alpha)
   \end{pmatrix},
\]
то есть $\widetilde{G_{2}} \big(\vec{l}_1, \vec{l}_2, \vec{l}_3, \vec{l}_4\big) = \big(\vec{l}_3, \vec{l}_4, \vec{l}_1, \vec{l}_2\big)$. Следовательно, $\widetilde{G_{2}}$ ---  собственная симметрия цепной дроби $\cf(l_1,l_2,l_3,l_4)$ и $\textup{ord}({\sigma_{\widetilde{G_{2}}}}) = 2$. Обратно, предположим, что $\widetilde{G_{2}}$ ---  собственная симметрия $\cf(l_1,l_2,l_3,l_4)$ и $\textup{ord}({\sigma_{\widetilde{G_{2}}}}) = 2$. Тогда существует такое $\mu_3$, что с точностью до перестановки индексов
  \[
    \widetilde{G_{2}}\vec{l}_1=
    \begin{pmatrix}
      1  \\
      \gamma \\
      1 - \beta - (\alpha + \gamma) \\
      \alpha
    \end{pmatrix}
    =\mu_3
    \begin{pmatrix}
      1  \\
      \sigma_{3}(\alpha) \\
      \sigma_{3}(\beta) \\
      \sigma_{3}(\gamma)
    \end{pmatrix}, 
  \]
  откуда $\mu_3 = 1$,  $\gamma =\sigma_{3}(\alpha)$, $\sigma^{2}_3(\alpha)=\sigma_{3}(\gamma) = \alpha$, $\sigma^{2}_{3}(\gamma) = \sigma_{3}(\alpha) = \gamma$, $\beta + \sigma_{3}(\beta) = 1 - \big(\alpha + \sigma_{3}(\alpha)\big)$, $\sigma^{2}_3(\beta) = - \sigma_3(\beta) + \sigma_3(1) - \big(\sigma_3(\alpha) + \sigma^{2}_{3}(\alpha)\big) = - \sigma_3(\beta) + 1 - \big(\alpha + \sigma_3(\alpha)\big) = \beta$. Существует такое $\mu_4$, что
  \[ 
    \widetilde{G_{2}}\vec{l}_2=
    \begin{pmatrix}
      1  \\
      \sigma_{2}(\gamma) \\
      1 -  \sigma_{2}(\beta) - \big(\sigma_{2}(\alpha) +  \sigma_{2}(\gamma)\big) \\
       \sigma_{2}(\alpha)
    \end{pmatrix}
    =\mu_4
    \begin{pmatrix}
      1  \\
      \sigma_{4}(\alpha) \\
      \sigma_{4}(\beta) \\
      \sigma_{4}(\gamma)
    \end{pmatrix},
  \]
   откуда $\mu_4= 1$, $\sigma_{4}(\alpha) = \sigma_{2}(\gamma) = \sigma_{2}\sigma_{3}(\alpha)$, $\sigma_{4}(\gamma) = \sigma_{2}(\alpha) = \sigma_{2}\sigma_{3}(\gamma)$, $\sigma_{4}(\beta) = 1 -  \sigma_{2}(\beta) - \big(\sigma_{2}(\alpha) +  \sigma_{2}(\gamma)\big) = \sigma_{2}\big(1 -  \beta - (\alpha + \gamma)\big) = \sigma_{2}\sigma_{3}(\beta)$. Стало быть, $\cf(l_1, l_2, l_3, l_4) \in  \widetilde{\mathbf{CF}_{2}}$, так как числа $1,\alpha,\beta, \gamma$ образуют базис поля $K=\Q(\alpha,\beta, \gamma)$.

  Пусть $\cf(l_1, l_2, l_3, l_4) \in  \widetilde{\mathbf{CF}_{3}}$. Заметим, что $\sigma_{2} = \sigma_{2}\sigma^{2}_{3} = \sigma_{4}\sigma_{3}$, $\sigma_2(\gamma) = \sigma_{2}\sigma_{3}(\alpha)$, $\sigma_3(\gamma) = \sigma^{2}_{3}(\alpha) = \alpha$, $\sigma_4(\gamma) = \sigma_{4}\sigma_{3}(\alpha) = \sigma_{2}(\alpha)$, $\sigma_{2}(\beta) + \sigma_{2}\sigma_{3}(\beta) = \sigma_{2}\big(\beta + \sigma_{3}(\beta)\big) =  \sigma_{2}\Big(2 - \big(\alpha + \sigma_{3}(\alpha)\big)\Big) = 2 - \big(\sigma_{2}(\alpha) + \sigma_{2}\sigma_{3}(\alpha)\big)$ и $\sigma_{4}(\beta) = \sigma_{2}\sigma_{3}(\beta)$. Тогда

\[\widetilde{G_{3}}\vec{l}_1 =
  \begin{pmatrix}
  1\\
  \sigma_{3}(\alpha)\\
  2 - \beta - \big(\alpha + \sigma_{3}(\alpha)\big)\\
  \alpha
   \end{pmatrix},
    \quad
    \widetilde{G_{3}}\vec{l}_2 =
  \begin{pmatrix}
  1\\
  \sigma_{2}\sigma_{3}(\alpha)\\
  2 - \sigma_{2}(\beta) - \big(\sigma_{2}(\alpha) + \sigma_{2}\sigma_{3}(\alpha)\big)\\
  \sigma_{2}(\alpha)
   \end{pmatrix},
\]
\[\widetilde{G_{3}}\vec{l}_3 =
  \begin{pmatrix}
  1\\
  \alpha\\
  2 - \sigma_{3}(\beta) - \big(\sigma_{3}(\alpha) + \alpha\big)\\
  \sigma_{3}(\alpha)
   \end{pmatrix},
    \quad
    \widetilde{G_{3}}\vec{l}_4 =
  \begin{pmatrix}
  1\\
  \sigma_{2}(\alpha)\\
  2 - \sigma_{2}\sigma_{3}(\beta) - \big(\sigma_{2}\sigma_{3}(\alpha) + \sigma_{2}(\alpha)\big)\\
  \sigma_{2}\sigma_{3}(\alpha)
   \end{pmatrix},
\]
то есть $\widetilde{G_{3}} \big(\vec{l}_1, \vec{l}_2, \vec{l}_3, \vec{l}_4\big) = \big(\vec{l}_3, \vec{l}_4, \vec{l}_1, \vec{l}_2\big)$. Следовательно, $\widetilde{G_{3}}$ ---  собственная симметрия цепной дроби $\cf(l_1,l_2,l_3,l_4)$ и $\textup{ord}({\sigma_{\widetilde{G_{3}}}}) = 2$. Обратно, предположим, что $\widetilde{G_{3}}$ ---  собственная симметрия $\cf(l_1,l_2,l_3,l_4)$ и $\textup{ord}({\sigma_{\widetilde{G_{3}}}}) = 2$. Тогда существует такое $\mu_3$, что с точностью до перестановки индексов
  \[
    \widetilde{G_{3}}\vec{l}_1=
    \begin{pmatrix}
      1  \\
      \gamma \\
      2 - \beta - (\alpha + \gamma) \\
      \alpha
    \end{pmatrix}
    =\mu_3
    \begin{pmatrix}
      1  \\
      \sigma_{3}(\alpha) \\
      \sigma_{3}(\beta) \\
      \sigma_{3}(\gamma)
    \end{pmatrix}, 
  \]
  откуда $\mu_3 = 1$,  $\gamma =\sigma_{3}(\alpha)$, $\sigma^{2}_3(\alpha)=\sigma_{3}(\gamma) = \alpha$, $\sigma^{2}_{3}(\gamma) = \sigma_{3}(\alpha) = \gamma$, $\beta + \sigma_{3}(\beta) = 2 - \big(\alpha + \sigma_{3}(\alpha)\big)$, $\sigma^{2}_3(\beta) = - \sigma_3(\beta) + \sigma_3(2) - \big(\sigma_3(\alpha) + \sigma^{2}_{3}(\alpha)\big) = - \sigma_3(\beta) + 2 - \big(\alpha + \sigma_3(\alpha)\big) = \beta$. Существует такое $\mu_4$, что
  \[ 
    \widetilde{G_{3}}\vec{l}_2=
    \begin{pmatrix}
      1  \\
      \sigma_{2}(\gamma) \\
      2 -  \sigma_{2}(\beta) - \big(\sigma_{2}(\alpha) +  \sigma_{2}(\gamma)\big) \\
       \sigma_{2}(\alpha)
    \end{pmatrix}
    =\mu_4
    \begin{pmatrix}
      1  \\
      \sigma_{4}(\alpha) \\
      \sigma_{4}(\beta) \\
      \sigma_{4}(\gamma)
    \end{pmatrix},
  \]
   откуда $\mu_4= 1$, $\sigma_{4}(\alpha) = \sigma_{2}(\gamma) = \sigma_{2}\sigma_{3}(\alpha)$, $\sigma_{4}(\gamma) = \sigma_{2}(\alpha) = \sigma_{2}\sigma_{3}(\gamma)$, $\sigma_{4}(\beta) = 2 -  \sigma_{2}(\beta) - \big(\sigma_{2}(\alpha) +  \sigma_{2}(\gamma)\big) = \sigma_{2}\big(2 -  \beta - (\alpha + \gamma)\big) = \sigma_{2}\sigma_{3}(\beta)$. Стало быть, $\cf(l_1, l_2, l_3, l_4) \in  \widetilde{\mathbf{CF}_{3}}$, так как числа $1,\alpha,\beta, \gamma$ образуют базис поля $K=\Q(\alpha,\beta, \gamma)$.

  Пусть $\cf(l_1, l_2, l_3, l_4) \in  \widetilde{\mathbf{CF}_{4}}$. Заметим, что $\sigma_{2} = \sigma_{2}\sigma^{2}_{3} = \sigma_{4}\sigma_{3}$, $\sigma_2(\gamma) = \sigma_{2}(\frac{\alpha + \sigma_{3}(\alpha)}{2}) =  \frac{\sigma_{2}(\alpha) + \sigma_{2}\sigma_{3}(\alpha)}{2}$, $\sigma_3(\gamma) = \sigma_3(\frac{\alpha + \sigma_{3}(\alpha)}{2}) = \frac{\sigma_{3}(\alpha) + \alpha}{2}$, $\sigma_4(\gamma) = \sigma_4(\frac{\alpha + \sigma_{3}(\alpha)}{2}) = \frac{\sigma_{2}\sigma_{3}(\alpha) + \sigma_{2}(\alpha)}{2}$, $\sigma_{2}(\beta) + \sigma_{2}\sigma_{3}(\beta) = \sigma_{2}\big(\beta + \sigma_{3}(\beta)\big) =  \sigma_{2}\Big(-\big(\alpha + \sigma_{3}(\alpha)\big)\Big) = -\big(\sigma_{2}(\alpha) + \sigma_{2}\sigma_{3}(\alpha)\big)$ и $\sigma_{4}(\beta) = \sigma_{2}\sigma_{3}(\beta)$. Тогда
  
  \[\widetilde{G_{4}}\vec{l}_1 =
  \begin{pmatrix}
  1\\
  \sigma_{3}(\alpha)\\
  - \beta - \big(\alpha + \sigma_{3}(\alpha)\big)\\
  \frac{1}{2}(\alpha + \sigma_{3}(\alpha))
   \end{pmatrix},
    \quad
    \widetilde{G_{4}}\vec{l}_2 =
  \begin{pmatrix}
  1\\
  \sigma_{2}\sigma_{3}(\alpha)\\
  - \sigma_{2}(\beta) - \big(\sigma_{2}(\alpha) + \sigma_{2}\sigma_{3}(\alpha)\big)\\
  \frac{1}{2}(\sigma_{2}(\alpha) + \sigma_{2}\sigma_{3}(\alpha))
   \end{pmatrix},
\]
\[\widetilde{G_{4}}\vec{l}_3 =
  \begin{pmatrix}
  1\\
  \alpha\\
  - \sigma_{3}(\beta) - \big(\sigma_{3}(\alpha) + \alpha\big)\\
  \frac{1}{2}(\sigma_{3}(\alpha) + \alpha)
   \end{pmatrix},
    \quad
    \widetilde{G_{4}}\vec{l}_4 =
  \begin{pmatrix}
  1\\
  \sigma_{2}(\alpha)\\
  - \sigma_{2}\sigma_{3}(\beta) - \big(\sigma_{2}\sigma_{3}(\alpha) + \sigma_{2}(\alpha)\big)\\
  \frac{1}{2}(\sigma_{2}\sigma_{3}(\alpha) + \sigma_{2}(\alpha))
   \end{pmatrix},
\]
то есть $\widetilde{G_{4}} \big(\vec{l}_1, \vec{l}_2, \vec{l}_3, \vec{l}_4\big) = \big(\vec{l}_3, \vec{l}_4, \vec{l}_1, \vec{l}_2\big)$. Следовательно, $\widetilde{G_{4}}$ ---  собственная симметрия цепной дроби $\cf(l_1,l_2,l_3,l_4)$ и $\textup{ord}({\sigma_{\widetilde{G_{4}}}}) = 2$. Обратно, предположим, что $\widetilde{G_{4}}$ ---  собственная симметрия $\cf(l_1,l_2,l_3,l_4)$ и $\textup{ord}({\sigma_{\widetilde{G_{4}}}}) = 2$. Тогда существует такое $\mu_3$, что с точностью до перестановки индексов
 \[
    \widetilde{G_{4}}\vec{l}_1=
    \begin{pmatrix}
      1  \\
      -\alpha + 2\gamma  \\
      - \beta - 2\gamma \\
      \gamma
    \end{pmatrix}
    =\mu_3
    \begin{pmatrix}
      1  \\
      \sigma_{3}(\alpha) \\
      \sigma_{3}(\beta) \\
      \sigma_{3}(\gamma)
    \end{pmatrix}, 
  \]
откуда $\mu_3 = 1$,  $\gamma =  \frac{\alpha + \sigma_{3}(\alpha)}{2}$, $\sigma^{2}_{3}(\alpha) = 2\sigma_{3}(\gamma) - \sigma_{3}(\alpha) = 2\gamma -  \sigma_{3}(\alpha) = \alpha$, $\sigma^{2}_{3}(\gamma) = \sigma_{3}(\gamma) = \gamma$, $\beta + \sigma_{3}(\beta) = - 2\gamma = -\big(\alpha + \sigma_{3}(\alpha)\big)$, $\sigma^{2}_3(\beta) = - \sigma_3(\beta) - 2\sigma_3(\gamma) = - \sigma_3(\beta) - 2\gamma = \beta$. Существует такое $\mu_4$, что
  \[ 
    \widetilde{G_{4}}\vec{l}_2=
    \begin{pmatrix}
      1  \\
      -\sigma_{2}(\alpha) + 2\sigma_{2}(\gamma)  \\
      -  \sigma_{2}(\beta) - 2\sigma_{2}(\gamma) \\
       \sigma_{2}(\gamma)
    \end{pmatrix}
    =\mu_4
    \begin{pmatrix}
      1  \\
      \sigma_{4}(\alpha) \\
      \sigma_{4}(\beta) \\
      \sigma_{4}(\gamma)
    \end{pmatrix},
  \]
   откуда $\mu_4 = 1$, $\sigma_{4}(\alpha) = \sigma_{2}(2\gamma - \alpha) = \sigma_{2}\sigma_{3}(\alpha)$, $\sigma_{4}(\gamma) = \sigma_{2}(\gamma) = \sigma_{2}\sigma_{3}(\gamma)$, $\sigma_{4}(\beta) = - \sigma_{2}(\beta) -  2\sigma_{2}(\gamma) = \sigma_{2}(-\beta - 2\gamma) = \sigma_{2}\sigma_{3}(\beta)$. Стало быть, $\cf(l_1, l_2, l_3, l_4) \in  \widetilde{\mathbf{CF}_{4}}$, так как числа $1,\alpha,\beta, \gamma$ образуют базис поля $K=\Q(\alpha,\beta, \gamma)$.

  Пусть $\cf(l_1, l_2, l_3, l_4) \in  \widetilde{\mathbf{CF}_{5}}$. Заметим, что $\sigma_{2} = \sigma_{2}\sigma^{2}_{3} = \sigma_{4}\sigma_{3}$, $\sigma_2(\gamma) = \sigma_{2}(\frac{\alpha + \sigma_{3}(\alpha)}{2}) =  \frac{\sigma_{2}(\alpha) + \sigma_{2}\sigma_{3}(\alpha)}{2}$, $\sigma_3(\gamma) = \sigma_3(\frac{\alpha + \sigma_{3}(\alpha)}{2}) = \frac{\sigma_{3}(\alpha) + \alpha}{2}$, $\sigma_4(\gamma) = \sigma_4(\frac{\alpha + \sigma_{3}(\alpha)}{2}) = \frac{\sigma_{2}\sigma_{3}(\alpha) + \sigma_{2}(\alpha)}{2}$, $\sigma_{2}(\beta) + \sigma_{2}\sigma_{3}(\beta) = \sigma_{2}\big(\beta + \sigma_{3}(\beta)\big) =  \sigma_{2}\Big(2 - \big(\alpha + \sigma_{3}(\alpha)\big)\Big) = 2 -\big(\sigma_{2}(\alpha) + \sigma_{2}\sigma_{3}(\alpha)\big)$ и $\sigma_{4}(\beta) = \sigma_{2}\sigma_{3}(\beta)$. Тогда

 \[\widetilde{G_{5}}\vec{l}_1 =
  \begin{pmatrix}
  1\\
  \sigma_{3}(\alpha)\\
  2 - \beta - \big(\alpha + \sigma_{3}(\alpha)\big)\\
  \frac{1}{2}(\alpha + \sigma_{3}(\alpha))
   \end{pmatrix},
    \quad
    \widetilde{G_{5}}\vec{l}_2 =
  \begin{pmatrix}
  1\\
  \sigma_{2}\sigma_{3}(\alpha)\\
  2 - \sigma_{2}(\beta) - \big(\sigma_{2}(\alpha) + \sigma_{2}\sigma_{3}(\alpha)\big)\\
  \frac{1}{2}(\sigma_{2}(\alpha) + \sigma_{2}\sigma_{3}(\alpha))
   \end{pmatrix},
\]
\[\widetilde{G_{5}}\vec{l}_3 =
  \begin{pmatrix}
  1\\
  \alpha\\
  2 - \sigma_{3}(\beta) - \big(\sigma_{3}(\alpha) + \alpha\big)\\
  \frac{1}{2}(\sigma_{3}(\alpha) + \alpha)
   \end{pmatrix},
    \quad
    \widetilde{G_{5}}\vec{l}_4 =
  \begin{pmatrix}
  1\\
  \sigma_{2}(\alpha)\\
  2 - \sigma_{2}\sigma_{3}(\beta) - \big(\sigma_{2}\sigma_{3}(\alpha) + \sigma_{2}(\alpha)\big)\\
  \frac{1}{2}(\sigma_{2}\sigma_{3}(\alpha) + \sigma_{2}(\alpha))
   \end{pmatrix},
\]
то есть $\widetilde{G_{5}} \big(\vec{l}_1, \vec{l}_2, \vec{l}_3, \vec{l}_4\big) = \big(\vec{l}_3, \vec{l}_4, \vec{l}_1, \vec{l}_2\big)$. Следовательно, $\widetilde{G_{5}}$ ---  собственная симметрия цепной дроби $\cf(l_1,l_2,l_3,l_4)$ и $\textup{ord}({\sigma_{\widetilde{G_{5}}}}) = 2$. Обратно, предположим, что $\widetilde{G_{5}}$ ---  собственная симметрия $\cf(l_1,l_2,l_3,l_4)$ и $\textup{ord}({\sigma_{\widetilde{G_{5}}}}) = 2$. Тогда существует такое $\mu_3$, что с точностью до перестановки индексов
 \[
    \widetilde{G_{5}}\vec{l}_1=
    \begin{pmatrix}
      1  \\
      -\alpha + 2\gamma  \\
      2 - \beta - 2\gamma \\
      \gamma
    \end{pmatrix}
    =\mu_3
    \begin{pmatrix}
      1  \\
      \sigma_{3}(\alpha) \\
      \sigma_{3}(\beta) \\
      \sigma_{3}(\gamma)
    \end{pmatrix}, 
  \]
откуда $\mu_3 = 1$,  $\gamma =  \frac{\alpha + \sigma_{3}(\alpha)}{2}$, $\sigma^{2}_{3}(\alpha) = 2\sigma_{3}(\gamma) - \sigma_{3}(\alpha) = 2\gamma -  \sigma_{3}(\alpha) = \alpha$, $\sigma^{2}_{3}(\gamma) = \sigma_{3}(\gamma) = \gamma$, $\beta + \sigma_{3}(\beta) = 2 - 2\gamma = 2 - \big(\alpha + \sigma_{3}(\alpha)\big)$, $\sigma^{2}_3(\beta) = \sigma_3(2) - \sigma_3(\beta) - 2\sigma_3(\gamma) = 2 - \sigma_3(\beta) - 2\gamma = \beta$. Существует такое $\mu_4$, что
  \[ 
    \widetilde{G_{5}}\vec{l}_2=
    \begin{pmatrix}
      1  \\
      -\sigma_{2}(\alpha) + 2\sigma_{2}(\gamma)  \\
     2 -  \sigma_{2}(\beta) - 2\sigma_{2}(\gamma) \\
       \sigma_{2}(\gamma)
    \end{pmatrix}
    =\mu_4
    \begin{pmatrix}
      1  \\
      \sigma_{4}(\alpha) \\
      \sigma_{4}(\beta) \\
      \sigma_{4}(\gamma)
    \end{pmatrix},
  \]
   откуда $\mu_4 = 1$, $\sigma_{4}(\alpha) = \sigma_{2}(2\gamma - \alpha) = \sigma_{2}\sigma_{3}(\alpha)$, $\sigma_{4}(\gamma) = \sigma_{2}(\gamma) = \sigma_{2}\sigma_{3}(\gamma)$, $\sigma_{4}(\beta) = 2 - \sigma_{2}(\beta) -  2\sigma_{2}(\gamma) = \sigma_{2}(2 -\beta - 2\gamma) = \sigma_{2}\sigma_{3}(\beta)$. Стало быть, $\cf(l_1, l_2, l_3, l_4) \in  \widetilde{\mathbf{CF}_{5}}$, так как числа $1,\alpha,\beta, \gamma$ образуют базис поля $K=\Q(\alpha,\beta, \gamma)$.

  Пусть $\cf(l_1, l_2, l_3, l_4) \in  \widetilde{\mathbf{CF}_{6}}$. Заметим, что $\sigma_{2} = \sigma_{2}\sigma^{2}_{3} = \sigma_{4}\sigma_{3}$, $\sigma_2(\gamma) = \sigma_{2}(\frac{\alpha + \sigma_{3}(\alpha) + 1}{2}) =  \frac{\sigma_{2}(\alpha) + \sigma_{2}\sigma_{3}(\alpha) + 1}{2}$, $\sigma_3(\gamma) = \sigma_3(\frac{\alpha + \sigma_{3}(\alpha) + 1}{2}) = \frac{\sigma_{3}(\alpha) + \alpha + 1}{2}$, $\sigma_4(\gamma) = \sigma_4(\frac{\alpha + \sigma_{3}(\alpha) + 1}{2}) = \frac{\sigma_{2}\sigma_{3}(\alpha) + \sigma_{2}(\alpha) + 1}{2}$, $\sigma_{2}(\beta) + \sigma_{2}\sigma_{3}(\beta) = \sigma_{2}\big(\beta + \sigma_{3}(\beta)\big) =  \sigma_{2}\big(- (\alpha + \sigma_{3}(\alpha))\big) = - \big(\sigma_{2}(\alpha) + \sigma_{2}\sigma_{3}(\alpha)\big)$ и $\sigma_{4}(\beta) = \sigma_{2}\sigma_{3}(\beta)$. Тогда
  
  \[\widetilde{G_{6}}\vec{l}_1 =
  \begin{pmatrix}
  1\\
  \sigma_{3}(\alpha)\\
  - \beta - \big(\alpha + \sigma_{3}(\alpha)\big)\\
  \frac{1}{2}(\alpha + \sigma_{3}(\alpha) + 1)
   \end{pmatrix},
    \quad
    \widetilde{G_{6}}\vec{l}_2 =
  \begin{pmatrix}
  1\\
  \sigma_{2}\sigma_{3}(\alpha)\\
  - \sigma_{2}(\beta) - \big(\sigma_{2}(\alpha) + \sigma_{2}\sigma_{3}(\alpha)\big)\\
  \frac{1}{2}(\sigma_{2}(\alpha) + \sigma_{2}\sigma_{3}(\alpha) + 1)
   \end{pmatrix},
\]
\[\widetilde{G_{6}}\vec{l}_3 =
  \begin{pmatrix}
  1\\
  \alpha\\
  - \sigma_{3}(\beta) - \big(\sigma_{3}(\alpha) + \alpha\big)\\
  \frac{1}{2}(\sigma_{3}(\alpha) + \alpha + 1)
   \end{pmatrix},
    \quad
    \widetilde{G_{6}}\vec{l}_4 =
  \begin{pmatrix}
  1\\
  \sigma_{2}(\alpha)\\
  - \sigma_{2}\sigma_{3}(\beta) - \big(\sigma_{2}\sigma_{3}(\alpha) + \sigma_{2}(\alpha)\big)\\
  \frac{1}{2}(\sigma_{2}\sigma_{3}(\alpha) + \sigma_{2}(\alpha) + 1)
   \end{pmatrix},
\]
то есть $\widetilde{G_{6}} \big(\vec{l}_1, \vec{l}_2, \vec{l}_3, \vec{l}_4\big) = \big(\vec{l}_3, \vec{l}_4, \vec{l}_1, \vec{l}_2\big)$. Следовательно, $\widetilde{G_{6}}$ ---  собственная симметрия цепной дроби $\cf(l_1,l_2,l_3,l_4)$ и $\textup{ord}({\sigma_{\widetilde{G_{6}}}}) = 2$. Обратно, предположим, что $\widetilde{G_{6}}$ ---  собственная симметрия $\cf(l_1,l_2,l_3,l_4)$ и $\textup{ord}({\sigma_{\widetilde{G_{6}}}}) = 2$. Тогда существует такое $\mu_3$, что с точностью до перестановки индексов
 \[
    \widetilde{G_{6}}\vec{l}_1=
    \begin{pmatrix}
      1  \\
      -1 -\alpha + 2\gamma  \\
      1 - \beta - 2\gamma \\
      \gamma
    \end{pmatrix}
    =\mu_3
    \begin{pmatrix}
      1  \\
      \sigma_{3}(\alpha) \\
      \sigma_{3}(\beta) \\
      \sigma_{3}(\gamma)
    \end{pmatrix}, 
  \]
откуда $\mu_3 = 1$,  $\gamma =  \frac{\alpha + \sigma_{3}(\alpha) + 1}{2}$, $\sigma^{2}_{3}(\alpha) = 2\sigma_{3}(\gamma) - \sigma_{3}(\alpha) - 1 = 2\gamma -  \sigma_{3}(\alpha) - 1= \alpha$, $\sigma^{2}_{3}(\gamma) = \sigma_{3}(\gamma) = \gamma$, $\beta + \sigma_{3}(\beta) = 1 - 2\gamma = - \big(\alpha + \sigma_{3}(\alpha)\big)$, $\sigma^{2}_3(\beta) = - \sigma_3(\beta) + \sigma_3(1) - 2\sigma_3(\gamma) = - \sigma_3(\beta) + 1 - 2\gamma = \beta$. Существует такое $\mu_4$, что
  \[ 
    \widetilde{G_{6}}\vec{l}_2=
    \begin{pmatrix}
      1  \\
      -1 -\sigma_{2}(\alpha) + 2\sigma_{2}(\gamma)  \\
      1 -  \sigma_{2}(\beta) - 2\sigma_{2}(\gamma) \\
       \sigma_{2}(\gamma)
    \end{pmatrix}
    =\mu_4
    \begin{pmatrix}
      1  \\
      \sigma_{4}(\alpha) \\
      \sigma_{4}(\beta) \\
      \sigma_{4}(\gamma)
    \end{pmatrix},
  \]
   откуда $\mu_4 = 1$, $\sigma_{4}(\alpha) = \sigma_{2}(2\gamma - \alpha - 1) = \sigma_{2}\sigma_{3}(\alpha)$, $\sigma_{4}(\gamma) = \sigma_{2}(\gamma) = \sigma_{2}\sigma_{3}(\gamma)$, $\sigma_{4}(\beta) = 1 -  \sigma_{2}(\beta) -  2\sigma_{2}(\gamma) = \sigma_{2}(1 -  \beta - 2\gamma) = \sigma_{2}\sigma_{3}(\beta)$. Стало быть, $\cf(l_1, l_2, l_3, l_4) \in  \widetilde{\mathbf{CF}_{6}}$, так как числа $1,\alpha,\beta, \gamma$ образуют базис поля $K=\Q(\alpha,\beta, \gamma)$.

  Пусть $\cf(l_1, l_2, l_3, l_4) \in  \widetilde{\mathbf{CF}_{7}}$. Заметим, что $\sigma_{2} = \sigma_{2}\sigma^{2}_{3} = \sigma_{4}\sigma_{3}$, $\sigma_2(\gamma) = \sigma_{2}(\frac{\alpha + \sigma_{3}(\alpha) + 1}{2}) =  \frac{\sigma_{2}(\alpha) + \sigma_{2}\sigma_{3}(\alpha) + 1}{2}$, $\sigma_3(\gamma) = \sigma_3(\frac{\alpha + \sigma_{3}(\alpha) + 1}{2}) = \frac{\sigma_{3}(\alpha) + \alpha + 1}{2}$, $\sigma_4(\gamma) = \sigma_4(\frac{\alpha + \sigma_{3}(\alpha) + 1}{2}) = \frac{\sigma_{2}\sigma_{3}(\alpha) + \sigma_{2}(\alpha) + 1}{2}$, $\sigma_{2}(\beta) + \sigma_{2}\sigma_{3}(\beta) = \sigma_{2}\big(\beta + \sigma_{3}(\beta)\big) =  \sigma_{2}\big(2 - (\alpha + \sigma_{3}(\alpha))\big) = 2 - \big(\sigma_{2}(\alpha) + \sigma_{2}\sigma_{3}(\alpha)\big)$ и $\sigma_{4}(\beta) = \sigma_{2}\sigma_{3}(\beta)$. Тогда
  
  \[\widetilde{G_{7}}\vec{l}_1 =
  \begin{pmatrix}
  1\\
  \sigma_{3}(\alpha)\\
  2 - \beta - \big(\alpha + \sigma_{3}(\alpha)\big)\\
  \frac{1}{2}(\alpha + \sigma_{3}(\alpha) + 1)
   \end{pmatrix},
    \quad
    \widetilde{G_{7}}\vec{l}_2 =
  \begin{pmatrix}
  1\\
  \sigma_{2}\sigma_{3}(\alpha)\\
  2 - \sigma_{2}(\beta) - \big(\sigma_{2}(\alpha) + \sigma_{2}\sigma_{3}(\alpha)\big)\\
  \frac{1}{2}(\sigma_{2}(\alpha) + \sigma_{2}\sigma_{3}(\alpha) + 1)
   \end{pmatrix},
\]
\[\widetilde{G_{7}}\vec{l}_3 =
  \begin{pmatrix}
  1\\
  \alpha\\
  2 - \sigma_{3}(\beta) - \big(\sigma_{3}(\alpha) + \alpha\big)\\
  \frac{1}{2}(\sigma_{3}(\alpha) + \alpha + 1)
   \end{pmatrix},
    \quad
    \widetilde{G_{7}}\vec{l}_4 =
  \begin{pmatrix}
  1\\
  \sigma_{2}(\alpha)\\
  2 - \sigma_{2}\sigma_{3}(\beta) - \big(\sigma_{2}\sigma_{3}(\alpha) + \sigma_{2}(\alpha)\big)\\
  \frac{1}{2}(\sigma_{2}\sigma_{3}(\alpha) + \sigma_{2}(\alpha) + 1)
   \end{pmatrix},
\]
то есть $\widetilde{G_{7}} \big(\vec{l}_1, \vec{l}_2, \vec{l}_3, \vec{l}_4\big) = \big(\vec{l}_3, \vec{l}_4, \vec{l}_1, \vec{l}_2\big)$. Следовательно, $\widetilde{G_{7}}$ ---  собственная симметрия цепной дроби $\cf(l_1,l_2,l_3,l_4)$ и $\textup{ord}({\sigma_{\widetilde{G_{7}}}}) = 2$. Обратно, предположим, что $\widetilde{G_{7}}$ ---  собственная симметрия $\cf(l_1,l_2,l_3,l_4)$ и $\textup{ord}({\sigma_{\widetilde{G_{7}}}}) = 2$. Тогда существует такое  $\mu_3$, что с точностью до перестановки индексов
 \[
    \widetilde{G_{7}}\vec{l}_1=
    \begin{pmatrix}
      1  \\
      -1 -\alpha + 2\gamma  \\
      3 - \beta - 2\gamma \\
      \gamma
    \end{pmatrix}
    =\mu_3
    \begin{pmatrix}
      1  \\
      \sigma_{3}(\alpha) \\
      \sigma_{3}(\beta) \\
      \sigma_{3}(\gamma)
    \end{pmatrix}, 
  \]
откуда $\mu_3 = 1$,  $\gamma =  \frac{\alpha + \sigma_{3}(\alpha) + 1}{2}$, $\sigma^{2}_{3}(\alpha) = 2\sigma_{3}(\gamma) - \sigma_{3}(\alpha) - 1 = 2\gamma -  \sigma_{3}(\alpha) - 1= \alpha$, $\sigma^{2}_{3}(\gamma) = \sigma_{3}(\gamma) = \gamma$, $\beta + \sigma_{3}(\beta) = 3 - 2\gamma = 2 - \big(\alpha + \sigma_{3}(\alpha)\big)$, $\sigma^{2}_3(\beta) = - \sigma_3(\beta) + \sigma_3(3) - 2\sigma_3(\gamma) = - \sigma_3(\beta) + 3 - 2\gamma = \beta$. Существует такое $\mu_4$, что
  \[ 
    \widetilde{G_{7}}\vec{l}_2=
    \begin{pmatrix}
      1  \\
      -1 -\sigma_{2}(\alpha) + 2\sigma_{2}(\gamma)  \\
      3 -  \sigma_{2}(\beta) - 2\sigma_{2}(\gamma) \\
       \sigma_{2}(\gamma)
    \end{pmatrix}
    =\mu_4
    \begin{pmatrix}
      1  \\
      \sigma_{4}(\alpha) \\
      \sigma_{4}(\beta) \\
      \sigma_{4}(\gamma)
    \end{pmatrix},
  \]
   откуда $\mu_4 = 1$, $\sigma_{4}(\alpha) = \sigma_{2}(2\gamma - \alpha - 1) = \sigma_{2}\sigma_{3}(\alpha)$, $\sigma_{4}(\gamma) = \sigma_{2}(\gamma) = \sigma_{2}\sigma_{3}(\gamma)$, $\sigma_{4}(\beta) = 3 -  \sigma_{2}(\beta) -  2\sigma_{2}(\gamma) = \sigma_{2}(3 -  \beta - 2\gamma) = \sigma_{2}\sigma_{3}(\beta)$. Стало быть, $\cf(l_1, l_2, l_3, l_4) \in  \widetilde{\mathbf{CF}_{7}}$, так как числа $1,\alpha,\beta, \gamma$ образуют базис поля $K=\Q(\alpha,\beta, \gamma)$.

  Пусть $\cf(l_1, l_2, l_3, l_4) \in  \widetilde{\mathbf{CF}_{8}}$. Заметим, что $\sigma_{2} = \sigma_{2}\sigma^{2}_{3} = \sigma_{4}\sigma_{3}$, $\sigma_2(\gamma) = \sigma_{2}(\frac{\alpha + \sigma_{3}(\alpha)}{2}) =  \frac{\sigma_{2}(\alpha) + \sigma_{2}\sigma_{3}(\alpha)}{2}$, $\sigma_3(\gamma) = \sigma_3(\frac{\alpha + \sigma_{3}(\alpha)}{2}) = \frac{\sigma_{3}(\alpha) + \alpha}{2}$, $\sigma_4(\gamma) = \sigma_4(\frac{\alpha + \sigma_{3}(\alpha)}{2}) = \frac{\sigma_{2}\sigma_{3}(\alpha) + \sigma_{2}(\alpha)}{2}$, $\sigma_{2}(\beta) + \sigma_{2}\sigma_{3}(\beta) = \sigma_{2}\big(\beta + \sigma_{3}(\beta)\big) =  \sigma_{2}\big(1 - \frac{\alpha + \sigma_{3}(\alpha)}{2}\big) = 1 - \frac{\sigma_{2}(\alpha) + \sigma_{2}\sigma_{3}(\alpha)}{2}$ и $\sigma_{4}(\beta) = \sigma_{2}\sigma_{3}(\beta)$. Тогда
  
   \[\widetilde{G_{8}}\vec{l}_1 =
  \begin{pmatrix}
  1\\
  \sigma_{3}(\alpha)\\
  1 - \beta - \frac{1}{2}(\alpha + \sigma_{3}(\alpha))\\
  \frac{1}{2}(\alpha + \sigma_{3}(\alpha))
   \end{pmatrix},
    \quad
    \widetilde{G_{8}}\vec{l}_2 =
  \begin{pmatrix}
  1\\
  \sigma_{2}\sigma_{3}(\alpha)\\
  1 - \sigma_{2}(\beta) - \frac{1}{2}(\sigma_{2}(\alpha) + \sigma_{2}\sigma_{3}(\alpha))\\
  \frac{1}{2}(\sigma_{2}(\alpha) + \sigma_{2}\sigma_{3}(\alpha))
   \end{pmatrix},
\]
\[\widetilde{G_{8}}\vec{l}_3 =
  \begin{pmatrix}
  1\\
  \alpha\\
  1 - \sigma_{3}(\beta) - \frac{1}{2}(\sigma_{3}(\alpha) + \alpha)\\
  \frac{1}{2}(\sigma_{3}(\alpha) + \alpha)
   \end{pmatrix},
    \quad
    \widetilde{G_{8}}\vec{l}_4 =
  \begin{pmatrix}
  1\\
  \sigma_{2}(\alpha)\\
  1 - \sigma_{2}\sigma_{3}(\beta) - \frac{1}{2}(\sigma_{2}\sigma_{3}(\alpha) + \sigma_{2}(\alpha))\\
  \frac{1}{2}(\sigma_{2}\sigma_{3}(\alpha) + \sigma_{2}(\alpha))
   \end{pmatrix},
\]
то есть $\widetilde{G_{8}} \big(\vec{l}_1, \vec{l}_2, \vec{l}_3, \vec{l}_4\big) = \big(\vec{l}_3, \vec{l}_4, \vec{l}_1, \vec{l}_2\big)$. Следовательно, $\widetilde{G_{8}}$ ---  собственная симметрия цепной дроби $\cf(l_1,l_2,l_3,l_4)$ и $\textup{ord}({\sigma_{\widetilde{G_{8}}}}) = 2$. Обратно, предположим, что $\widetilde{G_{8}}$ ---  собственная симметрия $\cf(l_1,l_2,l_3,l_4)$ и $\textup{ord}({\sigma_{\widetilde{G_{8}}}}) = 2$. Тогда существует такое $\mu_3$, что с точностью до перестановки индексов
 \[
    \widetilde{G_{8}}\vec{l}_1=
    \begin{pmatrix}
      1  \\
      -\alpha + 2\gamma  \\
      1 - \beta - \gamma \\
      \gamma
    \end{pmatrix}
    =\mu_3
    \begin{pmatrix}
      1  \\
      \sigma_{3}(\alpha) \\
      \sigma_{3}(\beta) \\
      \sigma_{3}(\gamma)
    \end{pmatrix}, 
  \]
откуда $\mu_3 = 1$,  $\gamma =  \frac{\alpha + \sigma_{3}(\alpha)}{2}$, $\sigma^{2}_{3}(\alpha) = 2\sigma_{3}(\gamma) - \sigma_{3}(\alpha) = 2\gamma -  \sigma_{3}(\alpha) = \alpha$, $\sigma^{2}_{3}(\gamma) = \sigma_{3}(\gamma) = \gamma$, $\beta + \sigma_{3}(\beta) = 1 - \gamma = 1 - \frac{\alpha + \sigma_{3}(\alpha)}{2}$, $\sigma^{2}_3(\beta) = - \sigma_3(\beta) + \sigma_3(1) - \sigma_3(\gamma) = - \sigma_3(\beta) + 1 - \gamma = \beta$. Существует такое $\mu_4$, что
  \[ 
    \widetilde{G_{8}}\vec{l}_2=
    \begin{pmatrix}
      1  \\
      -\sigma_{2}(\alpha) + 2\sigma_{2}(\gamma)  \\
      1 -  \sigma_{2}(\beta) - \sigma_{2}(\gamma) \\
       \sigma_{2}(\gamma)
    \end{pmatrix}
    =\mu_4
    \begin{pmatrix}
      1  \\
      \sigma_{4}(\alpha) \\
      \sigma_{4}(\beta) \\
      \sigma_{4}(\gamma)
    \end{pmatrix},
  \]
   откуда $\mu_4 = 1$, $\sigma_{4}(\alpha) = \sigma_{2}(2\gamma - \alpha) = \sigma_{2}\sigma_{3}(\alpha)$, $\sigma_{4}(\gamma) = \sigma_{2}(\gamma) = \sigma_{2}\sigma_{3}(\gamma)$, $\sigma_{4}(\beta) = 1 -  \sigma_{2}(\beta) -  \sigma_{2}(\gamma) = \sigma_{2}(1 -  \beta - \gamma) = \sigma_{2}\sigma_{3}(\beta)$. Стало быть, $\cf(l_1, l_2, l_3, l_4) \in  \widetilde{\mathbf{CF}_{8}}$, так как числа $1,\alpha,\beta, \gamma$ образуют базис поля $K=\Q(\alpha,\beta, \gamma)$.

  Пусть $\cf(l_1, l_2, l_3, l_4) \in  \widetilde{\mathbf{CF}_{9}}$. Заметим, что $\sigma_{2} = \sigma_{2}\sigma^{2}_{3} = \sigma_{4}\sigma_{3}$, $\sigma_2(\gamma) = \sigma_{2}(\frac{\alpha + \sigma_{3}(\alpha)}{2}) =  \frac{\sigma_{2}(\alpha) + \sigma_{2}\sigma_{3}(\alpha)}{2}$, $\sigma_3(\gamma) = \sigma_3(\frac{\alpha + \sigma_{3}(\alpha)}{2}) = \frac{\sigma_{3}(\alpha) + \alpha}{2}$, $\sigma_4(\gamma) = \sigma_4(\frac{\alpha + \sigma_{3}(\alpha)}{2}) = \frac{\sigma_{2}\sigma_{3}(\alpha) + \sigma_{2}(\alpha)}{2}$, $\sigma_{2}(\beta) + \sigma_{2}\sigma_{3}(\beta) = \sigma_{2}\big(\beta + \sigma_{3}(\beta)\big) =  \sigma_{2}\big(2 - \frac{\alpha + \sigma_{3}(\alpha)}{2}\big) = 2 - \frac{\sigma_{2}(\alpha) + \sigma_{2}\sigma_{3}(\alpha)}{2}$ и $\sigma_{4}(\beta) = \sigma_{2}\sigma_{3}(\beta)$. Тогда
  
  \[\widetilde{G_{9}}\vec{l}_1 =
  \begin{pmatrix}
  1\\
  \sigma_{3}(\alpha)\\
  2 - \beta - \frac{1}{2}(\alpha + \sigma_{3}(\alpha))\\
  \frac{1}{2}(\alpha + \sigma_{3}(\alpha))
   \end{pmatrix},
    \quad
    \widetilde{G_{9}}\vec{l}_2 =
  \begin{pmatrix}
  1\\
  \sigma_{2}\sigma_{3}(\alpha)\\
  2 - \sigma_{2}(\beta) - \frac{1}{2}(\sigma_{2}(\alpha) + \sigma_{2}\sigma_{3}(\alpha))\\
  \frac{1}{2}(\sigma_{2}(\alpha) + \sigma_{2}\sigma_{3}(\alpha))
   \end{pmatrix},
\]
\[\widetilde{G_{9}}\vec{l}_3 =
  \begin{pmatrix}
  1\\
  \alpha\\
  2 - \sigma_{3}(\beta) - \frac{1}{2}(\sigma_{3}(\alpha) + \alpha)\\
  \frac{1}{2}(\sigma_{3}(\alpha) + \alpha)
   \end{pmatrix},
    \quad
    \widetilde{G_{9}}\vec{l}_4 =
  \begin{pmatrix}
  1\\
  \sigma_{2}(\alpha)\\
  2 - \sigma_{2}\sigma_{3}(\beta) - \frac{1}{2}(\sigma_{2}\sigma_{3}(\alpha) + \sigma_{2}(\alpha))\\
  \frac{1}{2}(\sigma_{2}\sigma_{3}(\alpha) + \sigma_{2}(\alpha))
   \end{pmatrix},
\]
то есть $\widetilde{G_{9}} \big(\vec{l}_1, \vec{l}_2, \vec{l}_3, \vec{l}_4\big) = \big(\vec{l}_3, \vec{l}_4, \vec{l}_1, \vec{l}_2\big)$. Следовательно, $\widetilde{G_{9}}$ ---  собственная симметрия цепной дроби $\cf(l_1,l_2,l_3,l_4)$ и $\textup{ord}({\sigma_{\widetilde{G_{9}}}}) = 2$. Обратно, предположим, что $\widetilde{G_{9}}$ ---  собственная симметрия $\cf(l_1,l_2,l_3,l_4)$ и $\textup{ord}({\sigma_{\widetilde{G_{9}}}}) = 2$. Тогда существует такое $\mu_3$, что с точностью до перестановки индексов
 \[
    \widetilde{G_{9}}\vec{l}_1=
    \begin{pmatrix}
      1  \\
      -\alpha + 2\gamma  \\
      2 - \beta - \gamma \\
      \gamma
    \end{pmatrix}
    =\mu_3
    \begin{pmatrix}
      1  \\
      \sigma_{3}(\alpha) \\
      \sigma_{3}(\beta) \\
      \sigma_{3}(\gamma)
    \end{pmatrix}, 
  \]
откуда $\mu_3 = 1$,  $\gamma =  \frac{\alpha + \sigma_{3}(\alpha)}{2}$, $\sigma^{2}_{3}(\alpha) = 2\sigma_{3}(\gamma) - \sigma_{3}(\alpha) = 2\gamma -  \sigma_{3}(\alpha) = \alpha$, $\sigma^{2}_{3}(\gamma) = \sigma_{3}(\gamma) = \gamma$, $\beta + \sigma_{3}(\beta) = 2 - \gamma = 2 - \frac{\alpha + \sigma_{3}(\alpha)}{2}$, $\sigma^{2}_3(\beta) = - \sigma_3(\beta) + \sigma_3(2) - \sigma_3(\gamma) = - \sigma_3(\beta) + 2 - \gamma = \beta$. Существует такое $\mu_4$, что
  \[ 
    \widetilde{G_{9}}\vec{l}_2=
    \begin{pmatrix}
      1  \\
      -\sigma_{2}(\alpha) + 2\sigma_{2}(\gamma)  \\
      2 -  \sigma_{2}(\beta) - \sigma_{2}(\gamma) \\
       \sigma_{2}(\gamma)
    \end{pmatrix}
    =\mu_4
    \begin{pmatrix}
      1  \\
      \sigma_{4}(\alpha) \\
      \sigma_{4}(\beta) \\
      \sigma_{4}(\gamma)
    \end{pmatrix},
  \]
   откуда $\mu_4 = 1$, $\sigma_{4}(\alpha) = \sigma_{2}(2\gamma - \alpha) = \sigma_{2}\sigma_{3}(\alpha)$, $\sigma_{4}(\gamma) = \sigma_{2}(\gamma) = \sigma_{2}\sigma_{3}(\gamma)$, $\sigma_{4}(\beta) = 2 -  \sigma_{2}(\beta) -  \sigma_{2}(\gamma) = \sigma_{2}(2 -  \beta - \gamma) = \sigma_{2}\sigma_{3}(\beta)$. Стало быть, $\cf(l_1, l_2, l_3, l_4) \in  \widetilde{\mathbf{CF}_{9}}$, так как числа $1,\alpha,\beta, \gamma$ образуют базис поля $K=\Q(\alpha,\beta, \gamma)$.

  Пусть $\cf(l_1, l_2, l_3, l_4) \in  \widetilde{\mathbf{CF}_{10}}$. Заметим, что $\sigma_{2} = \sigma_{2}\sigma^{2}_{3} = \sigma_{4}\sigma_{3}$, $\sigma_2(\gamma) = \sigma_{2}(\frac{\sigma_{3}(\alpha) - \alpha}{4}) =  \frac{\sigma_{2}\sigma_{3}(\alpha) - \sigma_{2}(\alpha)}{4}$, $\sigma_3(\gamma) = \sigma_3(\frac{\sigma_{3}(\alpha) - \alpha}{4}) = \frac{\alpha - \sigma_{3}(\alpha)}{4}$, $\sigma_4(\gamma) = \sigma_4(\frac{\sigma_{3}(\alpha) - \alpha}{4}) = \frac{\sigma_{2}(\alpha) - \sigma_{2}\sigma_{3}(\alpha)}{4}$, $\sigma_{2}(\beta) + \sigma_{2}\sigma_{3}(\beta) = \sigma_{2}\big(\beta + \sigma_{3}(\beta)\big) =  \sigma_{2}\big(2 - \frac{\alpha + \sigma_{3}(\alpha)}{2}\big) = 2 - \frac{\sigma_{2}(\alpha) + \sigma_{2}\sigma_{3}(\alpha)}{2}$ и $\sigma_{4}(\beta) = \sigma_{2}\sigma_{3}(\beta)$. Тогда
  
  \[\widetilde{G_{10}}\vec{l}_1 =
  \begin{pmatrix}
  1\\
  \sigma_{3}(\alpha)\\
  2 - \beta - \frac{1}{2}(\alpha + \sigma_{3}(\alpha))\\
  \frac{1}{2}(\alpha - \sigma_{3}(\alpha))
   \end{pmatrix},
    \quad
    \widetilde{G_{10}}\vec{l}_2 =
  \begin{pmatrix}
  1\\
  \sigma_{2}\sigma_{3}(\alpha)\\
  2 - \sigma_{2}(\beta) - \frac{1}{2}(\sigma_{2}(\alpha) + \sigma_{2}\sigma_{3}(\alpha))\\
  \frac{1}{2}(\sigma_{2}(\alpha) - \sigma_{2}\sigma_{3}(\alpha))
   \end{pmatrix},
\]
\[\widetilde{G_{10}}\vec{l}_3 =
  \begin{pmatrix}
  1\\
  \alpha\\
  2 - \sigma_{3}(\beta) - \frac{1}{2}(\sigma_{3}(\alpha) + \alpha)\\
  \frac{1}{2}(\sigma_{3}(\alpha) - \alpha)
   \end{pmatrix},
    \quad
    \widetilde{G_{10}}\vec{l}_4 =
  \begin{pmatrix}
  1\\
  \sigma_{2}(\alpha)\\
  2 - \sigma_{2}\sigma_{3}(\beta) - \frac{1}{2}(\sigma_{2}\sigma_{3}(\alpha) + \sigma_{2}(\alpha))\\
  \frac{1}{2}(\sigma_{2}\sigma_{3}(\alpha) - \sigma_{2}(\alpha))
   \end{pmatrix},
\]
то есть $\widetilde{G_{10}} \big(\vec{l}_1, \vec{l}_2, \vec{l}_3, \vec{l}_4\big) = \big(\vec{l}_3, \vec{l}_4, \vec{l}_1, \vec{l}_2\big)$. Следовательно, $\widetilde{G_{10}}$ ---  собственная симметрия цепной дроби $\cf(l_1,l_2,l_3,l_4)$ и $\textup{ord}({\sigma_{\widetilde{G_{10}}}}) = 2$. Обратно, предположим, что $\widetilde{G_{10}}$ ---  собственная симметрия $\cf(l_1,l_2,l_3,l_4)$ и $\textup{ord}({\sigma_{\widetilde{G_{10}}}}) = 2$. Тогда существует такое $\mu_3$, что с точностью до перестановки индексов
 \[
    \widetilde{G_{10}}\vec{l}_1=
    \begin{pmatrix}
      1  \\
      \alpha + 4\gamma  \\
      2 - \beta - (\alpha + 2\gamma) \\
      -\gamma
    \end{pmatrix}
    =\mu_3
    \begin{pmatrix}
      1  \\
      \sigma_{3}(\alpha) \\
      \sigma_{3}(\beta) \\
      \sigma_{3}(\gamma)
    \end{pmatrix}, 
  \]
откуда $\mu_3 = 1$,  $\gamma =  \frac{\sigma_{3}(\alpha) - \alpha}{4}$, $\sigma^{2}_{3}(\alpha) = 4\sigma_{3}(\gamma) + \sigma_{3}(\alpha) = -4\gamma +  \sigma_{3}(\alpha) = \alpha$, $\sigma^{2}_{3}(\gamma) = \sigma_{3}(-\gamma) = \gamma$, $\beta + \sigma_{3}(\beta) = 2 - (\alpha + 2\gamma) = 2 - \frac{\alpha + \sigma_{3}(\alpha)}{2}$, $\sigma^{2}_3(\beta) = - \sigma_3(\beta) + \sigma_3(2) - \sigma_3(\alpha + 2\gamma) = - \sigma_3(\beta) + 2 -  \frac{\alpha + \sigma_{3}(\alpha)}{2} = \beta$. Существует такое $\mu_4$, что
  \[ 
    \widetilde{G_{10}}\vec{l}_2=
    \begin{pmatrix}
      1  \\
      \sigma_{2}(\alpha) + 4\sigma_{2}(\gamma)  \\
      2 -  \sigma_{2}(\beta) - \big(\sigma_{2}(\alpha) + 2\sigma_{2}(\gamma)\big) \\
       -\sigma_{2}(\gamma)
    \end{pmatrix}
    =\mu_4
    \begin{pmatrix}
      1  \\
      \sigma_{4}(\alpha) \\
      \sigma_{4}(\beta) \\
      \sigma_{4}(\gamma)
    \end{pmatrix},
  \]
   откуда $\mu_4 = 1$, $\sigma_{4}(\alpha) = \sigma_{2}(\alpha + 4\gamma) = \sigma_{2}\sigma_{3}(\alpha)$, $\sigma_{4}(\gamma) = \sigma_{2}(-\gamma) = \sigma_{2}\sigma_{3}(\gamma)$, $\sigma_{4}(\beta) = 2 -  \sigma_{2}(\beta) -  \big(\sigma_{2}(\alpha) + 2\sigma_{2}(\gamma)\big) = \sigma_{2}\big(2 -  \beta - (\alpha + 2\gamma)\big) = \sigma_{2}\big(2 -  \beta -  \frac{\alpha + \sigma_{3}(\alpha)}{2}\big) = \sigma_{2}\sigma_{3}(\beta)$. Стало быть, $\cf(l_1, l_2, l_3, l_4) \in  \widetilde{\mathbf{CF}_{10}}$, так как числа $1,\alpha,\beta, \gamma$ образуют базис поля $K=\Q(\alpha,\beta, \gamma)$.

\end{proof}

Мы будем обозначать через $\gA_{3}'$ множество всех трехмерных алгебраических цепных дробей, для которых поле $K$ из предложения \ref{prop:more_than_pelle_n_dim} --- вполне вещественное циклическое расширение Галуа. Пусть $\sigma$ --- образующая группы Галуа $\gal(K/\Q)$. Также мы выбирали такую нумерацию прямых $l_1, l_2, l_3, l_4$, что если через $\big(\vec l_1,\vec l_2, \vec l_3, \vec l_4 \big)$ обозначить матрицу со столбцами $\vec l_1,\vec l_2, \vec l_3, \vec l_4$, то 
 \[
  \big(\vec l_1,\vec l_2, \vec l_3, \vec l_4 \big)=
  \begin{pmatrix}
     1 & 1 & 1 & 1 \\
    \alpha & \sigma(\alpha) & \sigma^{2}(\alpha) & \sigma^{3}(\alpha) \\
    \beta & \sigma(\beta) & \sigma^{2}(\beta) & \sigma^{3}(\beta) \\
    \gamma & \sigma(\gamma) & \sigma^{2}(\gamma) & \sigma^{3}(\gamma)
  \end{pmatrix}.
\]

Для каждого $i=1, 2, \ldots, 7$ определим $\mathbf{CF}_{i}'$ как класс дробей из $\gA_{3}'$, удовлетворяющих тройке соотношений $\gQ_{i}$, где
  
  $\gQ_{1}$: $\beta = \sigma(\alpha), \gamma = \sigma^{2}(\alpha), \trace(\alpha) = 0$;
  
  $\gQ_{2}$: $\beta = \sigma(\alpha), \gamma = \sigma^{2}(\alpha), \trace(\alpha) = 1$;
  
  $\gQ_{3}$: $\beta = \sigma(\alpha), \gamma = \sigma^{2}(\alpha), \trace(\alpha) = 2$;
  
  $\gQ_{4}$: $\beta = \sigma(\alpha), \gamma = \frac{\alpha + \sigma^{2}(\alpha)}{2}, \trace(\alpha) = 0$;
  
  $\gQ_{5}$: $\beta = \sigma(\alpha), \gamma = \frac{\alpha + \sigma^{2}(\alpha)}{2}, \trace(\alpha) = 2$;
  
  $\gQ_{6}$: $\beta = \sigma(\alpha), \gamma = \frac{\alpha + \sigma^{2}(\alpha) + 1}{2}, \trace(\alpha) = 0$;  
  
  $\gQ_{7}$: $\beta = \sigma(\alpha), \gamma = \frac{\alpha + \sigma^{2}(\alpha) + 1}{2}, \trace(\alpha) = 2$.
  
  Покажем, что все дроби из классов $\mathbf{CF}_{i}'$, палиндромичны для каждого $i=1, 2, \ldots, 7$. Положим $G_{1}', G_{2}', \ldots, G_{7}'$ равными соответственно матрицам
  
\[
  \left(\begin{smallmatrix}
    1 & \phantom{-}0 & \phantom{-}0 & \phantom{-}0 \\
    0 & \phantom{-}0 & \phantom{-}1 & \phantom{-}0\\
    0 & \phantom{-}0 & \phantom{-}0 & \phantom{-}1\\
    0 & -1 & -1 & -1
  \end{smallmatrix}\right),  
  \left(\begin{smallmatrix}
    1 & \phantom{-}0 &  \phantom{-}0 & \phantom{-}0 \\
    0 & \phantom{-}0 &  \phantom{-}1 & \phantom{-}0 \\
    0 &  \phantom{-}0 &  \phantom{-}0 & \phantom{-}1 \\
    1 & -1 & -1 & -1
  \end{smallmatrix}\right), 
  \left(\begin{smallmatrix}
    1 & \phantom{-}0 &  \phantom{-}0 & \phantom{-}0 \\
    0 & \phantom{-}0 &  \phantom{-}1 & \phantom{-}0 \\
    0 &  \phantom{-}0 &  \phantom{-}0 & \phantom{-}1 \\
    2 & -1 & -1 & -1
  \end{smallmatrix}\right), 
  \left(\begin{smallmatrix}
    1 & \phantom{-}0 & \phantom{-}0 & \phantom{-}0 \\
    0 & \phantom{-}0 & \phantom{-}1 & \phantom{-}0 \\
    0 & -1 & \phantom{-}0 & \phantom{-}2 \\
    0 &\phantom{-}0  & \phantom{-}0 & -1
  \end{smallmatrix}\right),
 \]
 \[
  \left(\begin{smallmatrix}
    1 & \phantom{-}0 &  \phantom{-}0 & \phantom{-}0 \\
    0 & \phantom{-}0 &  \phantom{-}1 & \phantom{-}0 \\
    0 &  -1 &  \phantom{-}0 & \phantom{-}2 \\
    1 & \phantom{-}0 & \phantom{-}0 & -1
  \end{smallmatrix}\right),
  \left(\begin{smallmatrix}
    \phantom{-}1 & \phantom{-}0 &  \phantom{-}0 & \phantom{-}0 \\
     \phantom{-}0 & \phantom{-}0 & \phantom{-}1 & \phantom{-}0 \\
    -1 & -1 & \phantom{-}0 & \phantom{-}2 \\
     \phantom{-}1 & \phantom{-}0 & \phantom{-}0 & -1
  \end{smallmatrix}\right),
  \left(\begin{smallmatrix}
    \phantom{-}1 & \phantom{-}0 &  \phantom{-}0 & \phantom{-}0 \\
     \phantom{-}0 & \phantom{-}0 & \phantom{-}1 & \phantom{-}0 \\
    -1 & -1 & \phantom{-}0 & \phantom{-}2 \\
     \phantom{-}2 & \phantom{-}0 & \phantom{-}0 & -1
  \end{smallmatrix}\right).
\]

\begin{lemma}\label{oper_eq_4d}
  Пусть $\cf(l_1,l_2,l_3,l_4)\in\gA_3$ и $i\in\{1,2,3,4,5,6,7\}$. Тогда цепная дробь $\cf(l_1,l_2,l_3,l_4)$ принадлежит классу $\mathbf{CF}_{i}'$ в том и только в том случае, если $G_{i}'$ --- её собственная циклическая симметрия.
\end{lemma}

\begin{proof}
  В силу леммы \ref{prod_lemm_cyclic} оператор $G\in\Gl_4(\Z)$ является собственной циклической симметрией дроби $\cf(l_1,l_2,l_3,l_4)$ тогда и только тогда, когда с точностью до перестановки индексов существуют такие действительные числа $\mu_1,\mu_2,\mu_3,\mu_4$, что $G\big(\vec l_1,\vec l_2,\vec l_3,\vec l_4\big)=\big(\mu_2\vec l_2,\mu_3\vec l_3,\mu_4\vec l_4,\mu_1\vec l_1\big)$ и $\mu_{1}\mu_{2}\mu_{3}\mu_{4} = 1$.

  Пусть $\cf(l_1, l_2, l_3, l_4) \in \mathbf{CF}_{1}'$. Тогда
  
   \[ 
    G_{1}'\vec{l}_1=
    \begin{pmatrix}
      1  \\
      \sigma(\alpha) \\
      \sigma^{2}(\alpha) \\
      -\alpha - \sigma(\alpha) - \sigma^{2}(\alpha)
    \end{pmatrix},
    \quad
    G_{1}'\vec{l}_2=
    \begin{pmatrix}
      1  \\
      \sigma^{2}(\alpha) \\
      \sigma^{3}(\alpha) \\
      -\sigma(\alpha) - \sigma^{2}(\alpha) - \sigma^{3}(\alpha)
    \end{pmatrix},
  \]
  \[ 
    G_{1}'\vec{l}_3=
    \begin{pmatrix}
      1  \\
      \sigma^{3}(\alpha) \\
      \alpha \\
      -\sigma^{2}(\alpha) - \sigma^{3}(\alpha) - \alpha
    \end{pmatrix},
    \quad
    G_{1}'\vec{l}_4=
    \begin{pmatrix}
      1  \\
      \alpha \\
      \sigma(\alpha) \\
      -\sigma^{3}(\alpha) - \alpha - \sigma(\alpha) 
    \end{pmatrix},
  \]
 то есть $G_{1}' \big(\vec{l}_1, \vec{l}_2, \vec{l}_3, \vec{l}_4\big) = \big(\vec{l}_2, \vec{l}_3, \vec{l}_4, \vec{l}_1\big)$. Следовательно, $G_{1}'$ --- собственная циклическая симметрия цепной дроби $\cf(l_1,l_2,l_3,l_4)$. Обратно, предположим, что $G_{1}'$ --- собственная циклическая симметрия цепной дроби $\cf(l_1,l_2,l_3,l_4)$. Тогда существует такое $\mu_2$, что с точностью до перестановки индексов
  \[
    G_{1}'\vec{l}_1=
    \begin{pmatrix}
      1  \\
      \beta \\
      \gamma \\
      - \alpha - \beta - \gamma
    \end{pmatrix}
    =\mu_2
    \begin{pmatrix}
      1  \\
      \sigma_{2}(\alpha) \\
      \sigma_{2}(\beta) \\
      \sigma_{2}(\gamma)
    \end{pmatrix}, 
  \]
  откуда $\mu_2 = 1$, $\beta =\sigma_{2}(\alpha)$, $\gamma =\sigma_{2}(\beta)$, $- \alpha - \beta - \gamma=\sigma_{2}(\gamma)$. Существует такое $\mu_3$, что
  \[ 
    G_{1}'\vec{l}_2=
    \begin{pmatrix}
      1  \\
      \gamma \\
      - \alpha - \beta - \gamma \\
      \alpha
    \end{pmatrix}
    =\mu_3
    \begin{pmatrix}
      1  \\
      \sigma_{3}(\alpha) \\
      \sigma_{3}(\beta) \\
      \sigma_{3}(\gamma)
    \end{pmatrix},
  \]
   откуда $\mu_3 = 1$, $\sigma_{3}(\alpha) = \gamma = \sigma_{2}(\beta) = \sigma^{2}_{2}(\alpha)$, $\sigma_{3}(\beta) = - \alpha - \beta - \gamma = \sigma_{2}(\gamma) = \sigma^{2}_{2}(\beta)$, $\sigma_{3}(\gamma) = \alpha = -\beta - \gamma - (- \alpha - \beta - \gamma) = \sigma_{2}(- \alpha - \beta - \gamma) = \sigma^{2}_{2}(\gamma)$.
   Существует такое $\mu_4$, что
  \[ 
    G_{1}'\vec{l}_3=
    \begin{pmatrix}
      1  \\
      - \alpha - \beta - \gamma  \\
      \alpha \\
      \beta
    \end{pmatrix}
    =\mu_4
    \begin{pmatrix}
      1  \\
      \sigma_{4}(\alpha) \\
      \sigma_{4}(\beta) \\
      \sigma_{4}(\gamma)
    \end{pmatrix},
  \]
   откуда $\mu_4 = 1$, $\sigma_{4}(\alpha) = - \alpha - \beta - \gamma = \sigma_{2}(\gamma) =  \sigma^{3}_{2}(\alpha)$, $\sigma_{4}(\beta) = \alpha = -\beta - \gamma - (- \alpha - \beta - \gamma) = \sigma_{2}(-\alpha - \beta - \gamma) =  \sigma^{3}_{2}(\beta)$, $\sigma_{4}(\gamma) = \beta  = \sigma_{2}(\alpha) = \sigma^{3}_{2}(\gamma)$, $\trace(\alpha) = 0$. Стало быть, $\cf(l_1, l_2, l_3, l_4) \in \mathbf{CF}_{1}'$, так как числа $1,\alpha,\beta, \gamma$ образуют базис поля $K=\Q(\alpha,\beta, \gamma)$.

Пусть $\cf(l_1, l_2, l_3, l_4) \in \mathbf{CF}_{2}'$. Тогда

 \[ 
    G_{2}'\vec{l}_1=
    \begin{pmatrix}
      1  \\
      \sigma(\alpha) \\
      \sigma^{2}(\alpha) \\
      1 -\alpha - \sigma(\alpha) - \sigma^{2}(\alpha)
    \end{pmatrix},
    \quad
    G_{2}'\vec{l}_2=
    \begin{pmatrix}
      1  \\
      \sigma^{2}(\alpha) \\
      \sigma^{3}(\alpha) \\
      1 -\sigma(\alpha) - \sigma^{2}(\alpha) - \sigma^{3}(\alpha)
    \end{pmatrix},
  \]
  \[ 
    G_{2}'\vec{l}_3=
    \begin{pmatrix}
      1  \\
      \sigma^{3}(\alpha) \\
      \alpha \\
      1 -\sigma^{2}(\alpha) - \sigma^{3}(\alpha) - \alpha
    \end{pmatrix},
    \quad
    G_{2}'\vec{l}_4=
    \begin{pmatrix}
      1  \\
      \alpha \\
      \sigma(\alpha) \\
      1 -\sigma^{3}(\alpha) - \alpha - \sigma(\alpha) 
    \end{pmatrix},
  \]
 то есть $G_{2}' \big(\vec{l}_1, \vec{l}_2, \vec{l}_3, \vec{l}_4\big) = \big(\vec{l}_2, \vec{l}_3, \vec{l}_4, \vec{l}_1\big)$. Следовательно, $G_{2}'$ --- собственная циклическая симметрия цепной дроби $\cf(l_1,l_2,l_3,l_4)$. Обратно, предположим, что $G_{2}'$ --- собственная циклическая симметрия цепной дроби $\cf(l_1,l_2,l_3,l_4)$. Тогда существует такое $\mu_2$, что с точностью до перестановки индексов 
   \[
    G_{2}'\vec{l}_1=
    \begin{pmatrix}
      1  \\
      \beta \\
      \gamma \\
     1 - \alpha - \beta - \gamma
    \end{pmatrix}
    =\mu_2
    \begin{pmatrix}
      1  \\
      \sigma_{2}(\alpha) \\
      \sigma_{2}(\beta) \\
      \sigma_{2}(\gamma)
    \end{pmatrix}, 
  \]
  откуда $\mu_2 = 1$, $\beta =\sigma_{2}(\alpha)$, $\gamma =\sigma_{2}(\beta)$, $1 - \alpha - \beta - \gamma=\sigma_{2}(\gamma)$. Существует такое $\mu_3$, что
  \[ 
    G_{2}'\vec{l}_2=
    \begin{pmatrix}
      1  \\
      \gamma \\
      1 - \alpha - \beta - \gamma \\
      \alpha
    \end{pmatrix}
    =\mu_3
    \begin{pmatrix}
      1  \\
      \sigma_{3}(\alpha) \\
      \sigma_{3}(\beta) \\
      \sigma_{3}(\gamma)
    \end{pmatrix},
  \]
   откуда $\mu_3 = 1$, $\sigma_{3}(\alpha) = \gamma = \sigma_{2}(\beta) = \sigma^{2}_{2}(\alpha)$, $\sigma_{3}(\beta) = 1 - \alpha - \beta - \gamma = \sigma_{2}(\gamma) = \sigma^{2}_{2}(\beta)$, $\sigma_{3}(\gamma) = \alpha = 1 -\beta - \gamma - (1 - \alpha - \beta - \gamma) = \sigma_{2}(1 - \alpha - \beta - \gamma) = \sigma^{2}_{2}(\gamma)$.
   Существует такое $\mu_4$, что
  \[ 
    G_{2}'\vec{l}_3=
    \begin{pmatrix}
      1  \\
      1 - \alpha - \beta - \gamma  \\
      \alpha \\
      \beta
    \end{pmatrix}
    =\mu_4
    \begin{pmatrix}
      1  \\
      \sigma_{4}(\alpha) \\
      \sigma_{4}(\beta) \\
      \sigma_{4}(\gamma)
    \end{pmatrix},
  \]
   откуда $\mu_4 = 1$, $\sigma_{4}(\alpha) = 1 - \alpha - \beta - \gamma = \sigma_{2}(\gamma) =  \sigma^{3}_{2}(\alpha)$, $\sigma_{4}(\beta) = \alpha = 1 - \beta - \gamma - (1 - \alpha - \beta - \gamma) = \sigma_{2}(1 - \alpha - \beta - \gamma) =  \sigma^{3}_{2}(\beta)$, $\sigma_{4}(\gamma) = \beta  = \sigma_{2}(\alpha) = \sigma^{3}_{2}(\gamma)$, $\trace(\alpha) = 1$. Стало быть, $\cf(l_1, l_2, l_3, l_4) \in \mathbf{CF}_{2}'$, так как числа $1,\alpha,\beta, \gamma$ образуют базис поля $K=\Q(\alpha,\beta, \gamma)$.

Пусть $\cf(l_1, l_2, l_3, l_4) \in \mathbf{CF}_{3}'$. Тогда

 \[ 
    G_{3}'\vec{l}_1=
    \begin{pmatrix}
      1  \\
      \sigma(\alpha) \\
      \sigma^{2}(\alpha) \\
      2 -\alpha - \sigma(\alpha) - \sigma^{2}(\alpha)
    \end{pmatrix},
    \quad
    G_{3}'\vec{l}_2=
    \begin{pmatrix}
      1  \\
      \sigma^{2}(\alpha) \\
      \sigma^{3}(\alpha) \\
      2 -\sigma(\alpha) - \sigma^{2}(\alpha) - \sigma^{3}(\alpha)
    \end{pmatrix},
  \]
  \[ 
    G_{3}'\vec{l}_3=
    \begin{pmatrix}
      1  \\
      \sigma^{3}(\alpha) \\
      \alpha \\
      2 -\sigma^{2}(\alpha) - \sigma^{3}(\alpha) - \alpha
    \end{pmatrix},
    \quad
    G_{3}'\vec{l}_4=
    \begin{pmatrix}
      1  \\
      \alpha \\
      \sigma(\alpha) \\
      2 -\sigma^{3}(\alpha) - \alpha - \sigma(\alpha) 
    \end{pmatrix},
  \]
 то есть $G_{3}' \big(\vec{l}_1, \vec{l}_2, \vec{l}_3, \vec{l}_4\big) = \big(\vec{l}_2, \vec{l}_3, \vec{l}_4, \vec{l}_1\big)$. Следовательно, $G_{3}'$ --- собственная циклическая симметрия цепной дроби $\cf(l_1,l_2,l_3,l_4)$. Обратно, предположим, что $G_{3}'$ --- собственная циклическая симметрия цепной дроби $\cf(l_1,l_2,l_3,l_4)$. Тогда существует такое $\mu_2$, что с точностью до перестановки индексов
  \[
    G_{3}'\vec{l}_1=
    \begin{pmatrix}
      1  \\
      \beta \\
      \gamma \\
      2 - \alpha - \beta - \gamma
    \end{pmatrix}
    =\mu_2
    \begin{pmatrix}
      1  \\
      \sigma_{2}(\alpha) \\
      \sigma_{2}(\beta) \\
      \sigma_{2}(\gamma)
    \end{pmatrix}, 
  \]
  откуда $\mu_2 = 1$, $\beta =\sigma_{2}(\alpha)$, $\gamma =\sigma_{2}(\beta)$, $2 - \alpha - \beta - \gamma =\sigma_{2}(\gamma)$. Существует такое $\mu_3$, что   
  \[ 
    G_{3}'\vec{l}_2=
    \begin{pmatrix}
      1  \\
      \gamma \\
      2 - \alpha - \beta - \gamma \\
      \alpha
    \end{pmatrix}
    =\mu_3
    \begin{pmatrix}
      1  \\
      \sigma_{3}(\alpha) \\
      \sigma_{3}(\beta) \\
      \sigma_{3}(\gamma)
    \end{pmatrix},
  \]
   откуда $\mu_3 = 1$, $\sigma_{3}(\alpha) = \gamma = \sigma_{2}(\beta) = \sigma^{2}_{2}(\alpha)$, $\sigma_{3}(\beta) = 2 - \alpha - \beta - \gamma = \sigma_{2}(\gamma) = \sigma^{2}_{2}(\beta)$, $\sigma_{3}(\gamma) = \alpha = 2 - \beta - \gamma - (2 -\alpha - \beta -\gamma) = \sigma_{2}(2 - \alpha - \beta - \gamma) = \sigma^{2}_{2}(\gamma)$. Существует такое $\mu_4$, что
  \[ 
    G_{3}'\vec{l}_3=
    \begin{pmatrix}
      1  \\
      2 - \alpha - \beta - \gamma  \\
      \alpha \\
      \beta
    \end{pmatrix}
    =\mu_4
    \begin{pmatrix}
      1  \\
      \sigma_{4}(\alpha) \\
      \sigma_{4}(\beta) \\
      \sigma_{4}(\gamma)
    \end{pmatrix},
  \]
   откуда $\mu_4 = 1$, $\sigma_{4}(\alpha) = 2 - \alpha - \beta - \gamma = \sigma_{2}(\gamma) = \sigma^{3}_{2}(\alpha)$, $\sigma_{4}(\beta) = \alpha = 2 - \beta - \gamma - (2 -\alpha - \beta -\gamma) = \sigma_{2}(2 - \alpha - \beta - \gamma) = \sigma^{3}_{2}(\beta)$, $\sigma_{4}(\gamma) = \beta = \sigma_{2}(\alpha) = \sigma^{3}_{2}(\gamma)$, $\trace(\alpha) = 2$. Стало быть, $\cf(l_1, l_2, l_3, l_4) \in \mathbf{CF}_{3}'$, так как числа $1,\alpha,\beta, \gamma$ образуют базис поля $K=\Q(\alpha,\beta, \gamma)$.

   Пусть $\cf(l_1, l_2, l_3, l_4) \in \mathbf{CF}_{4}'$. Тогда
 
 \[ 
    G_{4}'\vec{l}_1=
    \begin{pmatrix}
      1  \\
      \sigma(\alpha) \\
      \sigma^{2}(\alpha) \\
      -\frac{1}{2}(\alpha + \sigma^{2}(\alpha))
    \end{pmatrix},
    \quad
    G_{4}'\vec{l}_2=
    \begin{pmatrix}
      1  \\
      \sigma^{2}(\alpha) \\
      \sigma^{3}(\alpha) \\
      -\frac{1}{2}(\sigma(\alpha) + \sigma^{3}(\alpha))
    \end{pmatrix},
  \]
  \[ 
    G_{4}'\vec{l}_3=
    \begin{pmatrix}
      1  \\
      \sigma^{3}(\alpha) \\
      \alpha \\
      -\frac{1}{2}(\sigma^{2}(\alpha) + \alpha)
    \end{pmatrix},
    \quad
    G_{4}'\vec{l}_4=
    \begin{pmatrix}
      1  \\
      \alpha \\
      \sigma(\alpha) \\
      -\frac{1}{2}(\sigma^{3}(\alpha) + \sigma(\alpha))
    \end{pmatrix},
  \]
 то есть $G_{4}' \big(\vec{l}_1, \vec{l}_2, \vec{l}_3, \vec{l}_4\big) = \big(\vec{l}_2, \vec{l}_3, \vec{l}_4, \vec{l}_1\big)$. Следовательно, $G_{4}'$ --- собственная циклическая симметрия цепной дроби $\cf(l_1,l_2,l_3,l_4)$. Обратно, предположим, что $G_{4}'$ --- собственная циклическая симметрия цепной дроби $\cf(l_1,l_2,l_3,l_4)$. Тогда существует такое $\mu_2$, что с точностью до перестановки индексов
  \[
    G_{4}'\vec{l}_1=
    \begin{pmatrix}
      1  \\
      \beta \\
      - \alpha + 2\gamma\\
      - \gamma
    \end{pmatrix}
    =\mu_2
    \begin{pmatrix}
      1  \\
      \sigma_{2}(\alpha) \\
      \sigma_{2}(\beta) \\
      \sigma_{2}(\gamma)
    \end{pmatrix}, 
  \]
откуда $\mu_2 = 1$, $\beta =\sigma_{2}(\alpha)$, $- \alpha + 2\gamma =\sigma_{2}(\beta)$, $ - \gamma =\sigma_{2}(\gamma)$. Существует такое $\mu_3$, что
    \[ 
    G_{4}'\vec{l}_2=
    \begin{pmatrix}
      1  \\
      - \alpha + 2\gamma  \\
      - \beta - 2\gamma \\
      \gamma
    \end{pmatrix}
    =\mu_3
    \begin{pmatrix}
      1  \\
      \sigma_{3}(\alpha) \\
      \sigma_{3}(\beta) \\
      \sigma_{3}(\gamma)
    \end{pmatrix},
  \]
откуда $\mu_3 = 1$, $\frac{\alpha + \sigma_{3}(\alpha)}{2} = \gamma, \sigma_{3}(\alpha) = - \alpha + 2\gamma = \sigma_{2}(\beta) = \sigma^{2}_{2}(\alpha)$,  $\sigma_{3}(\beta) = - \beta - 2\gamma = \sigma_{2}(-\alpha + 2\gamma) = \sigma^{2}_{2}(\beta)$, $\sigma_{3}(\gamma) = \gamma = \sigma_{2}(-\gamma) = \sigma^{2}_{2}(\gamma)$. Существует такое $\mu_4$, что
 \[ 
    G_{4}'\vec{l}_3=
    \begin{pmatrix}
      1  \\
      - \beta -  2\gamma  \\
      \alpha \\
      -\gamma
    \end{pmatrix}
    =\mu_4
    \begin{pmatrix}
      1  \\
      \sigma_{4}(\alpha) \\
      \sigma_{4}(\beta) \\
      \sigma_{4}(\gamma)
    \end{pmatrix},
  \]
откуда $\mu_4 = 1$, $\sigma_{4}(\alpha) = - \beta -  2\gamma = \sigma_{2}(-\alpha) + \sigma_{2}(2\gamma) = \sigma_{2}(-\alpha + 2\gamma) = \sigma^{3}_{2}(\alpha)$, $\sigma_{4}(\beta) = \alpha = 2\gamma - \sigma_{2}(\beta)  = \sigma_{2}(-2\gamma - \beta) =  \sigma^{3}_{2}(\beta)$, $\sigma_{4}(\gamma) = -\gamma = \sigma_{2}(\gamma)  =  \sigma^{3}_{2}(\gamma)$, $\trace(\alpha) = 0$. Стало быть, $\cf(l_1, l_2, l_3, l_4) \in \mathbf{CF}_{4}'$, так как числа $1,\alpha,\beta, \gamma$ образуют базис поля $K=\Q(\alpha,\beta, \gamma)$.

Пусть $\cf(l_1, l_2, l_3, l_4) \in \mathbf{CF}_{5}'$. Тогда

\[ 
    G_{5}'\vec{l}_1=
    \begin{pmatrix}
      1  \\
      \sigma(\alpha) \\
      \sigma^{2}(\alpha) \\
      \frac{1}{2}(2 - \alpha - \sigma^{2}(\alpha))
    \end{pmatrix},
    \quad
    G_{5}'\vec{l}_2=
    \begin{pmatrix}
      1  \\
      \sigma^{2}(\alpha) \\
      \sigma^{3}(\alpha) \\
      \frac{1}{2}(2 - \sigma(\alpha) - \sigma^{3}(\alpha))
    \end{pmatrix},
  \]
  \[ 
    G_{5}'\vec{l}_3=
    \begin{pmatrix}
      1  \\
      \sigma^{3}(\alpha) \\
      \alpha \\
      \frac{1}{2}(2 - \sigma^{2}(\alpha) - \alpha)
    \end{pmatrix},
    \quad
    G_{5}'\vec{l}_4=
    \begin{pmatrix}
      1  \\
      \alpha \\
      \sigma(\alpha) \\
      \frac{1}{2}(2 - \sigma^{3}(\alpha) - \sigma(\alpha))
    \end{pmatrix},
  \]
 то есть $G_{5}' \big(\vec{l}_1, \vec{l}_2, \vec{l}_3, \vec{l}_4\big) = \big(\vec{l}_2, \vec{l}_3, \vec{l}_4, \vec{l}_1\big)$. Следовательно, $G_{5}'$ --- собственная циклическая симметрия цепной дроби $\cf(l_1,l_2,l_3,l_4)$. Обратно, предположим, что $G_{5}'$ --- собственная циклическая симметрия цепной дроби $\cf(l_1,l_2,l_3,l_4)$. Тогда существует такое $\mu_2$, что с точностью до перестановки индексов
   \[
    G_{5}'\vec{l}_1=
    \begin{pmatrix}
      1  \\
      \beta \\
      - \alpha + 2\gamma\\
      1 - \gamma
    \end{pmatrix}
    =\mu_2
    \begin{pmatrix}
      1  \\
      \sigma_{2}(\alpha) \\
      \sigma_{2}(\beta) \\
      \sigma_{2}(\gamma)
    \end{pmatrix}, 
  \]
откуда $\mu_2 = 1$, $\beta =\sigma_{2}(\alpha)$, $- \alpha + 2\gamma =\sigma_{2}(\beta)$, $ 1 - \gamma =\sigma_{2}(\gamma)$. Существует такое $\mu_3$, что
    \[ 
    G_{5}'\vec{l}_2=
    \begin{pmatrix}
      1  \\
      - \alpha + 2\gamma  \\
     2 - \beta - 2\gamma \\
      \gamma
    \end{pmatrix}
    =\mu_3
    \begin{pmatrix}
      1  \\
      \sigma_{3}(\alpha) \\
      \sigma_{3}(\beta) \\
      \sigma_{3}(\gamma)
    \end{pmatrix},
  \]
откуда $\mu_3 = 1$, $\frac{\alpha + \sigma_{3}(\alpha)}{2} = \gamma, \sigma_{3}(\alpha) = - \alpha + 2\gamma = \sigma_{2}(\beta) = \sigma^{2}_{2}(\alpha)$,  $\sigma_{3}(\beta) = 2 - \beta - 2\gamma = \sigma_{2}(-\alpha + 2\gamma) = \sigma^{2}_{2}(\beta)$, $\sigma_{3}(\gamma) = \gamma = \sigma_{2}(1 - \gamma) = \sigma^{2}_{2}(\gamma)$. Существует такое $\mu_4$, что
 \[ 
    G_{5}'\vec{l}_3=
    \begin{pmatrix}
      1  \\
      2 - \beta -  2\gamma  \\
      \alpha \\
      1 -\gamma
    \end{pmatrix}
    =\mu_4
    \begin{pmatrix}
      1  \\
      \sigma_{4}(\alpha) \\
      \sigma_{4}(\beta) \\
      \sigma_{4}(\gamma)
    \end{pmatrix},
  \]
откуда $\mu_4 = 1$, $\sigma_{4}(\alpha) = 2 - \beta -  2\gamma = \sigma_{2}(-\alpha) + \sigma_{2}(2\gamma) = \sigma_{2}(-\alpha + 2\gamma) = \sigma^{3}_{2}(\alpha)$, $\sigma_{4}(\beta) = \alpha = 2\gamma - \sigma_{2}(\beta)  = \sigma_{2}(2 - 2\gamma - \beta) =  \sigma^{3}_{2}(\beta)$, $\sigma_{4}(\gamma) = 1 -\gamma = \sigma_{2}(\gamma)  =  \sigma^{3}_{2}(\gamma)$, $\trace(\alpha) = 2$. Стало быть, $\cf(l_1, l_2, l_3, l_4) \in \mathbf{CF}_{5}'$, так как числа $1,\alpha,\beta, \gamma$ образуют базис поля $K=\Q(\alpha,\beta, \gamma)$.

Пусть $\cf(l_1, l_2, l_3, l_4) \in \mathbf{CF}_{6}'$. Тогда

 \[ 
    G_{6}'\vec{l}_1=
    \begin{pmatrix}
      1  \\
      \sigma(\alpha) \\
      \sigma^{2}(\alpha) \\
      \frac{1}{2}(1 - \alpha - \sigma^{2}(\alpha))
    \end{pmatrix},
    \quad
    G_{6}'\vec{l}_2=
    \begin{pmatrix}
      1  \\
      \sigma^{2}(\alpha) \\
      \sigma^{3}(\alpha) \\
      \frac{1}{2}(1 - \sigma(\alpha) - \sigma^{3}(\alpha))
    \end{pmatrix},
  \]
  \[ 
    G_{6}'\vec{l}_3=
    \begin{pmatrix}
      1  \\
      \sigma^{3}(\alpha) \\
      \alpha \\
      \frac{1}{2}(1 - \sigma^{2}(\alpha) - \alpha)
    \end{pmatrix},
    \quad
    G_{6}'\vec{l}_4=
    \begin{pmatrix}
      1  \\
      \alpha \\
      \sigma(\alpha) \\
      \frac{1}{2}(1 - \sigma^{3}(\alpha) - \sigma(\alpha))
    \end{pmatrix},
  \]
 то есть $G_{6}' \big(\vec{l}_1, \vec{l}_2, \vec{l}_3, \vec{l}_4\big) = \big(\vec{l}_2, \vec{l}_3, \vec{l}_4, \vec{l}_1\big)$. Следовательно, $G_{6}'$ --- собственная циклическая симметрия цепной дроби $\cf(l_1,l_2,l_3,l_4)$. Обратно, предположим, что $G_{6}'$ --- собственная циклическая симметрия цепной дроби $\cf(l_1,l_2,l_3,l_4)$. Тогда существует такое $\mu_2$, что с точностью до перестановки индексов
  \[
    G_{6}'\vec{l}_1=
    \begin{pmatrix}
      1  \\
      \beta \\
      -1 - \alpha + 2\gamma \\
      1 - \gamma
    \end{pmatrix}
    =\mu_2
    \begin{pmatrix}
      1  \\
      \sigma_{2}(\alpha) \\
      \sigma_{2}(\beta) \\
      \sigma_{2}(\gamma)
    \end{pmatrix}, 
  \]
откуда $\mu_2 = 1$, $\beta =\sigma_{2}(\alpha)$, $-1 - \alpha + 2\gamma  =\sigma_{2}(\beta)$, $1 - \gamma =\sigma_{2}(\gamma)$. Существует такое $\mu_3$, что
   \[ 
    G_{6}'\vec{l}_2=
    \begin{pmatrix}
      1  \\
      -1 - \alpha + 2\gamma  \\
      1 - \beta - 2\gamma\\
      \gamma
    \end{pmatrix}
    =\mu_3
    \begin{pmatrix}
      1  \\
      \sigma_{3}(\alpha) \\
      \sigma_{3}(\beta) \\
      \sigma_{3}(\gamma)
    \end{pmatrix},
  \]
откуда $\mu_3 = 1$, $\frac{\alpha + \sigma_{3}(\alpha) + 1}{2} = \gamma, \sigma_{3}(\alpha) = -1 - \alpha + 2\gamma = \sigma_{2}(\beta) = \sigma^{2}_{2}(\alpha)$, $\sigma_{3}(\beta) = 1 - \beta - 2\gamma = -1 -\beta + (2 - 2\gamma) = \sigma_{2}(-1 -\alpha + 2\gamma) = \sigma^{2}_{2}(\beta)$, $\sigma_{3}(\gamma) = \gamma = \sigma_{2}(1 - \gamma) = \sigma^{3}_{2}(\gamma)$. Существует такое $\mu_4$, что
\[ 
    G_{6}'\vec{l}_3=
    \begin{pmatrix}
      1  \\
      1 - \beta -  2\gamma \\
      \alpha \\
      1 - \gamma
    \end{pmatrix}
    =\mu_4
    \begin{pmatrix}
      1  \\
      \sigma_{4}(\alpha) \\
      \sigma_{4}(\beta) \\
      \sigma_{4}(\gamma)
    \end{pmatrix},
  \]
откуда $\mu_4 = 1$, $\sigma_{4}(\alpha) = 1 - \beta -  2\gamma = -1 -\beta + (2 - 2\gamma) = \sigma_{2}(-1 -\alpha + 2\gamma) = \sigma^{3}_{2}(\alpha)$, $\sigma_{4}(\beta) = \alpha = 1 -(-1 -\alpha + 2\gamma) - (2 - 2\gamma) = \sigma_{2}(1 -\beta - 2\gamma) = \sigma^{3}_{2}(\beta)$, $\sigma_{4}(\gamma) = 1 - \gamma = \sigma_{2}(\gamma) = \sigma^{3}_{2}(\gamma)$, $\trace(\alpha) = 0$. Стало быть, $\cf(l_1, l_2, l_3, l_4) \in \mathbf{CF}_{6}'$, так как числа $1,\alpha,\beta, \gamma$ образуют базис поля $K=\Q(\alpha,\beta, \gamma)$.

 Пусть $\cf(l_1, l_2, l_3, l_4) \in \mathbf{CF}_{7}'$. Тогда

\[ 
    G_{7}'\vec{l}_1=
    \begin{pmatrix}
      1  \\
      \sigma(\alpha) \\
      \sigma^{2}(\alpha) \\
      \frac{1}{2}(3 - \alpha - \sigma^{2}(\alpha))
    \end{pmatrix},
    \quad
    G_{7}'\vec{l}_2=
    \begin{pmatrix}
      1  \\
      \sigma^{2}(\alpha) \\
      \sigma^{3}(\alpha) \\
      \frac{1}{2}(3 - \sigma(\alpha) - \sigma^{3}(\alpha))
    \end{pmatrix},
  \]
  \[ 
    G_{7}'\vec{l}_3=
    \begin{pmatrix}
      1  \\
      \sigma^{3}(\alpha) \\
      \alpha \\
      \frac{1}{2}(3 - \sigma^{2}(\alpha) - \alpha)
    \end{pmatrix},
    \quad
    G_{7}'\vec{l}_4=
    \begin{pmatrix}
      1  \\
      \alpha \\
      \sigma(\alpha) \\
      \frac{1}{2}(3 - \sigma^{3}(\alpha) - \sigma(\alpha))
    \end{pmatrix},
  \]
 то есть $G_{7}' \big(\vec{l}_1, \vec{l}_2, \vec{l}_3, \vec{l}_4\big) = \big(\vec{l}_2, \vec{l}_3, \vec{l}_4, \vec{l}_1\big)$. Следовательно, $G_{7}'$ --- собственная циклическая симметрия цепной дроби $\cf(l_1,l_2,l_3,l_4)$. Обратно, предположим, что $G_{7}'$ --- собственная циклическая симметрия цепной дроби $\cf(l_1,l_2,l_3,l_4)$. Тогда существует такое $\mu_2$, что с точностью до перестановки индексов
  \[
    G_{7}'\vec{l}_1=
    \begin{pmatrix}
      1  \\
      \beta \\
      -1 - \alpha + 2\gamma \\
      2 - \gamma
    \end{pmatrix}
    =\mu_2
    \begin{pmatrix}
      1  \\
      \sigma_{2}(\alpha) \\
      \sigma_{2}(\beta) \\
      \sigma_{2}(\gamma)
    \end{pmatrix}, 
  \]
откуда $\mu_2 = 1$, $\beta =\sigma_{2}(\alpha)$, $-1 - \alpha + 2\gamma = \sigma_{2}(\beta)$, $2 - \gamma =\sigma_{2}(\gamma)$. Существует такое $\mu_3$, что
   \[ 
    G_{7}'\vec{l}_2=
    \begin{pmatrix}
      1  \\
      -1 - \alpha + 2\gamma  \\
      3 - \beta - 2\gamma\\
      \gamma
    \end{pmatrix}
    =\mu_3
    \begin{pmatrix}
      1  \\
      \sigma_{3}(\alpha) \\
      \sigma_{3}(\beta) \\
      \sigma_{3}(\gamma)
    \end{pmatrix},
  \]
откуда $\mu_3 = 1$, $\frac{\alpha + \sigma_{3}(\alpha) + 1}{2} = \gamma, \sigma_{3}(\alpha) = -1 - \alpha + 2\gamma = \sigma_{2}(\beta) = \sigma^{2}_{2}(\alpha)$, $\sigma_{3}(\beta) = 3 - \beta - 2\gamma = -1 -\beta + (4 - 2\gamma) = \sigma_{2}(-1 -\alpha + 2\gamma) = \sigma^{2}_{2}(\beta)$, $\sigma_{3}(\gamma) = \gamma = \sigma_{2}(2 - \gamma) = \sigma^{3}_{2}(\gamma)$. Существует такое $\mu_4$, что
\[ 
    G_{7}'\vec{l}_3=
    \begin{pmatrix}
      1  \\
      3 - \beta -  2\gamma \\
      \alpha \\
      2- \gamma
    \end{pmatrix}
    =\mu_4
    \begin{pmatrix}
      1  \\
      \sigma_{4}(\alpha) \\
      \sigma_{4}(\beta) \\
      \sigma_{4}(\gamma)
    \end{pmatrix},
  \]
откуда $\mu_4 = 1$, $\sigma_{4}(\alpha) = 3 - \beta -  2\gamma = -1 -\beta + (4 - 2\gamma) = \sigma_{2}(-1 -\alpha + 2\gamma) = \sigma^{3}_{2}(\alpha)$, $\sigma_{4}(\beta) = \alpha = 3 -(-\alpha + 2\gamma - 1) - (4 - 2\gamma) = \sigma_{2}(3 -\beta - 2\gamma) = \sigma^{3}_{2}(\beta)$, $\sigma_{4}(\gamma) = 2 - \gamma = \sigma_{2}(\gamma) = \sigma^{3}_{2}(\gamma)$, $\trace(\alpha) = 2$. Стало быть, $\cf(l_1, l_2, l_3, l_4) \in \mathbf{CF}_{7}'$, так как числа $1,\alpha,\beta, \gamma$ образуют базис поля $K=\Q(\alpha,\beta, \gamma)$.
\end{proof}

\section{Доказательство теорем \ref{theorem_proper_2_2} и \ref{theorem_proper_4}}\label{proof_theorem_proper_2_2}

Обозначим для каждого $i=1,\ldots,10$ через $\overline{ \widetilde{\mathbf{CF}_{i}}}$ образ $\widetilde{\mathbf{CF}_{i}}$ при действии группы $\Gl_4(\Z)$:
\[
  \overline{\widetilde{\mathbf{CF}_{i}}}=
  \Big\{ \cf(l_1,l_2,l_3,l_4)\in\gA_3 \,\Big|\, \exists X\in\Gl_4(\Z):X\big(\cf(l_1,l_2,l_3,l_4)\big)\in\mathbf{CF}_i \Big\}.
\]

Также обозначим для каждого $i=1,\ldots,7$ через $\overline{\mathbf{CF}_{i}'}$ образ $\mathbf{CF}_{i}'$ при действии группы $\Gl_4(\Z)$:
\[
  \overline{\mathbf{CF}_{i}'}=
  \Big\{ \cf(l_1,l_2,l_3,l_4)\in\gA_3' \,\Big|\, \exists X\in\Gl_4(\Z):X\big(\cf(l_1,l_2,l_3,l_4)\big)\in\mathbf{CF}_i \Big\}.
\]

\begin{lemma}\label{l:CF_instead_of_statements_2_2}
  Для дроби $\cf(l_1,l_2,l_3,l_4)\in\gA_3$ выполняется условие $(i)$ теоремы \ref{theorem_proper_2_2} тогда и только тогда, когда $\cf(l_1,l_2,l_3,l_4)$ принадлежит классу $\overline{\widetilde{\mathbf{CF}_{i}}}$, где $i\in\{1, 2,\ldots, 10\}$.
\end{lemma}

\begin{lemma}\label{l:CF_instead_of_statements_4}
  Для дроби $\cf(l_1,l_2,l_3,l_4)\in\gA_3$ выполняется условие $(i)$ теоремы \ref{theorem_proper_4} тогда и только тогда, когда $\cf(l_1,l_2,l_3,l_4)$ принадлежит классу $\overline{\mathbf{CF}_{i}'}$, где $i\in\{1,2,3,4,5,6,7\}$.
\end{lemma}

\begin{proof}[Доказательство леммы \ref{l:CF_instead_of_statements_2_2} и леммы \ref{l:CF_instead_of_statements_4}]
  Для любого $X\in\Gl_4(\Z)$ гиперболичность оператора $A\in\Gl_4(\Z)$ равносильна гиперболичности оператора $XAX^{-1}$. При этом собственные подпространства гиперболического оператора однозначно восстанавливаются по любому его собственному вектору. Остаётся воспользоваться определением эквивалентности из параграфа \ref{intro}.
 \end{proof}

Теорему \ref{theorem_proper_2_2} при помощи леммы \ref{l:CF_instead_of_statements_2_2}  можно переформулировать следующим образом: \emph{дробь $\cf(l_1,l_2,l_3,l_4)\in\gA_3$ имеет собственную симметрию $G$ тогда и только тогда, когда она принадлежит одному из классов $\overline{ \widetilde{\mathbf{CF}_{i}}}$, где $i\in\{1,2,\ldots,10\}$}.

Аналогично, теорему \ref{theorem_proper_4} при помощи леммы \ref{l:CF_instead_of_statements_4}  можно переформулировать следующим образом: \emph{дробь $\cf(l_1,l_2,l_3,l_4)\in\gA_3$ имеет собственную симметрию $G$ тогда и только тогда, когда она принадлежит одному из классов $\overline{\mathbf{CF}_{i}'}$, где $i\in\{1,2,\ldots,7\}$}.

  \begin{proof}[Доказательство теоремы \ref{theorem_proper_2_2}]
Если $\cf(l_1,l_2,l_3,l_4)$ принадлежит какому-то $\overline{\widetilde{\mathbf{CF}_{i}}}$, то по лемме \ref{oper_eq_4d_2} она имеет собственную симметрию $G$, ибо действие оператора из $\Gl_4(\Z)$ сохраняет свойство существования у алгебраической цепной дроби собственной симметрии.

Обратно, пусть дробь $\cf(l_1,l_2,l_3,l_4)\in\gA_3$ имеет собственную симметрию $G$. Положим $F=G'$ и рассмотрим точки $\vec z_1$, $\vec z_2$, $\vec z_3$, $\vec z_4$ из леммы \ref{main_lem_2_2}. Обозначим также через $\vec{e}_1$, $\vec{e}_2$, $\vec{e}_3$, $\vec{e}_4$ стандартный базис $\R^4$. Для точек $\vec z_1$, $\vec z_2$, $\vec z_3$, $\vec z_4$ выполняется хотя бы одно из утверждений \textup{(1)} - \textup{(11)} леммы \ref{main_lem_2_2}.

Пусть выполняется утверждение \textup{(1)} леммы \ref{main_lem_2_2}. Рассмотрим такой оператор $X_{1} \in \Gl_4(\Z)$, что
\[X_{1}\big(\vec{z}_1, \vec{z}_2, \vec{z}_3, \frac{1}{4}(\vec{z}_{1}+\vec{z}_{2}+\vec{z}_{3}+\vec{z}_{4})\big) = \big(\vec{e}_1 - \vec{e}_2, \vec{e}_1 + \vec{e}_4, \vec{e}_1 + \vec{e}_3 - \vec{e}_4, \vec{e}_1\big).\]
Тогда $X_{1}(\vec{z}_4) = X_{1}(4\cdot\frac{1}{4}(\vec{z}_{1}+\vec{z}_{2}+\vec{z}_{3}+\vec{z}_{4}) - \vec{z}_1 - \vec{z}_2 - \vec{z}_3) = \vec{e}_1 + \vec{e}_2 - \vec{e}_3$ и $X_{1}GX_{1}^{-1} = \widetilde{G_{1}}$, так как по лемме \ref{main_lem_2_2}
\[X_{1}GX_{1}^{-1}\big(\vec{e}_1 - \vec{e}_2, \vec{e}_1 + \vec{e}_4, \vec{e}_1 + \vec{e}_3 - \vec{e}_4, \vec{e}_1 + \vec{e}_2 - \vec{e}_3\big) = \]
\[\big(\vec{e}_1 + \vec{e}_3 - \vec{e}_4, \vec{e}_1 + \vec{e}_2 - \vec{e}_3, \vec{e}_1 - \vec{e}_2, \vec{e}_1 + \vec{e}_4\big).\]
Стало быть, $X_{1}\big(\cf(l_1,l_2,l_3,l_4)\big) \in \widetilde{\mathbf{CF}_{1}}$, то есть $\cf(l_1,l_2,l_3,l_4)\in\overline{\widetilde{\mathbf{CF}_{1}}}$.

Пусть выполняется утверждение \textup{(2)} леммы \ref{main_lem_2_2}. Рассмотрим такой оператор $X_{2} \in \Gl_4(\Z)$, что
\[X_{2}\big(\vec{z}_1, \vec{z}_2, \vec{z}_3, \vec{z}_{4}\big) = \big(\vec{e}_1, \vec{e}_1 + \vec{e}_4, \vec{e}_1 + \vec{e}_3, \vec{e}_1 + \vec{e}_2\big).\]
Тогда $X_{2}GX_{2}^{-1} = \widetilde{G_{2}}$, так как по лемме \ref{main_lem_2_2}
\[X_{2}GX_{2}^{-1}\big(\vec{e}_1, \vec{e}_1 + \vec{e}_4, \vec{e}_1 + \vec{e}_3, \vec{e}_1 + \vec{e}_2\big) = \]
\[\big(\vec{e}_1 + \vec{e}_3, \vec{e}_1 + \vec{e}_2, \vec{e}_1, \vec{e}_1 + \vec{e}_4\big).\]
Стало быть, $X_{2}\big(\cf(l_1,l_2,l_3,l_4)\big) \in \widetilde{\mathbf{CF}_{2}}$, то есть $\cf(l_1,l_2,l_3,l_4)\in\overline{\widetilde{\mathbf{CF}_{2}}}$.

Пусть выполняется утверждение \textup{(3)} леммы \ref{main_lem_2_2}. Рассмотрим такой оператор $X_{3} \in \Gl_4(\Z)$, что
\[X_{3}\big(\vec{z}_1, \frac{1}{2}(\vec{z}_{1}+\vec{z}_{2}), \frac{1}{2}(\vec{z}_{1}+\vec{z}_{3}), \frac{1}{2}(\vec{z}_{1}+\vec{z}_{4})\big) = \big(\vec{e}_1, \vec{e}_1 + \vec{e}_4, \vec{e}_1 + \vec{e}_3, \vec{e}_1 + \vec{e}_2 \big).\]
Тогда $X_{3}(\vec{z}_2) = X_{3}(2\cdot\frac{1}{2}(\vec{z}_1 + \vec{z}_2) - \vec{z}_1) = \vec{e}_1  + 2\vec{e}_4$, $X_{3}(\vec{z}_3) = X_{3}(2\cdot\frac{1}{2}(\vec{z}_1 + \vec{z}_3) - \vec{z}_1) = \vec{e}_1  + 2\vec{e}_3$, $X_{3}(\vec{z}_4) = X_{3}(2\cdot\frac{1}{2}(\vec{z}_{1}+\vec{z}_{4})-\vec{z}_1) = \vec{e}_1 + 2\vec{e}_2$ и $X_{3}GX_{3}^{-1} = \widetilde{G_{3}}$, так как по лемме \ref{main_lem_2_2}
\[X_{3}GX_{3}^{-1}\big(\vec{e}_1, \vec{e}_1  + 2\vec{e}_4, \vec{e}_1 + 2\vec{e}_3, \vec{e}_1 + 2\vec{e}_2\big) = \]
\[\big(\vec{e}_1 + 2\vec{e}_3, \vec{e}_1 + 2\vec{e}_2, \vec{e}_1, \vec{e}_1  + 2\vec{e}_4\big).\]
Стало быть, $X_{3}\big(\cf(l_1,l_2,l_3,l_4)\big) \in \widetilde{\mathbf{CF}_{3}}$, то есть $\cf(l_1,l_2,l_3,l_4)\in\overline{\widetilde{\mathbf{CF}_{3}}}$.

Пусть выполняется утверждение \textup{(4)} леммы \ref{main_lem_2_2}. Рассмотрим такой оператор $X_{4} \in \Gl_4(\Z)$, что
\[X_{4}\big(\vec{z}_1, \vec{z}_2, \frac{1}{2}(\vec{z}_1 + \vec{z}_3),  \frac{1}{4}(\vec{z}_{1}+\vec{z}_{2}+\vec{z}_{3}+\vec{z}_{4})\big) = \big(\vec{e}_1 - \vec{e}_3 + \vec{e}_4, \vec{e}_1 -\vec{e}_2 + 2\vec{e}_3 - \vec{e}_4, \vec{e}_1 + \vec{e}_2 - \vec{e}_3 +\vec{e}_4, \vec{e}_1\big).\]
Тогда $X_{4}(\vec{z}_3) = X_{4}(2\cdot\frac{1}{2}(\vec{z}_1 + \vec{z}_3) - \vec{z}_1) = \vec{e}_1 + 2\vec{e}_2 - \vec{e}_3 + \vec{e}_4$, $X_{4}(\vec{z}_4) = X_{4}(4\cdot\frac{1}{4}(\vec{z}_{1}+\vec{z}_{2}+\vec{z}_{3}+\vec{z}_{4}) - \vec{z}_1 - \vec{z}_2 - \vec{z}_3) = \vec{e}_1 - \vec{e}_2 - \vec{e}_4$ и $X_{4}GX_{4}^{-1} = \widetilde{G_{4}}$, так как по лемме \ref{main_lem_2_2}
\[X_{4}GX_{4}^{-1}\big(\vec{e}_1 - \vec{e}_3 + \vec{e}_4, \vec{e}_1 -\vec{e}_2 + 2\vec{e}_3 - \vec{e}_4, \vec{e}_1 + 2\vec{e}_2 - \vec{e}_3 + \vec{e}_4, \vec{e}_1 - \vec{e}_2 - \vec{e}_4\big) = \]
\[\big(\vec{e}_1 + 2\vec{e}_2 - \vec{e}_3 + \vec{e}_4, \vec{e}_1 - \vec{e}_2 - \vec{e}_4, \vec{e}_1 - \vec{e}_3 + \vec{e}_4, \vec{e}_1 -\vec{e}_2 + 2\vec{e}_3 - \vec{e}_4\big).\]
Стало быть, $X_{4}\big(\cf(l_1,l_2,l_3,l_4)\big) \in \widetilde{\mathbf{CF}_{4}}$, то есть $\cf(l_1,l_2,l_3,l_4)\in\overline{\widetilde{\mathbf{CF}_{4}}}$.

Пусть выполняется утверждение \textup{(5)} леммы \ref{main_lem_2_2}. Рассмотрим такой оператор $X_{5} \in \Gl_4(\Z)$, что
\[X_{5}\big(\vec{z}_1, \vec{z}_2, \frac{1}{2}(\vec{z}_1 + \vec{z}_3),  \frac{1}{2}(\vec{z}_2 + \vec{z}_4)\big) = \big(\vec{e}_1, \vec{e}_1 + \vec{e}_4, \vec{e}_1 + \vec{e}_3, \vec{e}_1 + \vec{e}_2\big).\]
Тогда $X_{5}(\vec{z}_3) = X_{5}(2\cdot\frac{1}{2}(\vec{z}_1 + \vec{z}_3) - \vec{z}_1) = \vec{e}_1 + 2\vec{e}_3$, $X_{5}(\vec{z}_4) = X_{5}(2\cdot\frac{1}{2}(\vec{z}_2 + \vec{z}_4) - \vec{z}_2) = \vec{e}_1 + 2\vec{e}_2 + \vec{e}_4$ и $X_{5}GX_{5}^{-1} = \widetilde{G_{5}}$, так как по лемме \ref{main_lem_2_2}
\[X_{5}GX_{5}^{-1}\big(\vec{e}_1, \vec{e}_1 + \vec{e}_4, \vec{e}_1 + 2\vec{e}_3, \vec{e}_1 + 2\vec{e}_2 + \vec{e}_4\big) = \]
\[\big( \vec{e}_1 + 2\vec{e}_3, \vec{e}_1 + 2\vec{e}_2 + \vec{e}_4, \vec{e}_1, \vec{e}_1 + \vec{e}_4\big).\]
Стало быть, $X_{5}\big(\cf(l_1,l_2,l_3,l_4)\big) \in \widetilde{\mathbf{CF}_{5}}$, то есть $\cf(l_1,l_2,l_3,l_4)\in\overline{\widetilde{\mathbf{CF}_{5}}}$.

Пусть выполняется утверждение \textup{(6)} леммы \ref{main_lem_2_2}. Рассмотрим такой оператор $X_{6} \in \Gl_4(\Z)$, что
\[X_{6}\big(\vec{z}_1, \vec{z}_{2}, \vec{z}_{3}, \frac{1}{2}(\vec{z}_{1} + \vec{z}_{3} + \vec{z}_{4} - \vec{z}_{2})\big) =\]
\[\big(\vec{e}_1 + \vec{e}_2 + \vec{e}_4, \vec{e}_1, \vec{e}_1 - \vec{e}_3 + \vec{e}_4, \vec{e}_1 + \vec{e}_4\big).\]
Тогда $X_{6}(\vec{z}_4) = X_{6}(2\cdot\frac{1}{2}(\vec{z}_{1} + \vec{z}_{3} + \vec{z}_{4} - \vec{z}_{2}) - \vec{z}_1 + \vec{z}_2 - \vec{z}_3) = \vec{e}_1 - \vec{e}_2 + \vec{e}_3$ и $X_{6}GX_{6}^{-1} = \widetilde{G_{6}}$, так как по лемме \ref{main_lem_2_2}
\[X_{6}GX_{6}^{-1}\big(\vec{e}_1 + \vec{e}_2 + \vec{e}_4, \vec{e}_1, \vec{e}_1 - \vec{e}_3 + \vec{e}_4, \vec{e}_1 - \vec{e}_2 + \vec{e}_3\big) = \]
\[\big(\vec{e}_1 - \vec{e}_3 + \vec{e}_4, \vec{e}_1 - \vec{e}_2 + \vec{e}_3, \vec{e}_1 + \vec{e}_2 + \vec{e}_4, \vec{e}_1\big).\]
Стало быть, $X_{6}\big(\cf(l_1,l_2,l_3,l_4)\big) \in \widetilde{\mathbf{CF}_{6}}$, то есть $\cf(l_1,l_2,l_3,l_4)\in\overline{\widetilde{\mathbf{CF}_{6}}}$.

Пусть выполняется утверждение \textup{(7)} леммы \ref{main_lem_2_2}. Рассмотрим такой оператор $X_{7} \in \Gl_4(\Z)$, что
\[X_{7}\big(\vec{z}_1, \vec{z}_{2}, \vec{z}_{3}, \frac{1}{2}(\vec{z}_{1}+\vec{z}_{2}) + \frac{1}{4}(\vec{z}_{1}+\vec{z}_{4} - \vec{z}_{3} - \vec{z}_{2})\big) =\]
\[\big(\vec{e}_1+\vec{e}_2-\vec{e}_3+2\vec{e}_4, \vec{e}_1-\vec{e}_2+2\vec{e}_3, \vec{e}_1 + 2\vec{e}_2 + 2\vec{e}_4, \vec{e}_1+\vec{e}_4\big).\]
Тогда $X_{7}(\vec{z}_4) = X_{7}(4(\frac{3\vec{z}_1+\vec{z}_2-\vec{z}_3+\vec{z}_4}{4})-3\vec{z}_1-\vec{z}_2+\vec{z}_3) = \vec{e}_1  + \vec{e}_3$ и $X_{7}GX_{7}^{-1} = \widetilde{G_{7}}$, так как по лемме \ref{main_lem_2_2}
\[X_{7}GX_{7}^{-1}\big(\vec{e}_1+\vec{e}_2-\vec{e}_3+2\vec{e}_4, \vec{e}_1-\vec{e}_2+2\vec{e}_3, \vec{e}_1 + 2\vec{e}_2 + 2\vec{e}_4, \vec{e}_1  + \vec{e}_3\big) = \]
\[\big(\vec{e}_1 + 2\vec{e}_2 + 2\vec{e}_4, \vec{e}_1  + \vec{e}_3, \vec{e}_1+\vec{e}_2-\vec{e}_3+2\vec{e}_4, \vec{e}_1-\vec{e}_2+2\vec{e}_3\big).\]
Стало быть, $X_{7}\big(\cf(l_1,l_2,l_3,l_4)\big) \in \widetilde{\mathbf{CF}_{7}}$, то есть $\cf(l_1,l_2,l_3,l_4)\in\overline{\widetilde{\mathbf{CF}_{7}}}$.

Пусть выполняется утверждение \textup{(8)} леммы \ref{main_lem_2_2}. Рассмотрим такой оператор $X_{8} \in \Gl_4(\Z)$, что
\[X_{8}\big(\vec{z}_1, \vec{z}_2, \vec{z}_3, \frac{1}{2}(\vec{z}_{2} + \vec{z}_{4})\big) = \big(\vec{e}_1, \vec{e}_1 + \vec{e}_4, \vec{e}_1 + \vec{e}_3, \vec{e}_1 + \vec{e}_2 + \vec{e}_4\big).\]
Тогда $X_{8}(\vec{z}_4) = X_{8}(2\cdot\frac{1}{2}(\vec{z}_{2} + \vec{z}_{4}) - \vec{z}_2 = \vec{e}_1 + 2 \vec{e}_2 + \vec{e}_4$ и $X_{8}GX_{8}^{-1} = \widetilde{G_{8}}$, так как по лемме \ref{main_lem_2_2}
\[X_{8}GX_{8}^{-1}\big(\vec{e}_1, \vec{e}_1 + \vec{e}_4, \vec{e}_1 + \vec{e}_3, \vec{e}_1 + 2 \vec{e}_2 + \vec{e}_4\big) = \]
\[\big( \vec{e}_1 + \vec{e}_3, \vec{e}_1 + 2 \vec{e}_2 + \vec{e}_4, \vec{e}_1, \vec{e}_1 + \vec{e}_4\big).\]
Стало быть, $X_{8}\big(\cf(l_1,l_2,l_3,l_4)\big) \in \widetilde{\mathbf{CF}_{8}}$, то есть $\cf(l_1,l_2,l_3,l_4)\in\overline{\widetilde{\mathbf{CF}_{8}}}$.

Пусть выполняется утверждение \textup{(9)} леммы \ref{main_lem_2_2}. Рассмотрим такой оператор $X_{9} \in \Gl_4(\Z)$, что
\[X_{9}\big(\vec{z}_1, \vec{z}_2, \frac{1}{2}(\vec{z}_{1} + \vec{z}_{3}), \frac{1}{4}\vec{z}_{1} + \frac{1}{2}\vec{z}_{2} - \frac{1}{4}\vec{z}_{3} + \frac{1}{2}\vec{z}_{4}\big) = \big(\vec{e}_1, \vec{e}_1 + \vec{e}_4, \vec{e}_1 + \vec{e}_3, \vec{e}_1 + \vec{e}_2 + \vec{e}_4\big).\]
Тогда $X_{9}(\vec{z}_3) = X_{9}(2\cdot\frac{1}{2}(\vec{z}_{1} + \vec{z}_{3}) - \vec{z}_1) = \vec{e}_1 + 2 \vec{e}_3$, $X_{9}(\vec{z}_4) = X_{9}(2\cdot(\frac{1}{4}\vec{z}_{1} + \frac{1}{2}\vec{z}_{2} - \frac{1}{4}\vec{z}_{3} + \frac{1}{2}\vec{z}_{4}) - \frac{1}{2}\vec{z}_1 - \vec{z}_2 + \frac{1}{2}\vec{z}_3 = \vec{e}_1 + 2 \vec{e}_2 + \vec{e}_3 + \vec{e}_4$ и $X_{9}GX_{9}^{-1} = \widetilde{G_{9}}$, так как по лемме \ref{main_lem_2_2}
\[X_{9}GX_{9}^{-1}\big(\vec{e}_1, \vec{e}_1 + \vec{e}_4, \vec{e}_1 + 2 \vec{e}_3, \vec{e}_1 + 2 \vec{e}_2 + \vec{e}_3 + \vec{e}_4\big) = \]
\[\big( \vec{e}_1 + 2 \vec{e}_3, \vec{e}_1 + 2 \vec{e}_2 + \vec{e}_3 + \vec{e}_4, \vec{e}_1, \vec{e}_1 + \vec{e}_4\big).\]
Стало быть, $X_{9}\big(\cf(l_1,l_2,l_3,l_4)\big) \in \widetilde{\mathbf{CF}_{9}}$, то есть $\cf(l_1,l_2,l_3,l_4)\in\overline{\widetilde{\mathbf{CF}_{9}}}$.

Пусть выполняется утверждение \textup{(10)} леммы \ref{main_lem_2_2}. Рассмотрим такой оператор $X_{10} \in \Gl_4(\Z)$, что
\[X_{10}\big(\vec{z}_1, \vec{z}_2, \frac{1}{2}(\vec{z}_{1} + \vec{z}_{3}), \frac{1}{2}\vec{z}_{1} + \frac{1}{4}\vec{z}_{2} + \frac{1}{4}\vec{z}_{4}\big) = \big(\vec{e}_1, \vec{e}_1 + \vec{e}_4, \vec{e}_1 + \vec{e}_3, \vec{e}_1 + \vec{e}_2\big).\]
Тогда $X_{10}(\vec{z}_3) = X_{10}(2\cdot\frac{1}{2}(\vec{z}_{1} + \vec{z}_{3}) - \vec{z}_1 = \vec{e}_1 + 2 \vec{e}_3$, $X_{10}(\vec{z}_4) = X_{10}(4\cdot(\frac{1}{2}\vec{z}_{1} + \frac{1}{4}\vec{z}_{2} + \frac{1}{4}\vec{z}_{4}) - 2\vec{z}_1 - \vec{z}_2) = \vec{e}_1 + 4 \vec{e}_2 - \vec{e}_4$ и $X_{10}GX_{10}^{-1} = \widetilde{G_{10}}$, так как по лемме \ref{main_lem_2_2}
\[X_{10}GX_{10}^{-1}\big(\vec{e}_1, \vec{e}_1 + \vec{e}_4, \vec{e}_1 + 2 \vec{e}_3, \vec{e}_1 + 4 \vec{e}_2 - \vec{e}_4\big) = \]
\[\big(\vec{e}_1 + 2 \vec{e}_3, \vec{e}_1 + 4 \vec{e}_2 - \vec{e}_4, \vec{e}_1, \vec{e}_1 + \vec{e}_4\big).\]
Стало быть, $X_{10}\big(\cf(l_1,l_2,l_3,l_4)\big) \in \widetilde{\mathbf{CF}_{10}}$, то есть $\cf(l_1,l_2,l_3,l_4)\in\overline{\widetilde{\mathbf{CF}_{10}}}$.

Пусть выполняется утверждение \textup{(11)} леммы \ref{main_lem_2_2}. Рассмотрим такой оператор $X_{11} \in \Gl_4(\Z)$, что
\[X_{11}\big(\vec{z}_1, \frac{1}{2}(\vec{z}_1 + \vec{z}_2), \vec{z}_3, \frac{1}{2}(\vec{z}_1 - \vec{z}_2 + \vec{z}_3 + \vec{z}_4)\big) = \big(\vec{e}_1, \vec{e}_1 + \vec{e}_4, \vec{e}_1 + \vec{e}_3, \vec{e}_1 + \vec{e}_2 - \vec{e}_4\big).\]
Тогда $X_{11}(\vec{z}_2) = X_{11}(2\cdot\frac{1}{2}(\vec{z}_1 + \vec{z}_2) - \vec{z}_1) = \vec{e}_1 + 2\vec{e}_4$, $X_{11}(\vec{z}_4) = X_{11}(2\cdot\frac{1}{2}(\vec{z}_1 - \vec{z}_2 + \vec{z}_3 + \vec{z}_4) - \vec{z}_1 + \vec{z}_2 -  \vec{z}_3) = \vec{e}_1 + 2\vec{e}_2 - \vec{e}_3$ и $X_{11}GX_{11}^{-1} = \widetilde{G_{2}}$, так как по лемме \ref{main_lem_2_2}
\[X_{11}GX_{11}^{-1}\big(\vec{e}_1, \vec{e}_1 + 2\vec{e}_4, \vec{e}_1 + \vec{e}_3, \vec{e}_1 + 2\vec{e}_2 - \vec{e}_3\big) = \]
\[\big(\vec{e}_1 + \vec{e}_3, \vec{e}_1 + 2\vec{e}_2 - \vec{e}_3, \vec{e}_1, \vec{e}_1 + 2\vec{e}_4\big).\]
Стало быть, $X_{11}\big(\cf(l_1,l_2,l_3,l_4)\big) \in \widetilde{\mathbf{CF}_{2}}$, то есть $\cf(l_1,l_2,l_3,l_4)\in\overline{\widetilde{\mathbf{CF}_{2}}}$.
\end{proof}

\begin{proof}[Доказательство теоремы \ref{theorem_proper_4}]
Если $\cf(l_1,l_2,l_3,l_4)$ принадлежит какому-то $\overline{\mathbf{CF}}_i$, то по лемме \ref{oper_eq_4d} она имеет собственную циклическую симметрию $G$, ибо действие оператора из $\Gl_4(\Z)$ сохраняет свойство существования у алгебраической цепной дроби собственной циклической симметрии.

Обратно, пусть дробь $\cf(l_1,l_2,l_3,l_4)\in\gA_3$ имеет собственную циклическую симметрию $G$ с неподвижной точкой на некотором парусе $\partial(\cK(C)) \in \cf(l_1,l_2, l_3, l_4)$. Рассмотрим точки $\vec z_1$, $\vec z_2$, $\vec z_3$, $\vec z_4$ из леммы \ref{main_lem_4}. Обозначим также через $\vec{e}_1$, $\vec{e}_2$, $\vec{e}_3$, $\vec{e}_4$ стандартный базис $\R^4$. Для точек $\vec z_1$, $\vec z_2$, $\vec z_3$, $\vec z_4$ выполняется хотя бы одно из утверждений \textup{(1)} - \textup{(7)} леммы \ref{main_lem_4}.

Пусть выполняется утверждение \textup{(1)} леммы \ref{main_lem_4}. Рассмотрим такой оператор $X_{1} \in \Gl_4(\Z)$, что
\[X_{1}\big(\vec{z}_1, \vec{z}_2, \vec{z}_3, \frac{1}{4}(\vec{z}_{1}+\vec{z}_{2}+\vec{z}_{3}+\vec{z}_{4})\big) = \big(\vec{e}_1 - \vec{e}_2, \vec{e}_1 + \vec{e}_4, \vec{e}_1 + \vec{e}_3 - \vec{e}_4, \vec{e}_1\big).\]
Тогда $X_{1}(\vec{z}_4) = X_{1}(4\cdot\frac{1}{4}(\vec{z}_{1}+\vec{z}_{2}+\vec{z}_{3}+\vec{z}_{4}) - \vec{z}_1 - \vec{z}_2 - \vec{z}_3) = \vec{e}_1 + \vec{e}_2 - \vec{e}_3$ и $X_{1}GX_{1}^{-1} = G_{1}'$, так как по лемме \ref{main_lem_4}
\[X_{1}GX_{1}^{-1}\big(\vec{e}_1 - \vec{e}_2, \vec{e}_1 + \vec{e}_4, \vec{e}_1 + \vec{e}_3 - \vec{e}_4, \vec{e}_1 + \vec{e}_2 - \vec{e}_3\big) = \]
\[\big(\vec{e}_1 + \vec{e}_4, \vec{e}_1 + \vec{e}_3 - \vec{e}_4, \vec{e}_1 + \vec{e}_2 - \vec{e}_3, \vec{e}_1 - \vec{e}_2\big).\]
Стало быть, $X_{1}\big(\cf(l_1,l_2,l_3,l_4)\big) \in \mathbf{CF}_{1}'$, то есть $\cf(l_1,l_2,l_3,l_4)\in\overline{\mathbf{CF}_{1}'}$.

Пусть выполняется утверждение \textup{(2)} леммы \ref{main_lem_4}. Рассмотрим такой оператор $X_{2} \in \Gl_4(\Z)$, что
\[X_{2}\big(\vec{z}_1, \vec{z}_2, \vec{z}_3, \vec{z}_{4}\big) = \big(\vec{e}_1, \vec{e}_1 + \vec{e}_4, \vec{e}_1 + \vec{e}_3, \vec{e}_1 + \vec{e}_2\big).\]
Тогда $X_{2}GX_{2}^{-1} = G_{2}'$, так как по лемме \ref{main_lem_4}
\[X_{2}GX_{2}^{-1}\big(\vec{e}_1, \vec{e}_1 + \vec{e}_4, \vec{e}_1 + \vec{e}_3, \vec{e}_1 + \vec{e}_2\big) = \]
\[\big( \vec{e}_1 + \vec{e}_4, \vec{e}_1 + \vec{e}_3, \vec{e}_1 + \vec{e}_2, \vec{e}_1\big).\]
Стало быть, $X_{2}\big(\cf(l_1,l_2,l_3,l_4)\big) \in \mathbf{CF}_{2}'$, то есть $\cf(l_1,l_2,l_3,l_4)\in\overline{\mathbf{CF}_{2}'}$.

Пусть выполняется утверждение \textup{(3)} леммы \ref{main_lem_4}. Рассмотрим такой оператор $X_{3} \in \Gl_4(\Z)$, что
\[X_{3}\big(\vec{z}_1, \frac{1}{2}(\vec{z}_{1}+\vec{z}_{2}), \frac{1}{2}(\vec{z}_{1}+\vec{z}_{3}), \frac{1}{2}(\vec{z}_{1}+\vec{z}_{4})\big) = \big(\vec{e}_1, \vec{e}_1 + \vec{e}_4, \vec{e}_1 + \vec{e}_3, \vec{e}_1 + \vec{e}_2 \big).\]
Тогда $X_{3}(\vec{z}_2) = X_{3}(2\cdot\frac{1}{2}(\vec{z}_1 + \vec{z}_2) - \vec{z}_1) = \vec{e}_1  + 2\vec{e}_4$, $X_{3}(\vec{z}_3) = X_{3}(2\cdot\frac{1}{2}(\vec{z}_1 + \vec{z}_3) - \vec{z}_1) = \vec{e}_1  + 2\vec{e}_3$, $X_{3}(\vec{z}_4) = X_{3}(2\cdot\frac{1}{2}(\vec{z}_{1}+\vec{z}_{4})-\vec{z}_1) = \vec{e}_1 + 2\vec{e}_2$ и $X_{3}GX_{3}^{-1} = G_{3}'$, так как по лемме \ref{main_lem_4}
\[X_{3}GX_{3}^{-1}\big(\vec{e}_1, \vec{e}_1  + 2\vec{e}_4, \vec{e}_1 + 2\vec{e}_3, \vec{e}_1 + 2\vec{e}_2\big) = \]
\[\big(\vec{e}_1  + 2\vec{e}_4, \vec{e}_1 + 2\vec{e}_3, \vec{e}_1 + 2\vec{e}_2, \vec{e}_1\big).\]
Стало быть, $X_{3}\big(\cf(l_1,l_2,l_3,l_4)\big) \in \mathbf{CF}_{3}'$, то есть $\cf(l_1,l_2,l_3,l_4)\in\overline{\mathbf{CF}_{3}'}$.

Пусть выполняется утверждение \textup{(4)} леммы \ref{main_lem_4}. Рассмотрим такой оператор $X_{4} \in \Gl_4(\Z)$, что
\[X_{4}\big(\vec{z}_1, \vec{z}_2, \frac{1}{2}(\vec{z}_1 + \vec{z}_3),  \frac{1}{4}(\vec{z}_{1}+\vec{z}_{2}+\vec{z}_{3}+\vec{z}_{4})\big) = \big(\vec{e}_1 - \vec{e}_3 + \vec{e}_4, \vec{e}_1 -\vec{e}_2 + 2\vec{e}_3 - \vec{e}_4, \vec{e}_1 + \vec{e}_2 - \vec{e}_3 +\vec{e}_4, \vec{e}_1\big).\]
Тогда $X_{4}(\vec{z}_3) = X_{4}(2\cdot\frac{1}{2}(\vec{z}_1 + \vec{z}_3) - \vec{z}_1) = \vec{e}_1 + 2\vec{e}_2 - \vec{e}_3 + \vec{e}_4$, $X_{4}(\vec{z}_4) = X_{4}(4\cdot\frac{1}{4}(\vec{z}_{1}+\vec{z}_{2}+\vec{z}_{3}+\vec{z}_{4}) - \vec{z}_1 - \vec{z}_2 - \vec{z}_3) = \vec{e}_1 - \vec{e}_2 - \vec{e}_4$ и $X_{4}GX_{4}^{-1} = G_{4}'$, так как по лемме \ref{main_lem_4}
\[X_{4}GX_{4}^{-1}\big(\vec{e}_1 - \vec{e}_3 + \vec{e}_4, \vec{e}_1 -\vec{e}_2 + 2\vec{e}_3 - \vec{e}_4, \vec{e}_1 + 2\vec{e}_2 - \vec{e}_3 + \vec{e}_4, \vec{e}_1 - \vec{e}_2 - \vec{e}_4\big) = \]
\[\big(\vec{e}_1 -\vec{e}_2 + 2\vec{e}_3 - \vec{e}_4, \vec{e}_1 + 2\vec{e}_2 - \vec{e}_3 + \vec{e}_4, \vec{e}_1 - \vec{e}_2 - \vec{e}_4, \vec{e}_1 - \vec{e}_3 + \vec{e}_4\big).\]
Стало быть, $X_{4}\big(\cf(l_1,l_2,l_3,l_4)\big) \in \mathbf{CF}_{4}'$, то есть $\cf(l_1,l_2,l_3,l_4)\in\overline{\mathbf{CF}_{4}'}$.

Пусть выполняется утверждение \textup{(5)} леммы \ref{main_lem_4}. Рассмотрим такой оператор $X_{5} \in \Gl_4(\Z)$, что
\[X_{5}\big(\vec{z}_1, \vec{z}_2, \frac{1}{2}(\vec{z}_1 + \vec{z}_3),  \frac{1}{2}(\vec{z}_2 + \vec{z}_4)\big) = \big(\vec{e}_1, \vec{e}_1 + \vec{e}_4, \vec{e}_1 + \vec{e}_3, \vec{e}_1 + \vec{e}_2\big).\]
Тогда $X_{5}(\vec{z}_3) = X_{5}(2\cdot\frac{1}{2}(\vec{z}_1 + \vec{z}_3) - \vec{z}_1) = \vec{e}_1 + 2\vec{e}_3$, $X_{5}(\vec{z}_4) = X_{5}(2\cdot\frac{1}{2}(\vec{z}_2 + \vec{z}_4) - \vec{z}_2) = \vec{e}_1 + 2\vec{e}_2 + \vec{e}_4$ и $X_{5}GX_{5}^{-1} = G_{5}'$, так как по лемме \ref{main_lem_4}
\[X_{5}GX_{5}^{-1}\big(\vec{e}_1, \vec{e}_1 + \vec{e}_4, \vec{e}_1 + 2\vec{e}_3, \vec{e}_1 + 2\vec{e}_2 + \vec{e}_4\big) = \]
\[\big( \vec{e}_1 + \vec{e}_4, \vec{e}_1 + 2\vec{e}_3, \vec{e}_1 + 2\vec{e}_2 + \vec{e}_4, \vec{e}_1\big).\]
Стало быть, $X_{5}\big(\cf(l_1,l_2,l_3,l_4)\big) \in \mathbf{CF}_{5}'$, то есть $\cf(l_1,l_2,l_3,l_4)\in\overline{\mathbf{CF}_{5}'}$.

Пусть выполняется утверждение \textup{(6)} леммы \ref{main_lem_4}. Рассмотрим такой оператор $X_{6} \in \Gl_4(\Z)$, что
\[X_{6}\big(\vec{z}_1, \vec{z}_{2}, \vec{z}_{3}, \frac{1}{2}(\vec{z}_{1} + \vec{z}_{3} + \vec{z}_{4} - \vec{z}_{2})\big) =\]
\[\big(\vec{e}_1 + \vec{e}_2 + \vec{e}_4, \vec{e}_1, \vec{e}_1 - \vec{e}_3 + \vec{e}_4, \vec{e}_1 + \vec{e}_4\big).\]
Тогда $X_{6}(\vec{z}_4) = X_{6}(2\cdot\frac{1}{2}(\vec{z}_{1} + \vec{z}_{3} + \vec{z}_{4} - \vec{z}_{2}) - \vec{z}_1 + \vec{z}_2 - \vec{z}_3) = \vec{e}_1 - \vec{e}_2 + \vec{e}_3$ и $X_{6}GX_{6}^{-1} = G_{6}'$, так как по лемме \ref{main_lem_4}
\[X_{6}GX_{6}^{-1}\big(\vec{e}_1 + \vec{e}_2 + \vec{e}_4, \vec{e}_1, \vec{e}_1 - \vec{e}_3 + \vec{e}_4, \vec{e}_1 - \vec{e}_2 + \vec{e}_3\big) = \]
\[\big(\vec{e}_1, \vec{e}_1 - \vec{e}_3 + \vec{e}_4, \vec{e}_1 - \vec{e}_2 + \vec{e}_3, \vec{e}_1 + \vec{e}_2 + \vec{e}_4\big).\]
Стало быть, $X_{6}\big(\cf(l_1,l_2,l_3,l_4)\big) \in \mathbf{CF}_{6}'$, то есть $\cf(l_1,l_2,l_3,l_4)\in\overline{\mathbf{CF}_{6}'}$.

Пусть выполняется утверждение \textup{(7)} леммы \ref{main_lem_4}. Рассмотрим такой оператор $X_{7} \in \Gl_4(\Z)$, что
\[X_{7}\big(\vec{z}_1, \vec{z}_{2}, \vec{z}_{3}, \frac{1}{2}(\vec{z}_{1}+\vec{z}_{2}) + \frac{1}{4}(\vec{z}_{1}+\vec{z}_{4} - \vec{z}_{3} - \vec{z}_{2})\big) =\]
\[\big(\vec{e}_1+\vec{e}_2-\vec{e}_3+2\vec{e}_4, \vec{e}_1-\vec{e}_2+2\vec{e}_3, \vec{e}_1 + 2\vec{e}_2 + 2\vec{e}_4, \vec{e}_1+\vec{e}_4\big).\]
Тогда $X_{7}(\vec{z}_4) = X_{7}(4(\frac{3\vec{z}_1+\vec{z}_2-\vec{z}_3+\vec{z}_4}{4})-3\vec{z}_1-\vec{z}_2+\vec{z}_3) = \vec{e}_1  + \vec{e}_3$ и $X_{7}GX_{7}^{-1} = G_{7}'$, так как по лемме \ref{main_lem_4}
\[X_{7}GX_{7}^{-1}\big(\vec{e}_1+\vec{e}_2-\vec{e}_3+2\vec{e}_4, \vec{e}_1-\vec{e}_2+2\vec{e}_3, \vec{e}_1 + 2\vec{e}_2 + 2\vec{e}_4, \vec{e}_1  + \vec{e}_3\big) = \]
\[\big(\vec{e}_1-\vec{e}_2+2\vec{e}_3, \vec{e}_1 + 2\vec{e}_2 + 2\vec{e}_4, \vec{e}_1  + \vec{e}_3, \vec{e}_1+\vec{e}_2-\vec{e}_3+2\vec{e}_4\big).\]
Стало быть, $X_{7}\big(\cf(l_1,l_2,l_3,l_4)\big) \in \mathbf{CF}_{7}'$, то есть $\cf(l_1,l_2,l_3,l_4)\in\overline{\mathbf{CF}_{7}'}$.
\end{proof}

\section*{Благодарности}

Автор благодарит О.Н. Германа за внимание к работе и полезные обсуждения результатов.

\end{document}